\documentclass[letterpaper]{article}

\usepackage{fullpage}
\usepackage{graphicx}
\usepackage{bbm}
\usepackage{dsfont}
\usepackage{amsthm}
\usepackage{amsmath}
\usepackage{amssymb}
\usepackage{url}
\usepackage{algorithm}
\usepackage{algorithmic} 
\usepackage{array}
\usepackage{color}
\usepackage[justification=centering,format=hang,singlelinecheck=false]{subfig}
\usepackage{multirow}
\usepackage[normalem]{ulem}
\usepackage{environ} 

\usepackage{hyperref}

\newcommand{\R}{\mathbb{R}}
\newcommand{\N}{\mathbb{N}}
\renewcommand{\d}[1]{\mathrm{D}#1}
\newcommand{\D}[1]{\mathrm{D}#1}

\newcommand{\card}[1]{{\lvert#1\rvert}}
\renewcommand{\epsilon}{\varepsilon}

\newtheorem{theorem}{Theorem}

\newtheorem{corollary}{Corollary}
\newtheorem{lemma}{Lemma}
\newtheorem{assumption}{Assumption}
\newtheorem{definition}{Definition}

\DeclareMathOperator*{\argmin}{arg\,min}
\DeclareMathOperator*{\argmax}{arg\,max}

\theoremstyle{definition}
\newtheorem*{SSOCP}{Switched System Optimal Control Problem}
\newtheorem*{RSSOCP}{Relaxed Switched System Optimal Control Problem}

\newtheorem*{DRSSOCP}{Discretized Relaxed Switched System Optimal Control Problem}

\NewEnviron{enumerate_parentesis}{%
  
  \begin{enumerate}
    \BODY
  \end{enumerate}
  
}

\NewEnviron{optimization_problem}[1]{%
  \begin{flalign}
    & {(#1)} & &
    \BODY
    & & 
  \end{flalign}
}

\title{Consistent Approximations for the Optimal Control of Constrained Switched Systems}
\author{Ramanaryan Vasudevan, Humberto Gonzalez, Ruzena Bajcsy, and S. Shankar Sastry
\thanks{R. Vasudevan, H. Gonzalez,  R. Bajcsy, and S. S. Sastry are with the Department of Electrical Engineering and Computer Sciences, University of California at Berkeley, Berkeley, CA, 94720, \{ramv,hgonzale,bajcsy,sastry,\}@eecs.berkeley.edu}%
}

\begin{document}

\maketitle

\begin{abstract}
  Though switched dynamical systems have shown great utility in modeling a variety of physical phenomena, the construction of an optimal control of such systems has proven difficult since it demands some type of optimal mode scheduling. 
  In this paper, we devise an algorithm for the computation of an optimal control of constrained nonlinear switched dynamical systems. 
  The control parameter for such systems include a continuous-valued input and discrete-valued input, where the latter corresponds to the mode of the switched system that is active at a particular instance in time. 
  Our approach, which we prove converges to local minimizers of the constrained optimal control problem, first relaxes the discrete-valued input, then performs traditional optimal control, and then projects the constructed relaxed discrete-valued input back to a pure discrete-valued input by employing an extension to the classical Chattering Lemma that we prove. 
  We extend this algorithm by formulating a computationally implementable algorithm which works by discretizing the time interval over which the switched dynamical system is defined. 
  Importantly, we prove that this implementable algorithm constructs a sequence of points by recursive application that converge to the local minimizers of the original constrained optimal control problem. 
  Four simulation experiments are included to validate the theoretical developments.
\end{abstract}

\section{Introduction}
\label{sec:introduction}

Hybrid dynamical models arise naturally in systems in which discrete modes of operation interact with continuous state evolution. Such systems have been used in a variety of modeling applications including automobiles and locomotives employing different gears \cite{rantzer_switch, rinehart2008suboptimal}, biological systems \cite{ghosh2001hybrid}, situations where a control module has to switch its attention among a number of subsystems \cite{lincoln2001optimizing,rehbinder2004scheduling,walsh2002stability}, manufacturing systems \cite{cassandras2001} and situations where a control module has to collect data sequentially from a number of sensory sources \cite{brockett1995stabilization,egerstedt2002multi}. In addition, many complex nonlinear dynamical systems can be decomposed into simpler linear modes of operation that are more amenable to analysis and controller design \cite{frazzoli2000robust,gillula2011applications}. 

Given their utility, there has been considerable interest in devising algorithms to perform optimal control of such systems. In fact, even Branicky et al.'s seminal work which presented many of the theoretical underpinnings of hybrid systems included a set of sufficient conditions for the optimal control of such systems using quasi variational inequalities \cite{Branicky1998}. Though compelling from a theoretical perspective, the application of this set of conditions to the construction of a numerical optimal control algorithm for hybrid dynamical systems requires the application of value iterations which is particularly difficult in the context of switched systems, wherein the switching between different discrete modes is specified by a discrete-valued input signal. The control parameter for such systems has both a discrete component corresponding to the schedule of discrete modes visited and two continuous components corresponding to the duration of time spent in each mode in the mode schedule and the continuous input. The determination of an optimal control for this class of hybrid systems is particularly challenging due to the combinatorial nature of calculating an optimal mode schedule. 

\subsection{Related Work}

The algorithms to solve this switched system optimal control problem can be divided into two distinct groups according to whether they do or do not rely on the Maximum Principle \cite{piccoli1998hybrid,pontryagin1962mathematical,Sussmann1999}. Given the difficulty of the problem, both groups of approaches sometimes employ similar tactics during algorithm construction. A popular such tactic is one formalized by Xu et al. who proposed a bi-level optimization scheme that at a low level optimized the continuous components of the problem while keeping the mode schedule fixed and at a high level modified the mode schedule \cite{xu2002}.

We begin by describing the algorithms for switched system optimal control that rely on the Maximum Principle. One of the first such algorithms, presented by Alamir et al., applied the Maximum Principle directly to a discrete time switched dynamical system \cite{Alamir2004}. In order to construct such an algorithm for a continuous time switched dynamical system, Shaikh et al. employed the bi-level optimization scheme proposed by Xu et al.  and applied the Maximum Principle to perform optimization at the lower level and applied the Hamming distance to compare different possible nearby mode schedules \cite{shaikh2003}. 

Given the algorithm that we construct in this paper, the most relevant of the approaches that rely on the Maximum Principle is the one proposed by Bengea et al. who relax the discrete-valued input and treat it as a continuous-valued input over which they can apply the Maximum Principle to perform optimal control \cite{Bengea2005}. A search through all possible discrete valued inputs is required in order to find one that approximates the trajectory of the switched system due to the application of the constructed relaxed discrete-valued input. Though such a search is expensive, the existence of a discrete-valued input that approximates the behavior of the constructed relaxed discrete-valued input is proven by the Chattering Lemma \cite{berkovitz1974optimal}. Unfortunately this combinatorial search is unavoidable by employing the Chattering Lemma since it provides no means to construct a discrete-valued input that approximates a relaxed discrete-valued input with respect the trajectory of the switched system. Summarizing, those algorithms that rely on the Maximum Principle construct powerful necessary conditions for optimality. Unfortunately their numerical implementation for nonlinear switched systems is fundamentally restricted due to their reliance on approximating strong or needle variations with arbitrary precision as explained in \cite{mayne1975}.

Next, we describe the algorithms that do not rely on the Maximum Principle but rather employ weak variations. Several have focused on the optimization of autonomous switched dynamical systems (i.e. systems without a continuous input) by fixing the mode sequence and working on devising first \cite{Egerstedt2006} and second order \cite{johnson2011second} numerical optimal control algorithms to optimize the amount of time spent in each mode. In order to extend these optimization techniques, Axelsson et al. employed the bi-level optimization strategy proposed by Xu et al., and after performing optimization at the lower-level by employing a first order numerical optimal control algorithm to optimize the amount of time spent in each mode while keeping the mode schedule fixed, they modified the mode sequence by employing a single mode insertion technique \cite{Axelsson2008}. 

There have been two major extensions to Axelsson et al.'s algorithm. First, Wardi et al., extend the approach by performing several single mode insertions at each iteration \cite{Wardi2012}. Second, Gonzalez et al., extend the approach to make it applicable to constrained switched dynamical systems with a continuous-valued input \cite{Gonzalez2010,Gonzalez2010a}. Though these single mode insertion techniques avoid the computational expense of considering all possible mode schedules during the high-level optimization, this improvement comes at the expense of restricting the possible modifications of the existing mode schedule, which may introduce undue local minimizers, and at the expense of requiring a separate optimization for each of the potential mode schedule modifications, which is time consuming.

\subsection{Our Contribution and Organization}

Inspired by the potential of the Chattering Lemma, in this paper, we devise and implement a first order numerical optimal control algorithm for the optimal control of constrained nonlinear switched systems. In Section \ref{sec:preliminaries}, we introduce the notation and assumptions used throughout the paper and formalize the the optimal control for constrained nonlinear switched systems. Our approach to solve this problem, which is formulated in Section \ref{sec:optimization_algorithm}, first relaxes the optimal control problem by treating the discrete-valued input to be continuous-valued. Next, a first order numerical optimal control algorithm is devised for this relaxed problem. After this optimization is complete, an extension of the Chattering Lemma that we construct, allows us to design a projection that takes the computed relaxed discrete-valued input back to a ``pure'' discrete-valued input while controlling the quality of approximation of the trajectory of the switched dynamical system generated by applying the projected discrete-valued input rather than the relaxed discrete-valued input. In Section \ref{sec:algo_analysis}, we prove that the sequence of points generated by recursive application of our first order numerical optimal control algorithm converge to a point that satisfies a necessary condition for optimality of the constrained nonlinear switched system optimal control problem.

We then describe in Section \ref{sec:implementation} how our algorithm can be formulated in order to make numerical implementation feasible. In fact, in Section \ref{sec:imp_algo_analysis}, we prove that the this computationally implementable algorithm is a \emph{consistent approximation} of our original algorithm. This ensures that the sequence of points generated by the recursive application of this numerically implementable algorithm converge to a point that satisfies a necessary condition for optimality of the constrained nonlinear switched system optimal control problem. In Section \ref{sec:examples}, we implement this algorithm and compare its performance to a commercial mixed integer optimization algorithm on $4$ separate problems to illustrate its superior performance with respect to speed and quality of constructed minimizer.

\section{Preliminaries}
\label{sec:preliminaries}

In this section, we formalize the problem we solve in this paper. Before describing this problem, we define the function spaces and norms used throughout this paper.

\subsection{Norms and Functional Spaces}

This paper focuses on the optimization of functions with finite $L^2$-norm and finite bounded variation. To formalize this notion, we require a norm. For each $x \in \R^n$, $p \in \N$, and $p > 0$, we let $\left\|x \right\|_p$ denote the $p$--norm of $x$. For each $A \in \R^{n \times m}$, $p \in \N$, and $p > 0$, we let $\left\| A \right\|_{i,p}$ denote the induced $p$--norm of $A$. 

Given these definitions, we say a function, $f: [0,1] \to {\cal Y}$, where ${\cal Y} \subset \R^n$, belongs to $L^2([0,1],{\cal Y})$ with respect to the Lebesgue measure on $[0,1]$ if:
\begin{equation}
  \label{eq:L2_norm}
  \left\lVert f \right\rVert_{L^2} = \left(\int_0^1 \left\lVert f(t) \right\rVert_2^2 dt\right)^{\frac{1}{2}} < \infty.
\end{equation}

We say a function, $f: [0,1] \to {\cal Y}$, where ${\cal Y} \subset \R^n$, belongs to $L^{\infty}([0,1],{\cal Y})$ with respect to the Lebesgue measure on $[0,1]$ if:
\begin{equation}
	\label{eq:Linf_norm}
	\left\lVert f \right\rVert_{L^{\infty}} 
  = \inf \bigl\{ \alpha \geq 0 \mid \left\lVert f(x) \right\rVert_2 \leq \alpha\ \textrm{for almost every}\ x \in [0,1] \bigr\} < \infty.
\end{equation}

In order to define the space of functions of finite bounded variation, we first define the total variation of a function. Given $P$, the set of all finite partitions of $[0,1]$, we define the \emph{total variation} of $f: [0,1] \to {\cal Y}$ by:
\begin{equation}
  \label{eq:BV_norm}
  \left\lVert f \right\rVert_{BV}
  = \sup \left\{ \sum_{j=0}^{m-1} \left\lVert f(t_{j+1}) - f(t_{j}) \right\rVert_1 \mid \{ t_k \}_{k=0}^m \in P \right\}.
\end{equation}
Note that the total variation of $f$ is not a norm but rather a seminorm, i.e. it does not separate points. Regardless, we use the norm symbol for the total variation throughout this paper. We say that $f$ is of \emph{bounded variation} if $\| f \|_{BV} < \infty$, and we define $BV([0,1],{\cal Y})$ to be the set of all functions of bounded variation from $[0,1]$ to ${\cal Y}$.

There is an important connection between the functions of bounded variation and weak derivatives, which we rely on throughout this paper. Given $f: [0,1] \to {\cal Y}$, we say that $f$ has a \emph{weak derivative} if there exists a Radon signed measure $\mu$ over $[0,1]$ such that, for each smooth bounded function $v$ with $v(0) = v(1) = 0$,
\begin{equation}
  \int_0^1 f(t) \dot{v}(t) dt = - \int_0^1 v(t) d\mu(t).
\end{equation}
Moreover, we say that $\dot{f} = \frac{d\mu(t)}{dt}$, where the derivative is taken in the Radon--Nikodym sense, is the weak derivative of $f$.
Note that $\dot{f}$ is in general a distribution.
Perhaps the most common example of weak derivative is the Dirac Delta, which is the weak derivative of the Step Function.
The following result is fundamental in our analysis of functions of bounded variation:
\begin{theorem}[Exercise 5.1 in \cite{Ziemer1989}]
  \label{thm:bounded_variation}
  If $f \in BV([0,1],{\cal Y})$, then $f$ has a weak derivative, denoted $\dot{f}$.
  Moreover,
  \begin{equation}
    \| f \|_{BV} = \int_0^1 \bigl\lVert \dot{f}(t) \bigr\rVert_1 dt.
  \end{equation}
\end{theorem}
We omit the proof of this result since it is beyond the scope of this paper. More details about the functions of bounded variation and weak derivatives can be found in Sections 3.5 and 9 in \cite{Folland1999} and Section 5 in \cite{Ziemer1989}.

\subsection{Optimization Spaces}

We are interested in the control of systems whose trajectory is governed by a set of vector fields $f: \R \times \R^n \times \R^m \times {\cal Q} \rightarrow \R^n$, indexed by their last argument where ${\cal Q}=\{1, 2, \dots, q\}$. Each of these distinct vector fields is called a \emph{mode} of the switched system. To formalize the optimal control problem, we define three spaces: the \emph{pure discrete input space}, ${\cal D}_p$, the \emph{relaxed discrete input space}, ${\cal D}_r$, and the \emph{continuous input space}, ${\cal U}$. Throughout the document, we employ the following convention: given the pure or relaxed discrete input $d$, we denote its $i$--th coordinate by $d_i$.

Before formally defining each space, we require some notation. Let the $q$--simplex, $\Sigma^q_r$, be defined as:
\begin{equation}
  \label{eq:simplex}
  \Sigma^q_r = \left\{ (d_1,\ldots,d_q) \in [0,1]^q \mid \sum_{i=1}^q d_i = 1 \right\},
\end{equation}
and let the corners of the $q$-simplex, $\Sigma^q_p$, be defined as:
\begin{equation}
  \label{eq:corners_simplex}
  \Sigma^q_p = \left\{ (d_1,\ldots,d_q) \in \{0,1\}^q \mid \sum_{i=1}^q d_i = 1 \right\}.
\end{equation}
Note that $\Sigma_p^q \subset \Sigma_r^q$. 
Also, there are exactly as many corners, denoted $e_i$ for $i \in {\cal Q}$, of the $q$--simplex as there are distinct vector fields. 
Thus, $\Sigma_p^q = \{ e_1, \ldots, e_q \}$.

Using this notation, we define the pure discrete input space, ${\cal D}_p$, as:
\begin{equation}
  \label{eq:pure_discrete_input_space}
  {\cal D}_p = L^2( [0,1], \Sigma^q_p ) \cap BV([0,1],\Sigma^q_p).
\end{equation}
Next, we define the relaxed discrete input space, ${\cal D}_r$:
\begin{equation}
  \label{eq:relaxed_discrete_input_space}
  {\cal D}_r = L^2( [0,1], \Sigma^q_r ) \cap BV( [0,1], \Sigma^q_r ).
\end{equation}
Notice that the discrete input at each instance in time can be written as the linear combination of the corners of the simplex. Given this observation, we employ these corners to index the vector fields (i.e. for each $i \in {\cal Q}$ we write $f(\cdot,\cdot,\cdot,e_i)$ for $f(\cdot,\cdot,\cdot,i)$). Finally, we define the continuous input space, ${\cal U}$:
\begin{equation}
  \label{eq:input_space}
  {\cal U} = L^2( [0,1], U ) \cap BV( [0,1], U ),
\end{equation}
where $U \subset \R^m$ is a bounded, convex set.

Let ${\cal X} = L^\infty( [0,1], \R^m ) \times L^\infty( [0,1], \R^q )$ be endowed with the following norm for each $\xi = (u,d) \in {\cal X}$:
\begin{equation}
  \label{eq:defn_metric}
  \| \xi \|_{\cal X} = \| u \|_{L^2}  + \| d \|_{L^2},
\end{equation}
where the $L^2$--norm is as defined in Equation \eqref{eq:L2_norm}.
We combine ${\cal U}$ and ${\cal D}_p$ to define our \emph{pure optimization space}, ${\cal X}_p = {\cal U} \times {\cal D}_p$, and we endow it with the same norm as ${\cal X}$. 
Similarly, we combine ${\cal U}$ and ${\cal D}_r$ to define our \emph{relaxed optimization space}, ${\cal X}_r = {\cal U} \times {\cal D}_r$, and endow it with the ${\cal X}$--norm too. Note that ${\cal X}_p \subset {\cal X}_r \subset {\cal X}$.

\subsection{Trajectories, Cost, Constraint, and the Optimal Control Problem}
\label{sec:trajs_cost_cons_ocp}

Given $\xi = (u,d) \in {\cal X}_r$, for convenience throughout the paper we let:
\begin{equation}
	\label{eq:compact_traj_xi}
	f\bigl( t, x(t), u(t), d(t) \bigr)	= \sum_{i=1}^{q} d_i(t) f\bigl( t, x(t), u(t), e_i \bigr),
\end{equation}
where $d(t) = \sum_{i=1}^q d_i(t) e_i$. We employ the same convention when we consider the partial derivatives of $f$. Given $x_0 \in \R^n$, we say that a \emph{trajectory of the system} corresponding to $\xi \in {\cal X}_r$ is the solution to:
\begin{equation}
  \label{eq:traj_xi}
  \dot{x}(t) = f\big(t,x(t),u(t),d(t)), \quad \forall t \in [0, 1], \quad x(0) = x_0,
\end{equation}
and denote it by $x^{(\xi)}: [0,1] \to \R^n$, where we suppress the dependence on $x_0$ in $x^{(\xi)}$ since it is assumed given. To ensure the clarity of the ensuing analysis, it is useful to sometimes emphasize the dependence of $x^{(\xi)}(t)$ on $\xi$. Therefore, we define the \emph{flow of the system}, $\phi_t: {\cal X}_r \to \R^n$ for each $t \in [0,1]$ as:
	\begin{equation}
		\label{eq:flow_xi}
		\phi_t(\xi) = x^{(\xi)}(t).
	\end{equation}
	
To define the cost function, we assume that we are given a \emph{terminal cost}, $h_0: \R^n \to \R$. The \emph{cost function}, $J: {\cal X}_r \to \R$, for the optimal control problem is then defined as:
\begin{equation}
  \label{eq:cost}
  J(\xi) = h_0\big( \, x^{(\xi)} (1) \, \big).
\end{equation}
Notice that if the problem formulation includes a running cost, then one can extend the existing state vector by introducing a new state, and modifying the cost function to evaluate this new state at the final time, as shown in Section 4.1.2 in \cite{Polak1997}. By performing this type of modification, observe that each mode of the switched system can have a different running cost associated with it.

Next, we define a family of functions, $h_j:\R^n \to \R$ for $j \in {\cal J} = \{1,\ldots,N_c\}$. 
Given a $\xi \in {\cal X}_r$, the state $x^{(\xi)}$ is said to satisfy the constraint if $h_j(x^{(\xi)}(t)) \leq 0$ for each $t \in [0,1]$ and for each $j \in {\cal J}$.  
We compactly describe all the constraints by defining the \emph{constraint function} $\Psi: {\cal X}_r \to \R$, by: 
\begin{equation}
  \label{eq:super_const}
  \Psi(\xi) = \max_{j \in {\cal J},\; t \in [ 0, 1 ]} h_j\big( x^{(\xi)}(t) \big),
\end{equation}
since $h_j\big( x^{(\xi)}(t) \big) \leq 0$ for each $t$ and $j$ if and only if $\psi(\xi) \leq 0$. To ensure the clarity of the ensuing analysis, it is useful to sometimes emphasize the dependence of $h_j\big(x^{(\xi)}(t)\big)$ on $\xi$. Therefore, we define \emph{component constraint functions}, $\psi_{j,t}: {\cal X}_r \to \R$ for each $t \in [0,1]$ and $j \in {\cal J}$ as:
	\begin{equation}
		\label{eq:component_constraints}
		\psi_{j,t}(\xi) = h_j\left(\phi_t(\xi)\right).
	\end{equation}


With these definitions, we can state the Switched System Optimal Control Problem:
\begin{SSOCP}
  \begin{equation}
	\label{eq:ssocp}
    \min_{\xi \in {\cal X}_p} \left\{ J(\xi) \mid \Psi(\xi) \leq 0 \right\}.
  \end{equation}
\end{SSOCP}

\subsection{Assumptions and Uniqueness}
In order to devise an algorithm to solve Switched System Optimal Control Problem, we make the following assumptions about the dynamics, cost, and constraints: 
\begin{assumption}
  \label{assump:fns_continuity}
  For each $i \in {\cal Q}$, $f(\cdot,\cdot,\cdot,e_i)$ is differentiable in both $x$ and $u$. 
  Also, each $f(\cdot,\cdot,\cdot,e_i)$ and its partial derivatives are Lipschitz continuous with constant $L > 0$, i.e. given $t_1,t_2 \in [0,1]$, $x_1,x_2 \in \R^n$, and $u_1,u_2 \in U$:
	\begin{enumerate_parentesis}
  \item \label{assump:f_lipschitz} $\left\| f(t_1,x_1,u_1,e_i) - f(t_2,x_2,u_2,e_i) \right\|_2 \leq L \left( |t_1 - t_2 |+ \| x_1 - x_2 \|_2 + \| u_1 - u_2 \|_2 \right)$,
  \item \label{assump:dfdx_lipschitz} $\left\| \frac{\partial f}{\partial x}(t_1,x_1,u_1,e_i) - \frac{\partial f}{\partial x}(t_2,x_2,u_2,e_i) \right\|_{i,2} \leq L \left( | t_1 - t_2 | + \| x_1 - x_2 \|_2 + \| u_1 - u_2 \|_2 \right)$,
  \item \label{assump:dfdu_lipschitz} $\left\| \frac{\partial f}{\partial u}(t_1,x_1,u_1,e_i) - \frac{\partial f}{\partial u}(t_2,x_2,u_2,e_i) \right\|_{i,2} \leq L \left( | t_1 - t_2 | + \| x_1 - x_2 \|_2 + \| u_1 - u_2 \|_2 \right)$.
	\end{enumerate_parentesis}
\end{assumption}

\begin{assumption}
  \label{assump:constraint_fns}
  The functions $h_0$ and $h_j$ are Lipschitz continuous and differentiable in $x$ for all $j \in {\cal J}$. In addition, the derivatives of these functions with respect to $x$ are also Lipschitz continuous with constant $L > 0$, i.e. given $x_1, x_2 \in \R^n$, for each $j \in {\cal J}$:
	\begin{enumerate_parentesis}
  \item \label{assump:phi_lipschitz} $\left| h_0(x_1) - h_0(x_2) \right| \leq L \left\| x_1 - x_2 \right\|_2$,
  \item \label{assump:dphidx_lipschitz} $\left\| \frac{\partial h_0}{\partial x}(x_1) - \frac{\partial h_0}{\partial x}(x_2) \right\|_2 \leq L \left\| x_1 - x_2 \right\|_2$,
  \item \label{assump:hj_lipschitz} $\left| h_j(x_1) - h_j(x_2) \right| \leq L \left\| x_1 - x_2 \right\|_2$,
  \item \label{assump:dhjdx_lipschitz} $\left\| \frac{\partial h_j}{\partial x}(x_1) - \frac{\partial h_j}{\partial x}(x_2) \right\|_2 \leq L \left\| x_1 - x_2 \right\|_2$.
	\end{enumerate_parentesis}
\end{assumption}

If a running cost is included in the problem statement (i.e. if the cost also depends on the integral of a function), then this function must also satisfy Assumption \ref{assump:fns_continuity}. Assumption \ref{assump:constraint_fns} is a standard assumption on the objectives and constraints and is used to prove the convergence properties of the algorithm defined in the next section. These assumptions lead to the following result:
\begin{lemma}
	\label{lemma:x_bounded}
	There exists a constant $C > 0$ such that, for each $\xi \in {\cal X}_r$ and $t \in [0,1]$,
	\begin{equation}
		\bigl\| x^{(\xi)}(t) \bigr\|_2 \leq C,
	\end{equation}	
	where $x^{(\xi)}$ is a solution of Differential Equation \eqref{eq:traj_xi}.
\end{lemma}
\begin{proof}
	Given $\xi = (u,d) \in {\cal X}_r$ and noticing that $\left| d_i(t) \right| \leq 1$ for all $i \in {\cal Q}$ and $t \in [0,1]$, we have:
	\begin{equation}
		\label{eq:x_bounded_DEQ}
		\bigl\| x^{(\xi)}(t) \bigr\|_2 
    \leq \| x_0 \|_2 + \sum_{i=1}^q \int_0^t \bigl\| f\bigl( s, x^{(\xi)}(s), u(s), e_i \bigr) \bigr\|_2 ds.
	\end{equation}
  
	Next, observe that $\| f( 0, x_0, 0, e_i ) \|_2$ is bounded for all $i \in {\cal Q}$ and $u(s)$ is bounded for each $s \in [0,1]$ since $U$ is bounded. 
  Then by Assumption \ref{assump:fns_continuity}, we know there exists a $K > 0$ such that for each $s \in [0,1]$, $i \in {\cal Q}$, and $\xi \in {\cal X}_r$,
  \begin{equation}
    \bigl\| f\bigl( s, x^{(\xi)}(s), u(s), e_i \bigr) \bigr\|_2 \leq K \bigl( \bigl\| x^{(\xi)}(s) \bigr\|_2 + 1 \bigr).    
  \end{equation}
  Applying the Bellman-Gronwall Inequality (Lemma 5.6.4 in \cite{Polak1997}) to Equation \eqref{eq:x_bounded_DEQ}, we have $\bigl\| x^{(\xi)}(t) \bigr\|_2 \leq e^{qK} \bigl( 1 + \| x_0 \|_2 \bigr)$ for each $t \in [0,1]$. 
  Since $x_0$ is assumed given and bounded, we have our result.
\end{proof}

In fact, this implies that the dynamics, cost, constraints, and their derivatives are all bounded:
\begin{corollary}
	\label{corollary:fns_bounded}
	There exists a constant $C > 0$ such that for each $\xi = (u,d) \in {\cal X}_r$, $t \in [0,1]$, and $j \in {\cal J}$:
	\begin{enumerate_parentesis}
		\item \label{corollary:f_bounded} $\bigl\| f\bigl( t, x^{(\xi)}(t), u(t), d(t) \bigr) \bigr\|_2 \leq C$,\;
      $\left\| \frac{\partial f}{\partial x}\bigl( t, x^{(\xi)}(t), u(t), d(t) \bigr) \right\|_{i,2} \leq C$,\; 
      $\left\| \frac{\partial f}{\partial u}\bigl( t, x^{(\xi)}(t), u(t), d(t) \bigr) \right\|_{i,2} \leq C$,
		\item \label{corollary:phi_bounded} $\bigl| h_0\bigl( x^{(\xi)}(t) \bigr) \bigr| \leq C$,\; 
      $\Bigl\| \frac{\partial h_0}{\partial x}\bigl( x^{(\xi)}(t) \bigr) \Bigr\|_2 \leq C$,
		\item \label{corollary:hj_bounded} $\bigl| h_j\bigl( x^{(\xi)}(t) \bigr) \bigr| \leq C$,\; 
      $\left\| \frac{\partial h_j}{\partial x}\bigl( x^{(\xi)}(t) \bigr) \right\|_2 \leq C$,
	\end{enumerate_parentesis}
	where $x^{(\xi)}$ is a solution of Differential Equation \eqref{eq:traj_xi}.
\end{corollary}
\begin{proof}
	The result follows immediately from the continuity of $f$, $\frac{\partial f}{\partial x}$, $\frac{\partial f}{\partial u}$, $h_0$, $\frac{\partial h_0}{\partial x}$, $h_j$, and $\frac{\partial h_j}{\partial x}$ for each $j \in {\cal J}$, as stated in Assumptions \ref{assump:fns_continuity} and \ref{assump:constraint_fns}, and the fact that each of the arguments to these functions can be constrained to a compact domain, which follows from Lemma \ref{lemma:x_bounded} and the compactness of $U$ and $\Sigma_r^q$.
\end{proof}

An application of this corollary leads to a fundamental result:
\begin{theorem}
  \label{thm:existence_and_uniqueness}
	For each $\xi \in {\cal X}_r$ Differential Equation \eqref{eq:traj_xi} has a unique solution.
\end{theorem}
\begin{proof}
  First let us note that $f$, as defined in Equation \eqref{eq:compact_traj_xi}, is also Lipschitz with respect to its fourth argument. Indeed, given $t \in [0,1]$, $x \in \R^n$, $u \in U$, and $d_1,d_2 \in \Sigma_r^q$,
  \begin{equation}
    \begin{aligned}
      \bigl\| f( t, x, u, d_1 ) - f( t, x, u, d_2 ) \bigr\|_2
      &= \left\| \sum_{i=1}^q \bigl( d_{1,i} - d_{2,i} \bigr) f( t, x, u, e_i ) \right\|_2 \\
      &\leq C q \| d_1 - d_2 \|_2,
    \end{aligned}
  \end{equation}
  where $C > 0$ is as in Corollary \ref{corollary:fns_bounded}.
  
  Given that $f$ is Lipschitz with respect to all its arguments, the result follows as a direct extension of the classical existence and uniqueness theorem for nonlinear differential equations (see Section 2.4.1 in \cite{Vidyasagar2002} for a standard version of this theorem).
\end{proof}

Therefore, since $x^{(\xi)}$ is unique, it is not an abuse of notation to denote the solution of Differential Equation \eqref{eq:traj_xi} by $x^{(\xi)}$. Next, we develop an algorithm to solve the Switched System Optimal Control Problem.

\section{Optimization Algorithm}
\label{sec:optimization_algorithm}

In this section, we describe our optimization algorithm. 
Our approach proceeds as follows: first, we treat a given pure discrete input as a relaxed discrete input by allowing it to belong ${\cal D}_r$; second, we perform optimal control over the relaxed optimization space; and finally, we project the computed relaxed input into a pure input. 
Before describing our algorithm in detail, we begin with a brief digression to motivate why such a roundabout construction is required in order to devise a first order numerical optimal control scheme for the Switched System Optimal Control Problem defined in Equation \eqref{eq:ssocp}.

\subsection{Directional Derivatives}
\label{subsec:directional_derivatives}

To appreciate why the construction of a numerical scheme to find the local minima of the Switched System Optimal Control Problem defined in Equation \eqref{eq:ssocp} is difficult, suppose that the optimization in the problem took place over the relaxed optimization space rather than the pure optimization space. The Relaxed Switched System Optimal Control Problem is then defined as:
\begin{RSSOCP}
  \begin{equation}
	\label{eq:rssocp}
    \min_{\xi \in {\cal X}_r} \left\{ J(\xi) \mid \Psi(\xi) \leq 0 \right\}.
  \end{equation}
\end{RSSOCP}

The local minimizers of this problem are then defined as follows:
\begin{definition}
  \label{def:local_minimizer_Xr}
  Let us denote an $\epsilon$--ball in the ${\cal X}$--norm centered at $\xi$ by:
  \begin{equation}
    \label{eq:nbhd_X}
    {\cal N}_{\cal X}(\xi,\epsilon) = \left\{ \bar{\xi} \in {\cal X}_r \mid \bigl\lVert \xi - \bar{\xi} \bigr\rVert_{\cal X} < \epsilon \right\}.
  \end{equation}
	We say that a point $\xi \in {\cal X}_r$ is a \emph{local minimizer of the Relaxed Switched System Optimal Control Problem} defined in Equation \eqref{eq:rssocp} if $\Psi(\xi) \leq 0$ and there exists $\epsilon > 0$ such that $J( \hat{\xi} ) \geq J( \xi )$ for each $\hat{\xi} \in {\cal N}_{\cal X}(\xi,\epsilon) \cap \left\{ \bar{\xi} \in {\cal X}_r \mid \Psi( \bar{\xi} ) \leq 0 \right\}$.
\end{definition}
Given this definition, a first order numerical optimal control scheme can exploit the vector space structure of the relaxed optimization space in order to define directional derivatives that find local minimizers for this Relaxed Switched System Optimal Control Problem.

To concretize how such an algorithm would work, we introduce some additional notation. Given $\xi \in {\cal X}_r$, ${\cal Y}$ a Euclidean space, and any function $G: {\cal X}_r \to {\cal Y}$, the directional derivative of $G$ at $\xi$, denoted $\D{G}(\xi;\cdot): {\cal X} \to {\cal Y}$, is computed as:
\begin{equation}
  \label{eq:operator_D_definition}
  \D{G}(\xi;\xi') = \lim_{\lambda \downarrow 0} \frac{G( \xi + \lambda \xi' ) - G(\xi)}{\lambda}.
\end{equation}

To understand the connection between directional derivatives and local minimizers, suppose the Relaxed Switched System Optimal Control Problem is unconstrained and consider the first order approximation of the cost $J$ at a point $\xi \in {\cal X}_r$ in the $\xi' \in {\cal X}$ direction by employing the directional derivative $\D{J}(\xi;\xi')$:
\begin{equation}
  \label{eq:J_first_order_expansion}
  J( \xi + \lambda \xi' ) \approx J(\xi) + \lambda \d{J}(\xi;\xi'),
\end{equation}
where $0 \leq \lambda \ll 1$. 
It follows that if $\D{J}(\xi;\xi')$, whose existence is proven in Lemma \ref{lemma:DJ_definition}, is negative, then it is possible to decrease the cost by moving in the $\xi'$ direction. 
That is if the directional derivative of the cost at a point $\xi$ is negative along a certain direction, then for each $\epsilon > 0$ there exists a $\hat{\xi} \in {\cal N}_{\cal X}(\xi,\epsilon)$ such that $J(\hat{\xi}) <  J(\xi)$. 
Therefore if $\D{J}(\xi;\xi')$ is negative, then $\xi$ is not a local minimizer of the unconstrained Relaxed Switched System Optimal Control Problem.

Similarly, for the general Relaxed Switched System Optimal Control Problem, consider the first order approximation of each of the component constraint functions, $\psi_{j,t}$ for each $j \in {\cal J}$ and $t \in [0,1]$ at a point $\xi \in {\cal X}_r$ in the $\xi \in {\cal X}$ direction by employing the directional derivative $\D{\psi_{j,t}}(\xi;\xi')$:
\begin{equation}
  \label{eq:hj_first_order_expansion}
  \psi_{j,t}( \xi + \lambda \xi' ) \approx \psi_{j,t}(\xi) + \lambda \d{\psi_{j,t}}(\xi;\xi'),
\end{equation}
where $0 \leq \lambda \ll 1$. It follows that if $\D{\psi_{j,t}}(\xi;\xi')$, whose existence is proven in Lemma \ref{lemma:Dhj_definition}, is negative, then it is possible to decrease the infeasibility of $\phi_t(\xi)$ with respect to $h_j$ by moving in the $\xi'$ direction. That is if the directional derivatives of the cost and all of the component constraints for all $t \in [0,1]$ at a point $\xi$ are negative along a certain direction and $\Psi(\xi) = 0$, then for each $\epsilon > 0$ there exists a $\hat{\xi} \in \{\bar{\xi} \in {\cal X}_r \mid \Psi(\bar{\xi}) \leq 0\} \cap {\cal N}_{\cal X}(\xi,\epsilon)$ such that $J(\hat{\xi}) <  J(\xi)$. Therefore, if $\Psi(\xi) = 0$ and $\D{J}(\xi;\xi')$ and $\D{\psi_{j,t}}(\xi;\xi')$ are negative for all $j \in {\cal J}$ and $t \in [0,1]$, then $\xi$ is not a local minimizer of the Relaxed Hybrid Optimal Control Problem. Similarly, if $\Psi(\xi) < 0$ and $\D{J}(\xi;\xi')$ is negative, then $\xi$ is not a local minimizer of the Relaxed Hybrid Optimal Control Problem, even if $\D{\psi_{j,t}}(\xi;\xi')$ is greater than zero for all $j \in {\cal J}$ and $t \in [0,1]$. 


Returning to the Switched System Optimal Control Problem, it is unclear how to define a directional derivative for the pure discrete input space since it is not a vector space. Therefore, in contrast to the relaxed discrete and continuous input spaces, the construction of a first order numerical scheme for the optimization of the pure discrete input is non-trivial. One could imagine trying to exploit the directional derivatives in the relaxed optimization space in order to construct a first order numerical optimal control algorithm for the Switched System Optimal Control Problem, but this would require devising some type of connection between points belonging to the pure and relaxed optimization spaces. 

%

\subsection{The Weak Topology on the Optimization Space and Local Minimizers}

To motivate the type of relationship required between the pure and relaxed optimization space in order to construct a first order numerical optimal control scheme, we begin by describing the Chattering Lemma:
\begin{theorem}[Theorem 1 in \cite{Bengea2005}]
	\label{thm:bangbang_result}
	For each $\xi_r \in {\cal X}_r$ and $\epsilon > 0$ there exists a $\xi_p \in {\cal X}_p$ such that for each $t \in [0,1]$:
	\begin{equation}
		\left\| \phi_t(\xi_r) - \phi_t(\xi_p) \right\|_2 \leq \epsilon,
	\end{equation}
	where $\phi_t(\xi_r)$ and $\phi_t(\xi_p)$ are solutions to Differential Equation \eqref{eq:traj_xi} corresponding to $\xi_r$ and $\xi_p$, respectively.
\end{theorem}
\noindent The theorem as is proven in \cite{berkovitz1974optimal} is not immediately applicable to switched systems, but a straightforward extension as is proven in Theorem 1 in \cite{Bengea2005} makes that feasible. Note that the theorem as stated in \cite{Bengea2005}, considers only two vector fields (i.e. $q = 2$), but as the author's of the theorem remark, their proof can be generalized to an arbitrary number of vector fields. A particular version of this existence theorem can also be found in Lemma 1 \cite{Sussmann1972}.

Theorem \ref{thm:bangbang_result} says that the behavior of any element of the relaxed optimization space with respect to the trajectory of switched system can be approximated arbitrarily well by a point in the pure optimization space. Unfortunately, the relaxed and pure point as in Theorem \ref{thm:bangbang_result} need not be near one another in the metric induced by the ${\cal X}$-norm. Therefore, though there exists a relationship between the pure and relaxed optimization spaces, this connection is not reflected in the topology induced by the ${\cal X}$-norm; however, in a particular topology over the relaxed optimization space, a relaxed point and the pure point that approximates it as in Theorem \ref{thm:bangbang_result} can be made arbitrarily close:
\begin{definition}
	\label{definition:weak_topology}
  We say that the \emph{weak topology on ${\cal X}_r$ induced by Differential Equation \eqref{eq:traj_xi}} is the smallest topology on ${\cal X}_r$ such that the map $\xi \mapsto x^{(\xi)}$ is continuous.
  Moreover, an $\epsilon$--ball in the weak topology centered at $\xi$ is denoted by:
  \begin{equation}
		\label{eq:nbhd_weak_topology}
    {\cal N}_w(\xi,\epsilon) = \left\{ \bar{\xi} \in {\cal X}_r \mid \bigl\| x^{(\xi)} - x^{(\bar{\xi})} \bigr\|_{L^2} < \epsilon \right\}.
  \end{equation}
\end{definition}

A longer introduction to weak topology can be found in Section 3.8 in \cite{Rudin1991} or Section 2.3 in \cite{kurdila2005convex}, but before continuing we make an important observation that aids in motivating the ensuing analysis. In order to understand the relationship between the topology generated by the ${\cal X}$-norm on ${\cal X}_r$ and the weak topology on ${\cal X}_r$, observe that $\phi_t$ is Lipschitz continuous for all $t \in [0,1]$ (this is proven in Corollary \ref{corollary:x_lipschitz}). Therefore, for any $\epsilon > 0$ there exists a $\delta > 0$ such that if a pair of points of the relaxed optimization space belong to the same $\delta$--ball in the ${\cal X}$--norm, then the pair of points belong to the same $\epsilon$--ball in the weak topology on ${\cal X}_r$. 

Notice, however, that it is not possible to show that for every $\epsilon > 0$ that there exists a $\delta > 0$ such that if a pair of points of the relaxed optimization space belong to the same $\delta$--ball in the weak topology on ${\cal X}_r$, then the pair of points belong to the same $\epsilon$--ball in the ${\cal X}$--norm. More informally, a pair of points may generate trajectories that are near one another in the $L^2$--norm while not being near one another in the ${\cal X}$--norm. 
Since the weak topology, in contrast to the ${\cal X}$--norm induced topology, naturally places points that generate nearby trajectories next to one another, we extend Definition \ref{definition:weak_topology} in order to define a weak topology on ${\cal X}_p$ which we then use to define a notion of local minimizer for the Switched System Optimal Control Problem:
\begin{definition}
	\label{definition:SSOCP_minimizers}
  We say that a point $\xi \in {\cal X}_p$ is a \emph{local minimizers of the Switched System Optimal Control Problem} defined in Equation \eqref{eq:ssocp} if $\Psi(\xi) \leq 0$ and there exists $\epsilon > 0$ such that $J( \hat{\xi} ) \geq J(\xi)$ for each $\hat{\xi} \in {\cal N}_w(\xi,\epsilon) \cap \left\{ \bar{\xi} \in {\cal X}_p \mid \Psi( \bar{\xi} ) \leq 0 \right\}$, where ${\cal N}_w$ is as defined in Equation \eqref{eq:nbhd_weak_topology}.
\end{definition}

With this definition of local minimizer, we can exploit Theorem \ref{thm:bangbang_result}, even just as an existence result, along with the notion of directional derivative over the relaxed optimization space to construct a necessary condition for optimality for the Switched System Optimal Control Problem.


\subsection{An Optimality Condition}

Motivated by the approach undertaken in \cite{Polak1997}, we define an \emph{optimality function}, $\theta: {\cal X}_p \to (-\infty,0]$ that determines whether a given point is a local minimizer of the Switched System Optimal Control Problem and a corresponding \emph{descent direction}, $g: {\cal X}_p \to {\cal X}_r$: 
\begin{equation}
  \label{eq:theta_definition}
  \theta(\xi) = \min_{\xi' \in {\cal X}_r} \zeta( \xi, \xi' ), \qquad g(\xi) = \argmin_{\xi' \in {\cal X}_r} \zeta( \xi, \xi' ),
\end{equation}
where 
\begin{equation}
  \label{eq:zeta_definition}
  \zeta(\xi,\xi') =
  \begin{cases}
    \max \left\{ 
      \d{J}(\xi; \xi' - \xi), \max_{ j \in {\cal J},\; t\in[0,1] }\limits \D{\psi_{j,t}}(\xi;\xi'-\xi) + \gamma \Psi(\xi) 
    \right\} + \| \xi' - \xi \|_{\cal X}
    & \text{if}\ \Psi(\xi) \leq 0, \\
    \max \left\{ 
      \d{J}(\xi; \xi' - \xi) - \Psi(\xi), \max_{ j \in {\cal J},\; t\in[0,1] }\limits \D{\psi_{j,t}}(\xi;\xi'-\xi) 
    \right\} + \| \xi' - \xi \|_{\cal X}
    & \text{if}\ \Psi(\xi) > 0, \\
  \end{cases}
\end{equation}
where $\gamma > 0$ is a design parameters. For notational convenience in the previous equation we have left out the natural inclusion of $\xi$ from ${\cal X}_p$ to ${\cal X}_r$. Before proceeding, we make two observations. First, note that $\theta(\xi) \leq 0$ for each $\xi \in {\cal X}_p$, since we can always choose $\xi' = \xi$ which leaves the trajectory unmodified. Second, note that at a point $\xi \in {\cal X}_p$ the directional derivatives in the optimality function consider directions $\xi' - \xi$ with $\xi' \in {\cal X}_r$ in order to ensure that first order approximations constructed as in Equations \eqref{eq:J_first_order_expansion} and \eqref{eq:hj_first_order_expansion} belong to the relaxed optimization space ${\cal X}_r$ which is convex (e.g. for $0 < \lambda \ll 1$, $J(\xi) + \lambda \d{J}(\xi;\xi' - \xi) \approx J( (1 - \lambda) \xi + \lambda \xi' )$  where $(1 - \lambda) \xi + \lambda \xi' \in  {\cal X}_r$). 

To understand how the optimality function behaves, consider several cases. First, if $\theta(\xi) < 0$ and $\Psi(\xi) = 0$, then there exists a $\xi' \in {\cal X}_r$ such that both $\D{J}(\xi;\xi'-\xi)$ and $\D{\psi_{j,t}}(\xi;\xi'-\xi)$ are negative for all $j \in {\cal J}$ and $t \in [0,1]$. By employing the aforementioned first order approximation, we can show that for each $\epsilon > 0$ there exists an $\epsilon$--ball in the ${\cal X}$-norm centered at $\xi$ such that $J(\hat{\xi}) < J(\xi)$ for some $\hat{\xi} \in \{\bar{\xi} \in {\cal X}_r \mid \Psi(\bar{\xi}) \leq 0\} \cap {\cal N}_{\cal X}(\xi,\epsilon)$. As a result and because the cost and each of the component constraint functions are assumed Lipschitz continuous and $\phi_t$ for all $t \in [0,1]$ is Lipschitz continuous as is proven in Corollary \ref{corollary:x_lipschitz}, an application of Theorem \ref{thm:bangbang_result} allows us to show that for each $\epsilon > 0$ there exists an $\epsilon$--ball in the weak topology on ${\cal X}_p$ centered at $\xi$ such that $J(\xi_p) < J(\xi)$ for some $\xi_p \in \{\bar{\xi} \in {\cal X}_p \mid \Psi(\bar{\xi}) \leq 0\} \cap {\cal N}_w(\xi,\epsilon)$. Therefore, it follows that if $\theta(\xi) < 0$ and $\Psi(\xi) = 0$, then $\xi$ is not a local minimizer of the Switched System Optimal Control Problem.

Second, if $\theta(\xi) < 0$ and $\Psi(\xi) < 0$, then there exists a $\xi' \in {\cal X}_r$ such that $\D{J}(\xi;\xi'-\xi)$ is negative. Though $\D{\psi_{j,t}}(\xi;\xi'-\xi)$ maybe positive for some $j \in {\cal J}$ and $t \in [0,1]$, by employing the aforementioned first order approximation, we can show that for each $\epsilon > 0$ there exists an $\epsilon$--ball in the ${\cal X}$-norm centered at $\xi$ such that $J(\hat{\xi}) < J(\xi)$ for some $\hat{\xi} \in \{\bar{\xi} \in {\cal X}_r \mid \Psi(\bar{\xi}) \leq 0\} \cap {\cal N}_{\cal X}(\xi,\epsilon)$. As a result and because the cost and each of the constraint functions are assumed Lipschitz continuous and $\phi_t$ for all $t \in [0,1]$ is Lipschitz continuous as is proven in Corollary \ref{corollary:x_lipschitz}, an application of Theorem \ref{thm:bangbang_result} allows us to show that for each $\epsilon > 0$ there exists an $\epsilon$--ball in the weak topology on ${\cal X}_p$ centered at $\xi$ such that $J(\xi_p) < J(\xi)$ for some $\xi_p \in \{\bar{\xi} \in {\cal X}_p \mid \Psi(\bar{\xi}) \leq 0\} \cap {\cal N}_w(\xi,\epsilon)$. Therefore, it follows that if $\theta(\xi) < 0$ and $\Psi(\xi) < 0$, then $\xi$ is not a local minimizer of the Switched System Optimal Control Problem. In this case, the addition of the $\Psi$ term in $\zeta$ ensures that a direction that reduces the cost does not simultaneously require a decrease in the infeasibility in order to be considered as a potential descent direction.

Third, if $\theta(\xi) < 0$ and $\Psi(\xi) > 0$, then there exists a $\xi' \in {\cal X}_r$ such that $\D{\psi_{j,t}}(\xi;\xi'-\xi)$ is negative for all $j \in {\cal J}$ and $t \in [0,1]$. By employing the aforementioned first order approximation, we can show for each $\epsilon > 0$ there exists an $\epsilon$--ball in the ${\cal X}$-norm centered at $\xi$ such that $\Psi(\hat{\xi}) < \Psi(\xi)$ for some $\hat{\xi} \in {\cal N}_{\cal X}(\xi,\epsilon)$. As a result and because each of the constraint functions are assumed Lipschitz continuous and $\phi_t$ for all $t \in [0,1]$ is Lipschitz continuous as is proven in Corollary \ref{corollary:x_lipschitz}, an application of Theorem \ref{thm:bangbang_result} allows us to show that for each $\epsilon > 0$ there exists an $\epsilon$--ball in the weak topology on ${\cal X}_p$ centered at $\xi$ such that $\Psi(\xi_p) < \Psi(\xi)$ for some $\xi_p \in {\cal N}_w(\xi,\epsilon)$. Therefore, though it is clear that $\xi$ is not a local minimizer of the Switched System Optimal Control Problem since $\Psi(\xi) > 0$, it follows that if $\theta(\xi) < 0$ and $\Psi(\xi) > 0$, then it is possible to locally reduce the infeasibility of $\xi$. In this case, the addition of the $\D{J}$ term in $\zeta$ serves as a heuristic to ensure that the reduction in infeasibility does not come at the price of an undue increase in the cost. 

These observations are formalized in Theorem \ref{thm:theta_optimality_function} where we prove that if $\xi$ is a local minimizer of the Switched System Optimal Control Problem, then $\theta(\xi) = 0$, or that $\theta(\xi) = 0$ is a necessary condition for the optimality of $\xi$. To illustrate the importance of $\theta$ satisfying this property, recall how the directional derivative of a cost function is employed during unconstrained finite dimensional optimization. Since the directional derivative of the cost function at a point being equal to zero in all directions is a necessary condition for optimality for an unconstrained finite dimensional optimization problem, it is used as a stopping criterion by first order numerical algorithms (Corollary 1.1.3 and Algorithm Model 1.2.23 in \cite{Polak1997}). Similarly, by satisfying Theorem \ref{thm:theta_optimality_function}, $\theta$ is a necessary condition for optimality for the Switched System Optimal Control Problem and can therefore be used as a stopping criterion for a first order numerical optimal control algorithm trying to solve the Switched System Optimal Control Problem. Given $\theta$'s importance, we say \emph{a point, $\xi \in {\cal X}_p$, satisfies the optimality condition} if $\theta(\xi) = 0$.


\subsection{Choosing a Step Size and Projecting the Relaxed Discrete Input}
\label{subsec:project_relaxed_discrete_input}

Impressively, Theorem \ref{thm:bangbang_result} just as an existence result is sufficient to allow for the construction of an optimality function that encapsulates a necessary condition for optimality for the Switched System Optimal Control Problem. Unfortunately, Theorem \ref{thm:bangbang_result} is unable to describe how to exploit the descent direction, $g(\xi)$, since its proof provides no means to construct a pure input that approximates the behavior of a relaxed input while controlling the quality of the approximation. In this paper, we extend Theorem \ref{thm:bangbang_result} by devising a scheme that remedies this shortcoming. This allows for the development of a numerical optimal control algorithm for the Switched System Optimal Control Problem that first, performs optimal control over the relaxed optimization space and then projects the computed relaxed control into a pure control.

Before describing the construction of this projection, we describe how the descent direction, $g(\xi)$, can be exploited to construct a point in the relaxed optimization space that either reduces the cost (if the $\xi$ is feasible) or the infeasibility (if $\xi$ is infeasibile). Comparing our approach to finite dimensional optimization, the argument that minimizes $\zeta$ is a ``direction'' along which to move the inputs in order to reduce the cost in the relaxed optimization space, but we require an algorithm to choose a step size. We employ a line search algorithm similar to the traditional Armijo algorithm used during finite dimensional optimization in order to choose a step size (Algorithm Model 1.2.23 in \cite{Polak1997}). Fixing $\alpha \in (0,1)$ and $\beta \in ( 0, 1 )$, a step size for a point $\xi \in {\cal X}_p$ is chosen by solving the following optimization problem:
\begin{equation}
	\label{eq:armijo_line_search}
	\mu(\xi) = 
	\begin{cases}
    \min \Bigl\{ k \in \N \mid J\big( \xi + \beta^k ( g(\xi) - \xi ) \big) - J(\xi) \leq \alpha \beta^k \theta(\xi), \\ 
    \phantom{\min \Bigl\{ k \in \N \mid {}} \Psi\big( \xi + \beta^k( g(\xi) - \xi ) \big) \leq \alpha \beta^k \theta(\xi) \Bigr\}
    & \text{if}\ \Psi(\xi) \leq 0, \\
    \min \Bigl\{ k \in \N \mid \Psi\big( \xi + \beta^k ( g(\xi) - \xi ) \big) - \Psi(\xi) \leq \alpha \beta^k \theta(\xi) \Bigr\}
    & \text{if}\ \Psi(\xi) > 0.
	\end{cases}
\end{equation}
In Lemma \ref{lemma:mu_upper_bound}, we prove that for $\xi \in {\cal X}_p$, if $\theta(\xi) < 0$, then $\mu(\xi) < \infty$. Therefore, if $\theta(\xi) < 0$ for some $\xi \in {\cal X}_p$, then we can construct a descent direction, $g(\xi)$, and a step size, $\mu(\xi)$, and a new point $\left(\xi + \beta^{\mu(\xi)}( g(\xi) - \xi )\right) \in {\cal X}_r $ that produces a reduction in the cost (if $\xi$ is feasible) or a reduction in the infeasibility (if $\xi$ is infeasible).

We define the projection that takes this constructed point to a point belonging the pure optimization space while controlling the quality of approximation in two steps. First, we approximate the relaxed input by its $N$--th partial sum approximation via the Haar wavelet basis. To define this operation, ${\cal F}_N: L^2( [0,1], \R ) \cap BV( [0,1], \R ) \to L^2( [0,1], \R ) \cap BV( [0,1], \R )$, 
we employ the Haar wavelet (Section 7.2.2 in \cite{Mallat1999}): 
\begin{equation}
  \lambda(t) =
  \begin{cases}
    1 & \text{if}\ t \in \left[0,\frac{1}{2}\right), \\
    -1 & \text{if}\ t \in \left[ \frac{1}{2}, 1 \right), \\
    0 & \text{otherwise}. \\
  \end{cases}
\end{equation}
Letting $\mathds{1}: \R \to \R$ be the constant function equal to one and $b_{kj}: [0,1] \to \R$ for $k \in \N$ and $j \in \{ 0, \ldots, 2^k - 1\}$, be defined as $b_{kj}(t) = \lambda\big( 2^kt - j \big)$, the projection ${\cal F}_N$  for some $c \in L^2( [0,1], \R ) \cap BV( [0,1], \R ) \to L^2( [0,1], \R ) \cap BV( [0,1], \R )$ is defined as:
\begin{equation}
  \label{eq:wavelet_approx}
	[ {\cal F}_N(c) ](t) = \langle c, \mathds{1} \rangle + \sum_{k=0}^N \sum_{j=0}^{2^k-1} \langle c, b_{kj} \rangle \frac{ b_{kj}(t) }{\| b_{kj} \|_{L^2}^2}.
\end{equation}
Note that the inner product here is the traditional Hilbert space inner product. 

This projection is then applied to each of the coordinates of an element in the relaxed optimization space. To avoid introducing additional notation, we let the coordinate-wise application of ${\cal F}_N$ to some relaxed discrete input $d \in {\cal D}_r$ be denoted as ${\cal F}_N(d)$ and similarly for some continuous input $u \in {\cal U}$. Lemma \ref{lemma:proj_reasonable} proves that for each $N \in \N$, each $t \in [0,1]$, and each $i \in \{1,\ldots,q\}$, $ \left[{\cal F}_N(d)\right]_i(t) \in [0,1]$ and $\sum_{i=1}^q \left[{\cal F}_N(d)\right]_i(t) = 1$ for the projection ${\cal F}_N(d)$. Therefore it follows that for each $d \in {\cal D}_r$, ${\cal F}_N(d) \in {\cal D}_r$.

Second, we project the output of ${\cal F}_N(d)$ to a pure discrete input by employing the function ${\cal P}_N: {\cal D}_r \to {\cal D}_p$ which computes the pulse width modulation of its argument with frequency $2^{-N}$:
\begin{equation}
	\label{eq:pwm}
	[ {\cal P}_N(d) ]_i(t) = 
	\begin{cases}
    1 & \text{if}\ t \in \left[ 2^{-N} \left( k + \sum_{j=1}^{i-1} d_j\left(\frac{k}{2^N}\right) \right), 2^{-N} \left( k + \sum_{j=1}^i d_j\left(\frac{k}{2^N}\right) \right) \right),\ k \in \left\{ 0, 1, \ldots, 2^N - 1 \right\}, \\
		0 & \text{otherwise}. \\
	\end{cases}
\end{equation}
Lemma \ref{lemma:proj_reasonable} proves that for each $N \in \N$, each $t \in [0,1]$, and each $i \in \{1,\ldots,q\}$, $\big[ {\cal P}_N \big( {\cal F}_N(d) \big) \big]_i(t) \in \{0,1\}$ and $\sum_{i=1}^q \big[ {\cal P}_N \big( {\cal F}_N(d) \big) \big]_i(t) = 1$. This proves that ${\cal P}_N \big( {\cal F}_N(d) \big) \in {\cal D}_p$ for each $d \in {\cal D}_r$. 

Fixing $N \in \N$, we compose the two projections and define $\rho_N:{\cal X}_r \to {\cal X}_p$ as:
\begin{equation}
	\label{eq:rho}
	\rho_N(u,d) = \Big( {\cal F}_N(u), {\cal P}_N\big( {\cal F}_N(d) \big) \Big).
\end{equation}
Critically, as shown in Theorem \ref{thm:quality_of_approximation}, this projection allows us to extend Theorem \ref{thm:bangbang_result} by constructing an upper bound that goes to zero as $N$ goes infinity between the error of employing the relaxed control rather than its projection in the solution of Differential Equation \eqref{eq:traj_xi}. Therefore in a fashion similar to applying the Armijo algorithm, we choose an $N \in \N$ at which to perform pulse width modulation by performing a line search. Fixing $\bar{\alpha} \in (0,\infty)$, $\bar{\beta} \in \left( \frac{1}{\sqrt{2}}, 1 \right)$, and $\omega \in (0,1)$, a frequency at which to perform pulse width modulation for a point $\xi \in {\cal X}_p$ is computed by solving the following optimization problem:
\begin{equation}
	\label{eq:armijo_pwm}
	\nu(\xi) = 
	\begin{cases}
		\min \Bigl\{k \in \N \mid J\big( \rho_k( \xi + \beta^{\mu(\xi)}( g(\xi) - \xi ) ) \big) - J(\xi) \leq \big( \alpha \beta^{\mu(\xi)} - \bar{\alpha} \bar{\beta}^k \big) \theta(\xi), \\
    \phantom{\min \big\{k \in \N \mid {}} \Psi\big( \rho_k( \xi + \beta^{\mu(\xi)}( g(\xi) - \xi ) ) \big) \leq 0,\ 
    \bar{\alpha} \bar{\beta}^k \leq ( 1 - \omega ) \alpha \beta^{\mu(\xi)} \Bigr\} 
    & \text{if}\ \Psi(\xi) \leq 0, \\
		\min \Bigl\{k \in \N \mid \Psi\big( \rho_k( \xi + \beta^{\mu(\xi)}( g(\xi) - \xi ) ) \big) - \Psi(\xi) \leq \big( \alpha \beta^{\mu(\xi)} - \bar{\alpha} \bar{\beta}^k \big) \theta(\xi), \\
    \phantom{\min \big\{k \in \N \mid {}} \bar{\alpha} \bar{\beta}^k \leq ( 1 - \omega ) \alpha \beta^{\mu(\xi)} \Bigr\} 
    & \text{if}\ \Psi(\xi) > 0.
	\end{cases}
\end{equation}
In Lemma \ref{lemma:nu_upper_bound}, we prove that for $\xi \in {\cal X}_p$, if $\theta(\xi) < 0$, then $\nu(\xi) < \infty$. Therefore, if $\theta(\xi) < 0$ for some $\xi \in {\cal X}_p$, then we can construct a descent direction, $g(\xi)$, a step size, $\mu(\xi)$, a frequency at which to perform pulse width modulation, $\nu\left(\xi\right)$, and a new point $\rho_{\nu(\xi)}\bigl( \xi + \beta^{\mu(\xi)} ( g(\xi) - \xi ) \bigr) \in {\cal X}_p$ that produces a reduction in the cost (if $\xi$ is feasible) or a reduction in the infeasibility (if $\xi$ is infeasible).

\subsection{Switched System Optimal Control Algorithm}

Consolidating our definitions, Algorithm \ref{algo:main_algo} describes our numerical method to solve the Switched System Optimal Control Problem. For analysis purposes, we define $\Gamma: {\cal X}_p \to {\cal X}_p$ by 
\begin{equation}
  \label{eq:gamma_def}
  \Gamma(\xi) = \rho_{\nu(\xi)} \big( \xi + \beta^{\mu(\xi)} ( g(\xi) - \xi ) \big).  
\end{equation}
We say $\{\xi_j\}_{j \in \N}$ is \emph{a sequence generated by Algorithm \ref{algo:main_algo}} if $\xi_{j+1} = \Gamma(\xi_j)$ for each $j \in \N$. We can prove several important properties about the sequence generated by Algorithm \ref{algo:main_algo}. First, in Lemma \ref{lemma:phase12}, we prove that if there exists $i_0 \in \N$ such that $\Psi(\xi_{i_0}) \leq 0$, then $\Psi(\xi_i) \leq 0$ for each $i \geq i_0$. That is, if the Algorithm constructs a feasible point, then the sequence of points generated after this feasible point are always feasible. Second, in Theorem \ref{thm:convergence}, we prove $\lim_{j\to\infty} \theta(\xi_j) = 0$ or that Algorithm \ref{algo:main_algo} converges to a point that satisfies the optimality condition.


\begin{algorithm}[!ht]
  \begin{algorithmic}[1]
    \REQUIRE $\xi_0 \in {\cal X}_p$, 
    $\alpha \in (0,1)$, 
    $\bar{\alpha} \in (0,\infty)$, 
    $\beta \in ( 0, 1 )$, 
    $\bar{\beta} \in \left( \frac{1}{\sqrt{2}}, 1 \right)$, 
    $\gamma \in (0,\infty)$,
    $\omega \in (0,1)$.
    \STATE Set $j = 0$.
    \STATE \label{algo:step1} Compute $\theta(\xi_j)$ as defined in Equation \eqref{eq:theta_definition}. 
    \IF{\label{algo:step3} $\theta(\xi_j) = 0$}
    \RETURN $\xi_j$.
    \ENDIF
    \STATE Compute $g(\xi_j)$ as defined in Equation \eqref{eq:theta_definition}.
    \STATE Compute $\mu(\xi_j)$ as defined in Equation \eqref{eq:armijo_line_search}.
    \STATE Compute $\nu(\xi_j)$ as defined in Equation \eqref{eq:armijo_pwm}.
    \STATE Set $\xi_{j+1} = \rho_{\nu(\xi_j)}\bigl( \xi_j + \beta^{\mu(\xi_j)} ( g(\xi_j) - \xi_j ) \bigr)$, as defined in Equation \eqref{eq:rho}.
    \STATE Replace $j$ by $j+1$ and go to Line \ref{algo:step1}.
  \end{algorithmic}
  \caption{Optimization Algorithm for the Switched System Optimal Control Problem}
  \label{algo:main_algo}
\end{algorithm}

\section{Algorithm Analysis}
\label{sec:algo_analysis}

In this section, we derive the various components of Algorithm \ref{algo:main_algo} and prove that Algorithm \ref{algo:main_algo} converges to a point that satisfies our optimality condition. Our argument proceeds as follows: first, we prove the continuity of the state, cost, and constraint, which we employ in latter arguments; second, we construct the components of the optimality function and prove that these components satisfy various properties that ensure that the well-posedness of the optimality function; third, we prove that we can control the quality of approximation between the trajectories generated by a relaxed discrete input and its projection by $\rho_N$ as a function of $N$; finally, we prove the convergence of our algorithm.

\subsection{Continuity}

In this subsection, we prove the continuity of the state, cost, and constraint. We begin by proving the continuity of the solution to Differential Equation \eqref{eq:traj_xi} with respect to $\xi$ by proving that this mapping is sequentially continuous:
\begin{lemma}
	\label{lemma:x_sequential_continuity}
	Let $\{\xi_j\}_{j=1}^{\infty} \subset {\cal X}_r$ be a convergent sequence with limit $\xi \in {\cal X}_r$.
  Then the corresponding sequence of trajectories $\{ x^{(\xi_j)} \}_{j=1}^{\infty}$, as defined in Equation \eqref{eq:traj_xi}, converges uniformly to $x^{(\xi)}$.
\end{lemma}
\begin{proof}
  For notational convenience, let $\xi_j = (u_j,d_j)$, $\xi = (u,d)$, and $\phi_t$ as defined in Equation \eqref{eq:flow_xi}.
  We begin by proving the convergence of $\{ \phi_t(\xi) \}_{j=1}^{\infty}$ to $\phi_t(\xi)$ for each $t \in [0,1]$. 
  Consider
	\begin{equation}
		\left\| \phi_t(\xi_j) - \phi_t(\xi) \right\|_2 
    = \left\| \int_0^t \sum_{i=1}^q [d_j]_i(\tau) f\bigl( \tau, \phi_\tau(\xi_j), u_j(\tau), e_i \bigr) - d_i(\tau) f\bigl( \tau, \phi_\tau(\xi), u(\tau), e_i \bigr) d\tau \right\|_2.
	\end{equation}
	Therefore,
	\begin{multline}
		\left\| \phi_t(\xi_j) - \phi_t(\xi) \right\|_2 
    = \Biggl\| \int_0^t \sum_{i=1}^q \bigl( [d_j]_i(\tau) - d_i(\tau) \bigr) f\bigl( \tau, \phi_\tau(\xi_j), u_j(\tau), e_i \bigr) + \\
    + d_i(\tau) \bigl( f\bigl( \tau, \phi_\tau(\xi_j), u_j(\tau), e_i \bigr) - f\bigl( \tau, \phi_{\tau}(\xi), u_j(\tau), e_i \bigr) \bigr) + \\
    + d_i(\tau) \bigl( f\bigl( \tau, \phi_{\tau}(\xi), u_j(\tau), e_i \bigr) - f\bigl( \tau, \phi_{\tau}(\xi), u(\tau), e_i \bigr) \bigr)  d\tau \Biggr\|_2.
	\end{multline}
	Applying the triangle inequality, Assumption \ref{assump:fns_continuity}, Condition \ref{corollary:f_bounded} in Corollary \ref{corollary:fns_bounded}, and the boundedness of $d$, we have that there exists a $C > 0$ such that
	\begin{equation}
		\left\| \phi_t(\xi_j) - \phi_t(\xi) \right\|_2 
    \leq \int_0^1 \sum_{i=1}^q C \bigl| [d_j]_i(\tau) - d_i(\tau) \bigr| + 
    L \left\| \phi_\tau(\xi_j) - \phi_{\tau}(\xi) \right\|_2 + L \left\| u_j(\tau) - u(\tau) \right\|_2 d\tau.
	\end{equation}
	Applying the Bellman-Gronwall Inequality (Lemma 5.6.4 in \cite{Polak1997}), we have that
	\begin{equation}
		\left\| \phi_t(\xi_j) - \phi_t(\xi) \right\|_2 
    \leq e^{L} \left( \int_0^1 C \left\| d_{j}(\tau) - d(\tau) \right\|_1 + L \left\|u_j(\tau) - u(\tau) \right\|_2 d\tau \right).
	\end{equation}
	Note that $\| u \|_2 \leq \| u \|_1$ for each $u \in \R^m$. Then applying Holder's inequality (Proposition 6.2 in \cite{Folland1999}) to the vector valued function, we have:
	\begin{equation}
    \label{eq:1to2_norm_holder_trick}
    \int_0^1 \left\| d_j(\tau) - d(\tau) \right\|_1 d\tau \leq \left\| d_j - d \right\|_{L^2},\quad \text{and}\quad
    \int_0^1 \left\| u_j(\tau) - u(\tau) \right\|_1 d\tau \leq \left\| u_j - u \right\|_{L^2}.
	\end{equation}
	Since the sequence $\xi_j$ converges to $\xi$, for every $\epsilon > 0$ we know there exists some $j_0$ such that for all $j$ greater than $j_0$, $\| \xi_j - \xi \|_{\cal X} \leq \epsilon$. 
  Therefore $\| \phi_t(\xi_j) - \phi_t(\xi) \|_2 \leq e^{L} ( L + C ) \epsilon$, which proves the convergence of $\{ \phi_t(\xi_j) \}_{j=1}^{\infty}$ to $\phi_t(\xi)$ for each $t \in [0,1]$ as $j \to \infty$. 
  Since this bound does not depend on $t$, we in fact have the uniform convergence of $\{ x^{(\xi_j)} \}_{j=1}^{\infty}$ to $x^{(\xi)}$ as $j \to \infty$, hence obtaining our desired result.
\end{proof}

Notice that since ${\cal X}_r$ is a metric space, the previous result proves that the function $\phi_t$ which assigns $\xi \in {\cal X}_r$ to $\phi_t(\xi)$ as the solution of Differential Equation \eqref{eq:traj_xi} employing the notation defined in Equation \eqref{eq:flow_xi} is continuous.
\begin{corollary}
	\label{corollary:x_continuous}
	The function $\phi_t$ that maps $\xi \in {\cal X}_r$ to $\phi_t(\xi)$ as the solution of Differential Equation \eqref{eq:traj_xi} where we employ the notation defined in Equation \eqref{eq:flow_xi} is continuous for all $t \in [0,1]$.
\end{corollary}

In fact, our arguments have shown that this mapping is Lipschitz continuous:
\begin{corollary}
	\label{corollary:x_lipschitz}
	There exists a constant $L > 0$ such that for each $\xi_1,\xi_2 \in {\cal X}_r$ and $t \in [0,1]$:
	\begin{equation}
		\label{eq:x_lipschitz}
		\| \phi_t(\xi_1) - \phi_t(\xi_2) \|_2 \leq L \| \xi_1 - \xi_2 \|_{\cal X},
	\end{equation}
	where $\phi_t(\xi)$ is as defined in Equation \eqref{eq:flow_xi}.
\end{corollary}

As a result of this corollary, we immediately have the following results:
\begin{corollary}
	\label{corollary:vf_lipschitz}
	There exists a constant $L > 0$ such that for each $\xi_1 = (u_1,d_1) \in {\cal X}_r$, $\xi_2 = (u_2,d_2) \in {\cal X}_r$, and $t \in [0,1]$:
	\begin{enumerate_parentesis}
		\item \label{corollary:fxi_lipschitz} 
      \begin{math}
        \begin{aligned}[t]
          \bigl\| f\bigl( t, \phi_t(\xi_1), u_1(t), d_1(t) \bigr) - f\bigl( t, \phi_t(\xi_2), &u_2(t), d_2(t) \bigr) \bigr\|_2 \leq \\
          &\leq L\bigl( \left\| \xi_1 - \xi_2 \right\|_{\cal X} + \left\| u_1(t) - u_2(t)\right\|_2 + \left\| d_1(t) - d_2(t) \right\|_2 \bigr),
        \end{aligned}
      \end{math}
		\item \label{corollary:dfdxxi_lipschitz} 
      \begin{math}
        \begin{aligned}[t]
          \biggl\| \frac{\partial f}{\partial x} \bigl( t, \phi_t(\xi_1), u_1(t), d_1(t) \bigr) 
            - \frac{\partial f}{\partial x}\bigl( t, \phi_t(\xi_2), &u_2(t), d_2(t) \bigr) \biggr\|_{i,2} \leq \\
          &\leq L\bigl( \left\| \xi_1 - \xi_2 \right\|_{\cal X} + \left\| u_1(t) - u_2(t) \right\|_2 + \left\|d_1(t) - d_2(t) \right\|_2 \bigr),
        \end{aligned}
      \end{math}
		\item \label{corollary:dfduxi_lipschitz} 
      \begin{math}
        \begin{aligned}[t]
          \biggl\| \frac{\partial f}{\partial u}\bigl( t, \phi_t(\xi_1), u_1(t), d_1(t) \bigr) 
            - \frac{\partial f}{\partial u}\bigl( t, \phi_t(\xi_2), &u_2(t), d_2(t) \bigr) \biggr\|_{i,2} \leq \\
          &\leq L\bigl( \left\| \xi_1 - \xi_2 \right\|_{\cal X} + \left\| u_1(t) - u_2(t) \right\|_2 + \left\|d_1(t) - d_2(t) \right\|_2 \bigr),
        \end{aligned}
      \end{math}
	\end{enumerate_parentesis}
	where $\phi_t(\xi)$ is as defined in Equation \eqref{eq:flow_xi}.
\end{corollary}
\begin{proof}
	The proof of Condition \ref{corollary:fxi_lipschitz} follows by the fact that the vector field $f$ is Lipschitz in all its arguments, as shown in the proof of Theorem \ref{thm:existence_and_uniqueness}, and applying Corollary \ref{corollary:x_lipschitz}.
  The remaining conditions follow in a similar fashion.
\end{proof}

\begin{corollary}
  \label{corollary:constraints_lipschitz}
  There exists a constant $L > 0$ such that for each $\xi_1,\xi_2 \in {\cal X}_r$, $j \in {\cal J}$, and $t \in [0,1]$:
  \begin{enumerate_parentesis}
	\item \label{corollary:phixi_lipschitz} 
    $\left| h_0\bigl( \phi_1(\xi_1) \bigr) - h_0\bigl( \phi_1(\xi_2) \bigr) \right| 
    \leq L \left\| \xi_1 - \xi_2 \right\|_{\cal X}$,
	\item \label{corollary:dphidxxi_lipschitz} 
    $\Bigl\| \frac{\partial h_0}{\partial x}\bigl( \phi_1(\xi_1) \bigr) - \frac{\partial h_0}{\partial x}\bigl( \phi_1(\xi_2) \bigr) \Bigr\|_2 
    \leq L \left\| \xi_1 - \xi_2 \right\|_{\cal X}$,
	\item \label{corollary:hjxi_lipschitz} 
    $\left| h_j\bigl( \phi_t(\xi_1) \bigr) - h_j\bigl( \phi_t(\xi_2) \bigr) \right| 
    \leq L \left\| \xi_1 - \xi_2 \right\|_{\cal X}$,
	\item \label{corollary:dhjdxxi_lipschitz} 
    $\left\| \frac{\partial h_j}{\partial x}\bigl( \phi_t(\xi_1) \bigr) - \frac{\partial h_j}{\partial x}\bigl( \phi_t(\xi_2) \bigr) \right\|_2 
    \leq L \left\| \xi_1 - \xi_2 \right\|_{\cal X}$,
  \end{enumerate_parentesis}
  where $\phi_t(\xi)$ is as defined in Equation \eqref{eq:flow_xi}.
\end{corollary}
\begin{proof}
	This result follows by Assumption \ref{assump:constraint_fns} and Corollary \ref{corollary:x_lipschitz}.
\end{proof}

Even though it is a straightforward consequence of Condition \ref{corollary:phixi_lipschitz} in Corollary \ref{corollary:constraints_lipschitz}, we write the following result to stress its importance.
\begin{corollary}
	\label{corollary:J_continuous}
  There exists a constant $L > 0$ such that, for each $\xi_1, \xi_2 \in {\cal X}_r$:
  \begin{equation}
    \left\lvert J( \xi_1 ) - J( \xi_2 ) \right\rvert \leq L \left\lVert \xi_1 - \xi_2 \right\rVert_{\cal X}
  \end{equation}
  where $J$ is as defined in Equation \eqref{eq:cost}.
\end{corollary}

In fact, the $\Psi$ is also Lipschitz continuous:
\begin{lemma}
	\label{lemma:psi_continuous}
  There exists a constant $L > 0$ such that, for each $\xi_1, \xi_2 \in {\cal X}_r$:
  \begin{equation}
    \left\lvert \Psi( \xi_1 ) - \Psi( \xi_2 ) \right\rvert \leq L \left\lVert \xi_1 - \xi_2 \right\rVert_{\cal X}
  \end{equation}
  where $\Psi$ is as defined in Equation \eqref{eq:super_const}.
\end{lemma}
\begin{proof}
  Since the maximum in $\Psi$ is taken over ${\cal J} \times [0,1]$, which is compact, and the maps $(j,t) \mapsto \psi_{j,t}(\xi)$ are continuous for each $\xi \in {\cal X}$, we know from Condition \ref{corollary:hjxi_lipschitz} in Corollary \ref{corollary:constraints_lipschitz} that there exists $L > 0$ such that,
  \begin{equation}
    \begin{aligned}
      \Psi(\xi_1) - \Psi(\xi_2) &= \max_{(j,t) \in {\cal J} \times [0,1]} \psi_{j,t}(\xi_1) - \max_{(j,t) \in {\cal J} \times [0,1]} \psi_{j,t}(\xi_2) \\
      &\leq \max_{(j,t) \in {\cal J} \times [0,1]} \psi_{j,t}(\xi_1) - \psi_{j,t}(\xi_2) \\
      &\leq L \left\lVert \xi_1 - \xi_2 \right\rVert_{\cal X}.
    \end{aligned}
  \end{equation}
  By reversing $\xi_1$ and $\xi_2$, and applying the same argument we get the desired result.
\end{proof}
\subsection{Derivation of Algorithm Terms}

Next, we formally derive the components of the optimality function and prove the well-posedness of the optimality function. We begin by deriving the formal expression for the directional derivative of the trajectory of the switched system.

\begin{lemma}
  \label{lemma:dxt_definition}
  Let $\xi = (u,d) \in {\cal X}_r$, $\xi' = (u',d') \in {\cal X}$, and let $\phi_t: {\cal X}_r \to \R^n$ be as defined in Equation \eqref{eq:flow_xi}.
  Then the directional derivative of $\phi_t$, as defined in Equation \eqref{eq:operator_D_definition}, is given by
  \begin{equation}
    \label{eq:dxt_definition}
    \D{\phi_t}(\xi;\xi') = \int_0^t \Phi^{(\xi)}(t,\tau) \left( 
      \frac{\partial f}{\partial u}\big( \tau, \phi_\tau(\xi), u(\tau), d(\tau) \big) u'(\tau) 
    + \sum_{i=1}^q f\big( \tau, \phi_\tau(\xi), u(\tau), e_i \big) d_i'(\tau) \right) d\tau,
  \end{equation}
  where $\Phi^{(\xi)}(t,\tau)$ is the unique solution of the following matrix differential equation:
  \begin{equation}
    \label{eq:stm}
    \frac{\partial \Phi}{\partial t}(t,\tau) = \frac{\partial f}{\partial x}\bigl( t, \phi_t(\xi), u(t), d(t) \bigr) \Phi(t,\tau), \quad t \in [0,1], \quad \Phi(\tau,\tau) = I.
  \end{equation}

\end{lemma}
\begin{proof}
  For notational convenience, let $x^{(\lambda)} = x^{( \xi + \lambda \xi' )}$, $u^{(\lambda)} = u + \lambda u'$, and $d^{(\lambda)} = d + \lambda d'$. 
  Then, if we define $\Delta x^{(\lambda)} = x^{(\lambda)} - x^{(\xi)}$, 
  \begin{equation}
    \Delta x^{(\lambda)}(t) = \int_0^t f\big( \tau, x^{(\lambda)}(\tau), u^{(\lambda)}(\tau), d^{(\lambda)}(\tau) \big) - 
      f\big( \tau, x^{(\xi)}(t), u(\tau), d(\tau) \big) d\tau,
  \end{equation}
  thus, 
  \begin{multline}
    \Delta x^{(\lambda)}(t) =
    \int_0^t f\big( \tau, x^{(\lambda)}(\tau), u^{(\lambda)}(\tau), d^{(\lambda)}(\tau) \big) - 
      f\big( \tau, x^{(\lambda)}(t), u^{(\lambda)}(\tau), d(\tau) \big) d\tau + \\
    + \int_0^t f\big( \tau, x^{(\lambda)}(\tau), u^{(\lambda)}(\tau), d(\tau) \big) - 
      f\big( \tau, x^{(\xi)}(t), u^{(\lambda)}(\tau), d(\tau) \big) d\tau + \\
    + \int_0^t f\big( \tau, x^{(\xi)}(\tau), u^{(\lambda)}(\tau), d(\tau) \big) - 
      f\big( \tau, x^{(\xi)}(t), u(\tau), d(\tau) \big) d\tau,
  \end{multline}
  and applying the Mean Value Theorem,
  \begin{multline}
    \Delta x^{(\lambda)}(t) =
    \int_0^t \lambda \sum_{i=1}^q d'_i(\tau) f\big( \tau, x^{(\lambda)}(\tau), u^{(\lambda)}(\tau), e_i \big) + \\
    + \int_0^t \frac{\partial f}{\partial x}\big( \tau, x^{(\xi)}(\tau) + \nu_x(\tau) \Delta x^{(\lambda)}(\tau), u^{(\lambda)}(\tau), d(\tau) \big) \Delta x^{(\lambda)}(\tau) + \\
    + \int_0^t \lambda \frac{\partial f}{\partial u}\big( \tau, x^{(\xi)}(\tau), u(\tau) + \nu_u(\tau) \lambda u'(\tau), d(\tau) \big) u'(\tau) dt,
  \end{multline}
  where $\nu_u,\nu_x:[0,t] \to [0,1]$.

  Let $z(t)$ be the unique solution of the following differential equation:
  \begin{multline}
    \label{eq:dxt_definition_ode}
    \dot{z}(\tau) = \frac{\partial f}{\partial x}\big( \tau, x^{(\xi)}(\tau), u(\tau), d(\tau) \big) z(\tau) + \frac{\partial f}{\partial u}\big( \tau, x^{(\xi)}(\tau), u(\tau), d(\tau) \big) u'(\tau) + \\
    + \sum_{i=1}^q d'_i(\tau) f\big( \tau, x^{(\xi)}(\tau), u(\tau), e_i \big), \quad \tau \in [0,t], \quad z(0) = 0.
  \end{multline}
  We want to show that $\lim_{\lambda \downarrow 0} \left\| \frac{\Delta x^{(\lambda)}(t)}{\lambda} - z(t) \right\|_2 = 0$.
To prove this, consider the following inequalities that follow from Condition \ref{assump:dfdx_lipschitz} in Assumption \ref{assump:fns_continuity}:
  \begin{equation}
    \begin{aligned}
      \biggl\| \int_0^t \frac{\partial f}{\partial x}\big( \tau, &x^{(\xi)}(\tau), u(\tau), d(\tau) \big) z(\tau) - 
        \frac{\partial f}{\partial x}\big( \tau, x^{(\xi)}(\tau) + \nu_x(\tau) \Delta x^{(\lambda)}(\tau), u^{(\lambda)}(\tau), d(\tau) \big) \frac{\Delta x^{(\lambda)}(\tau)}{\lambda} d\tau \biggr\|_2 \leq \\
      &\leq \int_0^t \left\| \frac{\partial f}{\partial x} \right\|_{L^\infty} \left\| z(\tau) - \frac{\Delta x^{(\lambda)}(\tau)}{\lambda} \right\|_2 d\tau + 
      \int_0^t L \left( \bigl\| \nu_x(\tau) \Delta x^{(\lambda)}(\tau) \bigr\|_2 + \lambda \left\| u'(\tau) \right\|_2 \right) \left\| z(t) \right\|_2 d\tau \\
      &\leq L \int_0^t \left\| z(\tau) - \frac{\Delta x^{(\lambda)}(\tau)}{\lambda} \right\|_2 d\tau + 
      L \int_0^t \left( \bigl\| \Delta x^{(\lambda)}(\tau) \bigr\|_2 + \lambda \left\| u'(\tau) \right\|_2 \right) \left\| z(t) \right\|_2 d\tau,
    \end{aligned}
  \end{equation}
  also from Condition \ref{assump:dfdu_lipschitz} in Assumption \ref{assump:fns_continuity}:
  \begin{multline}
    \left\| \int_0^t \left( \frac{\partial f}{\partial u}\big( \tau, x^{(\xi)}(\tau), u(\tau), d(\tau) \big) - 
        \frac{\partial f}{\partial u}\big( \tau, x^{(\xi)}(\tau), u(\tau) + \nu_u(\tau) \lambda u'(\tau), d(\tau) \big) \right) u'(\tau) \right\|_2 \leq \\
    \leq L \int_0^t \lambda \left\| \nu_u(\tau) u'(\tau) \right\|_2 \left\| u'(\tau) \right\|_2 d\tau \leq L \int_0^t \lambda \left\| u'(\tau) \right\|_2^2 d\tau,
  \end{multline}
  and from Condition \ref{assump:f_lipschitz} in Assumption \ref{assump:fns_continuity}:
  \begin{multline}
    \left\| \int_0^t \sum_{i=1}^q d_i'(\tau) \left( f\big( \tau, x^{(\xi)}(\tau), u(\tau), e_i \big) - 
      f\big( \tau, x^{(\lambda)}(\tau), u^{(\lambda)}(\tau), e_i \big) \right) d\tau \right\|_2 \leq \\
    \leq L \int_0^t \sum_{i=1}^q d_i'(\tau) \left( \bigl\| \Delta x^{(\lambda)}(\tau) \bigr\|_2 + \lambda \left\| u'(\tau) \right\|_2 \right) d\tau.
  \end{multline}
  Now, using the Bellman-Gronwall Inequality (Lemma 5.6.4 in \cite{Polak1997}) and the inequalities above,
  \begin{multline}
    \left\| \frac{\Delta x^{(\lambda)}(t)}{\lambda} - z(t) \right\|_2 \leq 
      e^{Lt} L \Bigg( \int_0^t \left( \bigl\| \Delta x^{(\lambda)}(\tau) \bigr\|_2 + \lambda \left\| u'(\tau) \right\|_2 \right) \left\| z(t) \right\|_2 + \lambda\left\| u'(\tau) \right\|_2^2 + \\ 
      + \sum_{i=1}^q d_i'(\tau) \left( \bigl\| \Delta x^{(\lambda)}(\tau) \bigr\|_2 + \lambda \left\| u'(\tau) \right\|_2 \right) d\tau \Bigg),
  \end{multline}
  but note that every term in the integral above is bounded, and $\Delta x^{(\lambda)}(\tau) \to 0$ for each $\tau \in [0,t]$ since $x^{(\lambda)} \to x^{(\xi)}$ uniformly as shown in Lemma \ref{lemma:x_sequential_continuity}, thus by the Dominated Convergence Theorem (Theorem 2.24 in \cite{Folland1999}) and by noting that $\D{\phi_t}(\xi;\xi')$, as defined in Equation \eqref{eq:dxt_definition}, is exactly the solution of Differential Equation \eqref{eq:dxt_definition_ode} we get:
  \begin{equation}
    \lim_{\lambda \downarrow 0} \left\| \frac{\Delta x^{(\lambda)}(t)}{\lambda} - z(t) \right\|_2 
    = \lim_{\lambda \downarrow 0} \frac{1}{\lambda} \bigl\| x^{(\xi + \lambda \xi')}(t) - x^{(\xi)}(t) - D{\phi_t}(\xi;\lambda\xi') \bigr\|_2 = 0.
  \end{equation}
  The result of the Lemma then follows.

\end{proof}

Next, we prove that $\d{\phi_t}$ is bounded by proving that $\Phi^{(\xi)}$ is bounded:
 \begin{corollary}
  \label{corollary:stm_bounded}
  There exists a constant $C > 0$ such that for each $t,\tau \in [0,1]$ and $\xi \in {\cal X}_r$:
  \begin{equation}
    \bigl\| \Phi^{(\xi)}(t,\tau) \bigr\|_{i,2} \leq C,
  \end{equation}
  where $\Phi^{(\xi)}(t,\tau)$ is the solution to Differential Equation \eqref{eq:stm}.
\end{corollary}
\begin{proof}
  Notice that, since the induced matrix norm is submultiplicative,
  \begin{equation}
    \begin{aligned}
      \left\| \Phi^{(\xi)}(t,\tau) \right\|_{i,2} 
      &= \left\| \Phi^{(\xi)}(t,\tau) + \int_{\tau}^t \left( \frac{\partial f}{\partial x}\bigl( s, x^{(\xi)}(s), u(s), d(s) \bigr) \Phi^{(\xi)}(s,\tau) \right) ds \right\|_{i,2} \\
      &\leq 1 + \int_{\tau}^t \left\| \frac{\partial f}{\partial x} \bigl( s, x^{(\xi)}(s), u(s), d(s) \bigr) \right\|_{i,2} \left\|\Phi^{(\xi)}(t,s) \right\|_{i,2} ds \\
      &\leq e^{qC},
    \end{aligned}
  \end{equation}
  where in the last step we employed Condition \ref{corollary:f_bounded} from Corollary \ref{corollary:fns_bounded} with a constant $C > 0$ and the Bellman-Gronwall Inequality.
\end{proof}

\begin{corollary}
  \label{corollary:dxt_bounded}
  There exists a constant $C > 0$ such that for all $\xi \in {\cal X}_r$, $\xi' \in {\cal X}$, and $t \in [0,1]$:
  \begin{equation}
    \label{eq:dxt_bounded}
    \left\| \D{\phi_t}(\xi;\xi') \right\|_2 \leq C \left\| \xi' \right\|_{\cal X},
  \end{equation}
  where $\D{\phi_t}$ is as defined in Equation \eqref{eq:dxt_definition}.
\end{corollary}
\begin{proof}
  This result follows by employing the Cauchy-Schwarz Inequality, Corollary \ref{corollary:fns_bounded} and Corollary \ref{corollary:stm_bounded}.
\end{proof}

In fact, we can actually prove the Lipschitz continuity of $\Phi^{(\xi)}$:
\begin{lemma}
  \label{lemma:stm_lipschitz}
  There exists a constant $L > 0$ such that for each $\xi_1,\xi_2 \in {\cal X}_r$ and each $t,\tau \in [0,1]$:
  \begin{equation}
    \label{eq:stm_lipschitz}
    \left\| \Phi^{(\xi_1)}(t,\tau) - \Phi^{(\xi_2)}(t,\tau) \right\|_{i,2} \leq L \left\| \xi_1 - \xi_2 \right\|_{\cal X},
  \end{equation}
  where $\Phi^{(\xi)}$ is the solution to Differential Equation \eqref{eq:stm}. 
\end{lemma}
\begin{proof}
  Letting $\xi_1 = (u_1,d_1) \in {\cal X}_r$ and $\xi_2 = (u_2,d_2) \in {\cal X}_r$ and by applying the Triangular Inequality and noticing the induced matrix norm is compatible, observe:
  \begin{multline}
      \left\| \Phi^{(\xi_1)}(t,\tau) - \Phi^{(\xi_2)}(t,\tau) \right\|_{i,2} \leq 
      \int_{\tau}^t \left( \left\|\frac{\partial f}{\partial x}\left(s,x^{(\xi_2)}(s),u_2(s),d_2(s)\right) \right\|_{i,2} \left\| \Phi^{(\xi_1)}(s,\tau) - \Phi^{(\xi_2)}(s,\tau)\right\|_{i,2} \right) ds + \\
       + \int_{\tau}^t \left( \left\|\frac{\partial f}{\partial x}\left( s,x^{(\xi_1)}(s),u_1(s),d_1(s)\right) - \frac{\partial f}{\partial x}\left(s,x^{(\xi_2)}(s),u_2(s),d_2(s)\right) \right\|_{i,2} \left\|\Phi^{(\xi_1)}(s,\tau)\right\|_{i,2} \right) ds. 
  \end{multline}
  By applying Condition \ref{corollary:f_bounded} in Corollary \ref{corollary:fns_bounded}, Condition \ref{corollary:dfdxxi_lipschitz} in Corollary \ref{corollary:vf_lipschitz}, Corollary \ref{corollary:stm_bounded}, the same argument as in Equation \eqref{eq:1to2_norm_holder_trick}, and the Bellman-Gronwall Inequality (Lemma 5.6.4 in \cite{Polak1997}), our desired result follows.
\end{proof}

A simple extension of our previous argument shows that for all $t \in [0,1]$, $\D{\phi_t}(\xi;\cdot)$ is Lipschitz continuous with respect to its point of evaluation, $\xi$.
\begin{lemma}
  \label{lemma:Dxt_lipschitz}
  There exists a constant $L > 0$ such that for each $\xi_1, \xi_2 \in {\cal X}_r$, $\xi' \in {\cal X}$, and $t \in [0,1]$:
  \begin{equation}
    \label{eq:Dxt_lipschitz}
    \left\| \D{\phi_t}(\xi_1;\xi') - \D{\phi_t}(\xi_2;\xi') \right\|_2 \leq L \| \xi_1 - \xi_2\|_{\cal X} \left\| \xi' \right\|_{\cal X}
  \end{equation}
  where $\D{\phi_t}$ is as defined in Equation \eqref{eq:dxt_definition}.
\end{lemma}
\begin{proof}
  Let $\xi_1 = (u_1,d_1)$, $\xi_2 = (u_2,d_2)$, and $\xi' = (u',d')$. Then, by applying the Triangular Inequality, and noticing that the induced matrix norm is compatible, observe:
  \begin{multline}
    \left\| \D{\phi_t}(\xi_1;\xi') - \D{\phi_t}(\xi_2;\xi') \right\|_2 \leq
    \int_0^t \Bigg(
      \left\| \Phi^{(\xi_1)}(t,s) - \Phi^{(\xi_2)}(t,s) \right\|_{i,2} 
      \left\| \frac{\partial f}{\partial u}\bigl( s, x^{(\xi_1)}(s), u_1(s), d_1(s) \bigr) \right\|_{i,2} + \\
      + \left\| \Phi^{(\xi_2)}(t,s) \right\|_{i,2} 
      \left\| \frac{\partial f}{\partial u} \bigl( s, x^{(\xi_1)}(s), u_1(s), d_1(s) \bigr) - \frac{\partial f}{\partial u}\bigl( s, x^{(\xi_2)}(s), u_2(s), d_2(s) \bigr) \right\|_{i,2} 
      \Bigg) \left\| u'(s) \right\|_2 ds + \\
    + \int_0^t \sum_{i=1}^q \Bigg(
      \left\| \Phi^{(\xi_1)}(t,s) - \Phi^{(\xi_2)}(t,s) \right\|_{i,2} 
      \left\| f\bigl( s, x^{(\xi_1)}(s), u_1(s), e_i \bigr) \right\|_{2} +  \\
      + \left\| \Phi^{(\xi_2)}(t,s) \right\|_{i,2} 
      \left\| f\bigl( s, x^{(\xi_1)}(s), u_1(s), e_i \bigr) - f\bigl( s, x^{(\xi_2)}(s), u_2(s), e_i \bigr) \right\|_{2} 
      \Bigg) \left\| d'(s) \right\| ds.
  \end{multline}
  By applying Corollary \ref{corollary:stm_bounded}, Condition \ref{corollary:f_bounded} in Corollary \ref{corollary:fns_bounded}, Lemma \ref{lemma:stm_lipschitz}, Conditions \ref{corollary:fxi_lipschitz} and \ref{corollary:dfduxi_lipschitz} in Corollary \ref{corollary:vf_lipschitz}, together with the boundedness of $u'(s)$ and $d'(s)$, and an argument identical to the one used in Equation \eqref{eq:1to2_norm_holder_trick}, our desired result follows.
\end{proof}

Next, we prove that $\D{\phi_t}$ is simultaneously continuous with respect to both of its arguments.
\begin{lemma}
  \label{lemma:Dx_cont}
  For each $t \in [0,1]$, $\xi \in {\cal X}_r$, and $\xi' \in {\cal X}$, the map $(\xi,\xi') \mapsto \D{\phi_t}(\xi; \xi')$, as defined in Equation \eqref{eq:dxt_definition}, is continuous.
\end{lemma}
\begin{proof}
  To prove this result, we can employ an argument identical to the one used in the proof of Lemma \ref{lemma:x_sequential_continuity}.
  First, note that $u(t) \in U$ for each $t \in [0,1]$. 
  Second, note that $\Phi^{(\xi)}$, $f$, $\frac{\partial f}{\partial x}$, and $\frac{\partial f}{\partial u}$ are bounded, as shown in Corollary \ref{corollary:stm_bounded} and Condition \ref{corollary:f_bounded} in Corollary \ref{corollary:fns_bounded}.
  Third, recall that $\Phi^{(\xi)}$, $f$, and $\frac{\partial f}{\partial u}$ are Lipschitz continuous, as proven in Lemma \ref{lemma:stm_lipschitz} and Conditions \ref{corollary:fxi_lipschitz} and \ref{corollary:dfduxi_lipschitz} in Corollary \ref{corollary:vf_lipschitz}, respectively.
  Finally, the result follows after using an argument identical to the one used in Equation \eqref{eq:1to2_norm_holder_trick}.
\end{proof}

We can now construct the directional derivative of the cost $J$ and prove it is Lipschitz continuous.
\begin{lemma}
  \label{lemma:DJ_definition}
  Let $\xi \in {\cal X}_r$, $\xi' \in {\cal X}$, and $J$ be as defined in Equation \eqref{eq:cost}. Then the directional derivative of the cost $J$ in the $\xi'$ direction is:
  \begin{equation}
    \label{eq:DJ_definition}
    \D{J}(\xi;\xi') = \frac{\partial h_0}{\partial x}\big( \phi_1(\xi) \big) \D{\phi_1}(\xi;\xi').
  \end{equation}
\end{lemma}
\begin{proof}
  The result follows directly by the Chain Rule and Lemma \ref{lemma:dxt_definition}.
\end{proof}

\begin{corollary}
  \label{corollary:DJ_lipschitz}
  There exists a constant $L > 0$ such that for each $\xi_1,\xi_2 \in {\cal X}_r$ and $\xi' \in {\cal X}$:
  \begin{equation}
    \left| \D{J}(\xi_1;\xi') - \D{J}(\xi_2;\xi') \right| \leq L \left\| \xi_1 - \xi_2 \right\|_{\cal X} \left\| \xi' \right\|_{\cal X},
  \end{equation}
  where $\D{J}$ is as defined in Equation \eqref{eq:DJ_definition}.
\end{corollary}
\begin{proof}
  Notice by the Triangular Inequality and the Cauchy-Schwartz Inequality: 
  \begin{multline}
    \left| \D{J}(\xi_1;\xi') - \D{J}(\xi_2;\xi') \right| 
    \leq \left\| \frac{\partial h_0}{\partial x}\bigl( \phi_1(\xi_1) \bigr) \right\|_2 \left\| \D{\phi_1}(\xi_1;\xi') - \D{\phi_1}(\xi_2;\xi') \right\|_2 + \\ 
    + \left\| \frac{\partial h_0}{\partial x}\bigl( \phi_1(\xi_1) \bigr) - \frac{\partial h_0}{\partial x}\bigl( \phi_1(\xi_2) \bigr) \right\|_2 \left\| \D{\phi_1}(\xi_2;\xi') \right\|_2.
  \end{multline}
  The result then follows by applying Condition \ref{corollary:phi_bounded} in Corollary \ref{corollary:fns_bounded}, Condition \ref{corollary:dphidxxi_lipschitz} in Corollary \ref{corollary:constraints_lipschitz}, Corollary \ref{corollary:dxt_bounded}, and Lemma \ref{lemma:Dxt_lipschitz}.
\end{proof}

Next, we prove that $\D{J}$ is simultaneously continuous with respect to both of its arguments, which is a direct consequence of Lemma \ref{lemma:Dx_cont}.
\begin{corollary}
  \label{lemma:DJ_cont}
  For each $\xi \in {\cal X}_r$ and $\xi' \in {\cal X}$, the map $(\xi,\xi') \mapsto \D{J}(\xi;\xi')$, as defined in Equation \eqref{eq:DJ_definition}, is continuous.
\end{corollary}

Next, we construct the directional derivative of each of the component constraint functions $\psi_{j,t}$ and prove that each of the component constraints is Lipschitz continuous.
\begin{lemma}
  \label{lemma:Dhj_definition}
  Let $\xi \in {\cal X}_r$, $\xi' \in {\cal X}$, and $\psi_{j,t}$ defined as in Equation \eqref{eq:component_constraints}. 
  Then for each $j \in {\cal J}$ and $t \in [0,1]$, the directional derivative of $\psi_{j,t}$, denoted $\D{\psi_{j,t}}$, is given by:
  \begin{equation}
    \label{eq:Dhj_definition}
    \D{\psi_{j,t}}(\xi;\xi') = \frac{\partial h_j}{\partial x}\big( \phi_t(\xi) \big) \D{\phi_t}(\xi;\xi').
  \end{equation}
\end{lemma}
\begin{proof}
   The result follows using the Chain Rule and Lemma \ref{lemma:dxt_definition}.
\end{proof}

\begin{corollary}
  \label{corollary:Dhj_lipschitz}
  There exists a constant $L > 0$ such that for each $\xi_1,\xi_2 \in {\cal X}_r$, $\xi' \in {\cal X}$, and $t \in [0,1]$:
  \begin{equation}
    \left| \D{\psi_{j,t}}(\xi_1;\xi') - \D{\psi_{j,t}}(\xi_2;\xi') \right| 
    \leq L \left\| \xi_1 - \xi_2 \right\|_{\cal X} \left\| \xi' \right\|_{\cal X},
  \end{equation}
  where $\D{\psi_{j,t}}$ is as defined in Equation \eqref{eq:Dhj_definition}.
\end{corollary}
\begin{proof}
  Notice by the Triangular Inequality and the Cauchy Schwartz Inequality: 
  \begin{multline}
    \left| \D{\psi_{j,t}}(\xi_1;\xi') - \D{\psi_{j,t}}(\xi_2;\xi') \right| 
    \leq \left\| \frac{\partial h_j}{\partial x}\bigl( \phi_t(\xi_1) \bigr) \right\|_2 
    \left\| \D{\phi_t}(\xi_1;\xi') - \D{\phi_t}(\xi_2;\xi') \right\|_2 + \\ 
    + \left\| \frac{\partial h_j}{\partial x} \bigl( \phi_t(\xi_1) \bigr) - \frac{\partial h_j}{\partial x} \bigl( \phi_t(\xi_2) \bigr) \right\|_2 
    \left\| \D{\phi_t}(\xi_2;\xi') \right\|_2.
  \end{multline}
  The result then follows by applying Condition \ref{corollary:hj_bounded} in Corollary \ref{corollary:fns_bounded}, Condition \ref{corollary:dhjdxxi_lipschitz} in Corollary \ref{corollary:constraints_lipschitz}, Corollary \ref{corollary:dxt_bounded}, and Lemma \ref{lemma:Dxt_lipschitz}.
\end{proof}

Next, we prove that $\D{\psi_{j,t}}$ is simultaneously continuous with respect to both of its arguments, which follows directly from Lemma \ref{lemma:Dx_cont}:
\begin{corollary}
  \label{lemma:Dhj_cont}
  For each $\xi \in {\cal X}_r$, $\xi' \in {\cal X}$, and $t \in [0,1]$, the map $(\xi,\xi') \mapsto \D{\psi_{j,t}}(\xi;\xi')$, as defined in Equation \eqref{eq:Dhj_definition}, is continuous.
\end{corollary}

Given these results, we can begin describing the properties satisfied by the optimality function:
\begin{lemma}
  \label{lemma:zeta_lipschitz}
  Let $\zeta$ be defined as in Equation \eqref{eq:zeta_definition}.
  Then there exists a constant $L > 0$ such that, for each $\xi_1,\xi_2,\xi' \in {\cal X}_r$,
  \begin{equation}
    \left| \zeta(\xi_1,\xi') - \zeta(\xi_2,\xi') \right| \leq L \left\| \xi_1-\xi_2 \right\|_{\cal X}.
  \end{equation}
\end{lemma}
\begin{proof}
  To prove the result, first notice that for $\{x_i\}_{i \in {\cal I}}, \{y_i\}_{i \in {\cal I}} \subset \R$:
  \begin{equation}
    \Bigl| \max_{i \in {\cal I}} x_i \Bigr| \leq \max_{i \in {\cal I}} \left| x_i \right|,\quad \text{and}\quad
    \max_{i \in {\cal I}} x_i - \max_{i \in {\cal I}} y_i \leq \max_{i \in {\cal I}} \bigl\{ x_i - y_i \bigr\}.
  \end{equation}
  Therefore,
  \begin{equation}
    \label{eq:max_norm_trick}
    \Bigl| \max_{i \in {\cal I}} x_i - \max_{i \in {\cal I}} y_i \Bigr| 
    \leq \max_{i \in {\cal I}} | x_i- y_i |.
  \end{equation}
  Letting $\Psi^+(\xi) = \max\{0,\Psi(\xi)\}$ and $\Psi^-(\xi)=\max\{0,-\Psi(\xi)\}$, observe:
  \begin{equation}
    \zeta(\xi,\xi') =
    \max \left\{ \d{J}(\xi; \xi' - \xi) - \Psi^{+}(\xi), 
      \max_{ j\in{\cal J},\; t\in[0,1] }\D{\psi_{j,t}}(\xi;\xi'-\xi) - \gamma \Psi^{-}(\xi) 
    \right\} + \| \xi' - \xi \|_{\cal X}.
  \end{equation}
  Employing Equation \eqref{eq:max_norm_trick}:
  \begin{multline}
    \bigl| \zeta(\xi_1,\xi') - \zeta(\xi_2,\xi') \bigr| 
    \leq \max \biggl\{ \bigl| \d{J}(\xi_1; \xi' - \xi_1) - \d{J}(\xi_2; \xi' - \xi_2) \bigr| 
    + \bigl| \Psi^{+}(\xi_2) - \Psi^{+}(\xi_1) \bigr|, \\ 
    \max_{ j\in{\cal J},\; t\in[0,1] } \bigl| \D{\psi_{j,t}}(\xi_1;\xi'-\xi_1) - \D{\psi_{j,t}}(\xi_2;\xi'-\xi_2) \bigr| 
    + \gamma \bigl| \Psi^{-}(\xi_2) - \Psi^{-}(\xi_1) \bigr| \biggr\} 
    + \bigl| \| \xi' - \xi_1 \|_{\cal X} - \| \xi' - \xi_2 \|_{\cal X} \bigr|.
  \end{multline}
  
  We show three results that taken together with the Triangular Inequality prove the desired result. First, by applying the reverse triangle inequality:
  \begin{equation}
    \bigl| \| \xi' - \xi_1 \|_{\cal X} - \| \xi' - \xi_2 \|_{\cal X} \bigr| \leq \| \xi_1 - \xi_2 \|_{\cal X}.
  \end{equation}
  Second, 
  \begin{equation}
    \label{eq:DJ_triangle_lipschitz}
    \begin{aligned}
      \bigl| \D{J}(\xi_1;\xi'-\xi_1) - \D{J}(\xi_2;\xi'-\xi_2) \bigr| 
      &= \bigl| \D{J}(\xi_1;\xi'-\xi_1) - \D{J}(\xi_2;\xi'-\xi_1) + \D{J}(\xi_2;\xi_2-\xi_1) \bigr| \\
      &\leq \bigl| \D{J}(\xi_1;\xi') - \D{J}(\xi_2;\xi') \bigr| + \bigl| \D{J}(\xi_1;\xi_1) - \D{J}(\xi_2;\xi_1) \bigr| + \\ 
      &\qquad + \left| \frac{\partial h_0}{\partial x}\bigl( \phi_1(\xi_2) \bigr) \D{\phi_1}(\xi_2;\xi_2-\xi_1) \right| \\
      &\leq L \left\| \xi_1 - \xi_2 \right\|_{\cal X},
    \end{aligned}
  \end{equation}
  where $L > 0$ and we employed the linearity of $\D{J}$, Corollary \ref{corollary:DJ_lipschitz}, the fact that $\xi'$ and $\xi_1$ are bounded since $\xi',\xi_1 \in {\cal X}_r$, the Cauchy-Schwartz Inequality, Condition \ref{corollary:phi_bounded} in Corollary \ref{corollary:fns_bounded}, and Corollary \ref{corollary:dxt_bounded}. 
  Notice that by employing an argument identical to Equation \eqref{eq:DJ_triangle_lipschitz} and Corollary \ref{corollary:Dhj_lipschitz}, we can assume without loss of generality that $\bigl| \D{\psi_{j,t}}(\xi_1;\xi'-\xi_1) - \D{\psi_{j,t}}(\xi_2;\xi'-\xi_2) \bigr| \leq L \left\| \xi_1 - \xi_2 \right\|_{\cal X}$. 
  Finally, notice that by applying Lemma \ref{lemma:psi_continuous}, $\Psi^{+}(\xi)$ and $\Psi^{-}(\xi)$ are Lipschitz continuous.
\end{proof}

In fact, $\zeta$ satisfies an even more important property:
\begin{lemma}
  \label{lemma:zeta_strictly_convex}
  For each $\xi \in {\cal X}_p$, the map $\xi' \mapsto \zeta(\xi,\xi')$, as defined in Equation \eqref{eq:zeta_definition}, is strictly convex.
\end{lemma}
\begin{proof}
  The proof follows after noting that the maps $\xi' \mapsto \d{J}(\xi;\xi'-\xi)$ and 
  $\xi' \mapsto \frac{\partial h_j}{\partial x}\bigl( \phi_t(\xi) \bigr) \d{\phi_t}(\xi;\xi'-\xi)$ are affine, hence any maximum among these function is convex, and the map $\xi' \mapsto \| \xi' - \xi \|_{\cal X}$ is strictly convex since we chose the $2$--norm as our finite dimensional norm.
\end{proof}

The following theorem, which follows as a result of the previous lemma, is fundamental to our result since it shows that $g$, as defined in Equation \eqref{eq:theta_definition}, is a well-defined function.
We omit the proof since it is a particular case of a well known result regarding the existence of unique minimizers of strictly convex functions over bounded sets in Hilbert spaces (Proposition II.1.2 in \cite{Ekeland1976}).
\begin{theorem}
  \label{thm:zeta_unique_minimizer}
  For each $\xi \in {\cal X}_p$, the map $\xi' \mapsto \zeta(\xi,\xi')$, as defined in Equation \eqref{eq:zeta_definition}, has a unique minimizer.
\end{theorem}

Employing these results we can prove the continuity of the optimality function. This result is not strictly required in order to prove the convergence of Algorithm \ref{algo:main_algo} or in order to prove that the optimality function encodes local minimizers of the Switched System Optimal Control Problem, but is useful when we describe the implementation of our algorithm.


\begin{lemma}
  \label{lemma:theta_cont}
  The function $\theta$, as defined in Equation \eqref{eq:theta_definition}, is continuous.
\end{lemma}
\begin{proof}
  First, we show that $\theta$ is upper semi-continuous. 
  Consider a sequence $\{\xi_i\}_{i=1}^{\infty} \subset {\cal X}_r$ converging to $\xi$, and $\xi' \in {\cal X}_r$ such that $\theta(\xi) = \zeta(\xi,\xi')$, i.e. $\xi' = g(\xi)$, where $g$ is defined as in Equation \eqref{eq:theta_definition}. 
  Since $\theta(\xi_i) \leq \zeta(\xi_i,\xi')$ for all $i \in \N$, 
  \begin{equation}
    \limsup_{i\to\infty}\theta(\xi_i)\leq\limsup_{i\to\infty}\zeta(\xi_i,\xi')= \zeta(\xi,\xi') = \theta(\xi),
  \end{equation}
  which proves the upper semi-continuity of $\theta$.
  
  Second, we show that $\theta$ is lower semi-continuous. 
  Let $\{\xi'_i\}_{i\in\N}$ such that $\theta(\xi_i) = \zeta(\xi_i,\xi'_i)$, i.e. $\xi'_i = g(\xi_i)$. 
  From Lemma \ref{lemma:zeta_lipschitz}, we know there exists a Lipschitz constant $L > 0$ such that for each $i \in \N$, 
  $\left| \zeta(\xi,\xi'_i) - \zeta(\xi_i,\xi'_i)\right| \leq L \left\|\xi - \xi_i \right\|_{\cal X}$. 
  Consequently,
  \begin{equation}
    \theta(\xi) 
    \leq \bigl( \zeta(\xi,\xi'_i) - \zeta(\xi_i,\xi'_i) \bigr) + \zeta(\xi_i,\xi'_i) 
    \leq L \| \xi - \xi_i \|_{\cal X} + \theta(\xi_i).
  \end{equation}
  Taking limits we conclude that
  \begin{equation}
    \theta(\xi) \leq \liminf_{i\to\infty}\theta(\xi_i),
  \end{equation}
  which proves the lower semi-continuity of $\theta$, and our desired result.
\end{proof}

Finally, we can prove that $\theta$ encodes a necessary condition for optimality:
\begin{theorem}
  \label{thm:theta_optimality_function}
  Let $\theta$ be as defined in Equation \eqref{eq:theta_definition}, then:
  \begin{enumerate_parentesis}
  \item \label{thm:theta_negative} $\theta$ is non-positive valued, and
  \item \label{thm:local_min_theta_zero} If $\xi \in {\cal X}_p$ is a local minimizer of the Switched System Optimal Control Problem as in Definition \ref{definition:SSOCP_minimizers}, then $\theta(\xi) = 0$.
  \end{enumerate_parentesis}
\end{theorem}
\begin{proof}
  Notice that $\zeta(\xi,\xi) = 0$, therefore $\theta(\xi) = \min_{\xi' \in {\cal X}_r} \zeta(\xi,\xi') \leq \zeta(\xi,\xi) = 0$. 
  This proves Condition \ref{thm:theta_negative}.
  
  To prove Condition \ref{thm:local_min_theta_zero}, we begin by making several observations. Given $\xi' \in {\cal X}_r$ and $\lambda \in [0,1]$, using the Mean Value Theorem and Corollary \ref{corollary:DJ_lipschitz} we have that there exists $s \in (0,1)$ and $L > 0$ such that
  \begin{equation}
    \label{eq:theta_J_mvt}
    \begin{aligned}
      J\big( \xi + \lambda (\xi' - \xi) \big) - J(\xi) 
      &= \d{J}\bigl( \xi + s \lambda ( \xi' - \xi ); \lambda (\xi' - \xi) \bigr) \\
      &\leq \lambda \d{J}\bigl( \xi; \xi' - \xi \bigr) + L \lambda^2 \| \xi' - \xi \|^2_{\cal X}.
    \end{aligned}
  \end{equation}
  Letting ${\cal A}(\xi) = \left\{ (j,t) \in {\cal J} \times [0,1] \mid \Psi(\xi) = h_j\big( x^{(\xi)}(t) \big) \right\}$, similar to the equation above, there exists a pair $(j,t) \in {\cal A}\bigl( \xi + \lambda (\xi' - \xi) \bigr)$ and $s \in (0,1)$ such that, using Corollary \ref{corollary:Dhj_lipschitz},
  \begin{equation}
    \label{eq:theta_psi_mvt}
    \begin{aligned}
      \Psi\big( \xi + \lambda (\xi' - \xi) \big) - \Psi(\xi) 
      &\leq \psi_{j,t}\bigl( \xi + \lambda ( \xi' - \xi ) \bigr) - \Psi(\xi) \\
      &\leq \psi_{j,t}\bigl( \xi + \lambda ( \xi' - \xi ) \bigr) - \psi_{j,t}(\xi) \\
      &= \d{\psi_{j,t}}\bigl( \xi + s \lambda ( \xi' - \xi ); \lambda ( \xi' - \xi ) \bigr) \\
      &\leq \lambda \d{\psi_{j,t}}\bigl( \xi; \xi' - \xi \bigr) + L \lambda^2 \| \xi' - \xi \|^2_{\cal X}.
    \end{aligned}
  \end{equation}
  Finally, letting $L$ denote the Lipschitz constant as in Condition \ref{assump:phi_lipschitz} in Assumption \ref{assump:constraint_fns}, notice:
  \begin{equation}
    \label{eq:theta_psi_lipschitz}
    \begin{aligned}
      \Psi\big( \xi + \lambda( \xi' - \xi ) \big) - \Psi( \xi ) 
      &= \max_{(j,t) \in {\cal J} \times [0,1]} \psi_{j,t}\big( \xi + \lambda( \xi' - \xi ) \big) - \max_{(j,t) \in {\cal J} \times [0,1]} \psi_{j,t}( \xi ) \\
      &\leq  \max_{(j,t) \in {\cal J} \times [0,1]} \psi_{j,t}\big( \xi + \lambda( \xi' - \xi ) \big) - \psi_{j,t}(\xi) \\
      &\leq L \max_{t \in [0,1]} \left\|\phi_t(\xi + \lambda( \xi' - \xi )) - \phi_t(\xi) \right\|_2.
    \end{aligned}
  \end{equation}
  
  We prove Condition \ref{thm:local_min_theta_zero} by contradiction. 
  That is, using Definition \ref{definition:SSOCP_minimizers}, we assume that $\theta(\xi) < 0$ and show that for each $\epsilon > 0$ there exists $\hat{\xi} \in {\cal N}_w(\xi,\epsilon) \cap \bigl\{ \bar{\xi} \in {\cal X}_p \mid \Psi(\bar{\xi}) \leq 0 \bigr\}$ such that $J(\hat{\xi}) < J(\xi)$, where ${\cal N}_w(\xi,\epsilon)$ is as defined in Equation \eqref{eq:nbhd_weak_topology}, hence arriving at a contradiction.

  Before arriving at this contradiction, we make three initial observations. 
  First, notice that since $\xi \in {\cal X}_p$ is a local minimizer of the Switched System Optimal Control Problem, $\Psi(\xi) \leq 0$. 
  Second, consider $g$ as defined in Equation \eqref{eq:theta_definition}, which exists by Theorem \ref{thm:zeta_unique_minimizer}, and notice that since $\theta(\xi) < 0$, $g(\xi) \neq \xi$. 
  Third, notice that, as a result of Theorem \ref{thm:bangbang_result}, for each $\left(\xi + \lambda (g(\xi) - \xi)\right) \in {\cal X}_r$ and $\epsilon' > 0$ there exists a $\xi_\lambda \in {\cal X}_p$ such that 
  \begin{equation}
    \label{eq:theta_two_x_nbhd}
    \bigl\| x^{(\xi_\lambda)} - x^{(\xi + \lambda (g(\xi) - \xi))} \bigr\|_{L^\infty} < \epsilon'
  \end{equation}
  where $x^{(\xi)}$ is the solution to Differential Equation \eqref{eq:traj_xi}.

  Now, letting $\epsilon' = -\frac{\lambda \theta(\xi)}{2 L} > 0$ and using Corollary \ref{corollary:x_lipschitz}:
  \begin{equation}
    \label{eq:theta_weak_nbhd}
    \begin{aligned}
      \bigl\lVert x^{(\xi_\lambda)} - x^{(\xi)} \bigr\rVert_{L^2}
      &\leq \bigl\lVert x^{(\xi_\lambda)} - x^{(\xi + \lambda ( g(\xi) - \xi ))} \bigr\rVert_{L^2} 
      + \bigl\lVert x^{(\xi + \lambda ( g(\xi) - \xi ))} - x^{(\xi)} \bigr\rVert_{L^2} \\
      &\leq \left( - \frac{\theta(\xi)}{2L} + L \left\| g(\xi) - \xi \right\|_{\cal X} \right) \lambda.
    \end{aligned}
  \end{equation}

  Next, observe that:
  \begin{equation}
    \label{eq:theta_lm}
    \theta(\xi) = \max \left\{ \d{J}( \xi; g(\xi) - \xi ), 
      \max_{(j,t) \in {\cal J} \times [0,1]} \D{\psi_{j,t}}( \xi; g(\xi) - \xi ) + \gamma \Psi(\xi) \right\} + \left\|g(\xi) - \xi \right\|_{\cal X} < 0.
  \end{equation}
  Also, by Equations \eqref{eq:theta_J_mvt}, \eqref{eq:theta_two_x_nbhd}, and \eqref{eq:theta_lm}, together with Condition \ref{assump:phi_lipschitz} in Assumption \ref{assump:constraint_fns} and Corollary \ref{corollary:x_lipschitz}:
  \begin{equation}
    \begin{aligned}
      J( \xi_\lambda ) - J( \xi ) 
      &\leq J(\xi_\lambda) - J\bigl( \xi + \lambda ( g(\xi) - \xi ) \bigr) + J\bigl( \xi + \lambda ( g(\xi) - \xi ) \bigr) - J(\xi) \\
      &\leq L \left\| \phi_1(\xi_\lambda) - \phi_1 ( \xi + \lambda ( g(\xi) - \xi )) \right\|_2 + \theta(\xi)\lambda + 4 A^2 L \lambda^{2} \\
      &\leq L \epsilon' +  \theta(\xi)\lambda + 4 A^2 L \lambda^{2} \\
      &\leq \frac{\theta(\xi)\lambda}{2} + 4 A^2 L \lambda^{2},
    \end{aligned}
  \end{equation}
  where $A = \max\big\{ \|u\|_2 + 1 \mid u \in U \big\}$ and we used the fact that $\| \xi - \xi' \|_{\cal X}^2 \leq 4A^2$ and $\D{J}(\xi;\xi'-\xi) \leq \theta(\xi)$. 
  Hence for each $\lambda \in \left( 0, \frac{-\theta(\xi)}{8A^2L} \right)$, 
  \begin{equation}
    \label{eq:theta_Jless}
    J(\xi_\lambda) - J(\xi) < 0.
  \end{equation}

  Similarly, using Condition \ref{assump:phi_lipschitz} in Assumption \ref{assump:constraint_fns}, together with Equations \eqref{eq:theta_psi_mvt}, \eqref{eq:theta_psi_lipschitz}, and \eqref{eq:theta_lm}, we have:
  \begin{equation}
    \begin{aligned}
      \Psi( \xi_\lambda ) 
      &\leq \Psi( \xi_\lambda ) - \Psi\bigl( \xi + \lambda ( g(\xi) - \xi ) \bigr) + \Psi\bigl( \xi + \lambda ( g(\xi) - \xi ) \bigr) \\ 
      &\leq L \max_{t \in [0,1]} \left\| \phi_t( \xi_\lambda ) - \phi_t ( \xi + \lambda ( g(\xi) - \xi )) \right\|_2 + \Psi(\xi) + \bigl( \theta(\xi) - \gamma \Psi(\xi) \bigr) \lambda + 4 A^2 L \lambda^{2} \\
      &\leq L \epsilon' + \theta(\xi) \lambda + 4 A^2 L \lambda^{2} + ( 1 - \gamma \lambda) \Psi(\xi) \\
      &\leq \frac{\theta(\xi)\lambda}{2} + 4 A^2 L \lambda^{2} + ( 1 - \gamma \lambda) \Psi(\xi),
    \end{aligned}
  \end{equation}
  where $A = \max\big\{ \|u\|_2 + 1 \mid u \in U \big\}$ and we used the fact that $\| \xi - \xi' \|_{\cal X}^2 \leq 4A^2$ and $\D{\psi_{j,t}}(\xi;\xi'-\xi) \leq \theta(\xi) - \gamma \Psi(\xi)$ for each $(j,t) \in {\cal J} \times [0,1]$. 
  Hence for each $\lambda \in \left( 0, \min \left\{ \frac{-\theta(\xi)}{8A^2L}, \frac{1}{\gamma} \right\} \right)$:
  \begin{equation}
    \label{eq:theta_psisatisfied}
    \Psi(\xi_\lambda) \leq ( 1 - \gamma \lambda) \Psi(\xi) \leq 0.
  \end{equation}

  Summarizing, suppose $\xi \in {\cal X}_p$ is a local minimizer of the Switched System Optimal Control Problem and $\theta(\xi) < 0$. 
  For each $\epsilon > 0$, by choosing any 
  \begin{equation}
    \lambda \in \left( 0, \min\left\{\frac{-\theta(\xi)}{8A^2L}, \frac{1}{\gamma}, 
        \frac{2L\epsilon}{2 L^2 \left\|g(\xi) - \xi \right\|_{\cal X} - \theta(\xi)} \right\} \right), 
  \end{equation}
  we can construct a $\xi_\lambda \in {\cal X}_p$ such that $\xi_\lambda \in { \cal N }_w(\xi,\epsilon)$, by Equation \eqref{eq:theta_weak_nbhd}, such that $J(\xi_\lambda) < J(\xi)$, by Equation \eqref{eq:theta_Jless}, and $\Psi(\xi_\lambda) \leq 0$, by Equation \eqref{eq:theta_psisatisfied}. 
  Therefore, $\xi$ is not a local minimizer of the Switched System Optimal Control Problem, which is a contradiction and proves Condition \ref{thm:local_min_theta_zero}.
\end{proof}

\subsection{Approximating Relaxed Inputs}

In this subsection, we prove that the projection operation, $\rho_N$, allows us to control the quality of approximation between the trajectories generated by a relaxed discrete input and its projection. First, we prove for $d \in {\cal D}_r$, ${\cal F}_N(d) \in {\cal D}_r$ and ${\cal P}_N \big( {\cal F}_N(d) \big) \in {\cal D}_p$: 
\begin{lemma}
	\label{lemma:proj_reasonable}
	Let $d \in {\cal D}_r$, ${\cal F}_N$ be as defined in Equation \eqref{eq:wavelet_approx}, and ${\cal P}_N$ be as defined in Equation \eqref{eq:pwm}. Then for each $N \in \N$ and $t \in [0,1]$:
	\begin{enumerate_parentesis}
		\item \label{lemma:pr_1} $[ {\cal F}_N(d) ]_i(t) \in [0,1]$,
		\item \label{lemma:pr_2} $\sum_{i=1}^q [ {\cal F}_N(d) ]_i(t) = 1$,
		\item \label{lemma:pr_3} $\big[ {\cal P}_N \big( {\cal F}_N(d) \big) \big]_i(t) \in \{0,1\}$, 
		\item \label{lemma:pr_4} $\sum_{i=1}^q \big[ {\cal P}_N \big( {\cal F}_N(d) \big) \big]_i(t) = 1$.
	\end{enumerate_parentesis}
\end{lemma}
\begin{proof}
  Condition \ref{lemma:pr_1} follows due to the result in Section 3.3 in \cite{Haar1910}. 
  Condition \ref{lemma:pr_2} follows since the wavelet approximation is linear, thus,
  \begin{equation}
    \begin{aligned}
      \sum_{i=1}^q [ {\cal F}_N(d) ]_i &= \sum_{i=1}^q \left( \langle d_i, \mathds{1} \rangle + \sum_{k=0}^N \sum_{j=0}^{2^k-1} \langle d_{i}, b_{kj} \rangle \frac{ b_{kj} }{\| b_{kj} \|_{L^2}^2} \right) \\
      &= \langle \mathds{1}, \mathds{1} \rangle + \sum_{k=0}^N \sum_{j=0}^{2^k-1} \langle \mathds{1}, b_{kj} \rangle \frac{ b_{kj} }{\| b_{kj} \|_{L^2}^2} = \mathds{1},
    \end{aligned}
  \end{equation}
  where the last equality holds since $\langle \mathds{1}, b_{kj} \rangle = 0$ for each $k,j$.

  Conditions \ref{lemma:pr_3} and \ref{lemma:pr_4} are direct consequences of the definition of ${\cal P}_N$, since ${\cal P}_N$ can only take the values $0$ or $1$, and only one coordinate is equal to $1$ at any given time $t \in [0,1]$.
\end{proof}

Recall that in order to avoid the introduction of additional notation, we let the coordinate-wise application of ${\cal F}_N$ to some relaxed discrete input $d \in {\cal D}_r$ be denoted as ${\cal F}_N(d)$ and similarly for some continuous input $u \in {\cal U}$, but in fact ${\cal F}_N$ as originally defined took $L^2( [0,1], \R ) \cap BV( [0,1], \R )$ to $L^2( [0,1], \R ) \cap BV( [0,1], \R )$. Next, we prove that the wavelet approximation allows us to control the quality of approximation:
\begin{lemma}
  \label{lemma:wavelet_approximation}
  Let $f \in L^2( [0,1], \R ) \cap BV( [0,1], \R )$, then 
  \begin{equation}
    \left\| f - {\cal F}_N(f) \right\|_{L^2} \leq \frac{1}{2} \left( \frac{1}{\sqrt{2}} \right)^N \| f \|_{BV},
  \end{equation}
where ${\cal F}_N$ is as defined in Equation \eqref{eq:wavelet_approx}.
\end{lemma}
\begin{proof}
  Since $L^2$ is a Hilbert space and the collection $\{ b_{kj} \}_{k,j}$ is a basis, then
  \begin{equation}
    f = \langle f, \mathds{1} \rangle + 
    \sum_{k=0}^\infty \sum_{j=0}^{2^k-1} \langle f, b_{kj} \rangle \frac{ b_{kj} }{\| b_{kj} \|_{L^2}^2}.
  \end{equation}

  Note that $\| b_{kj} \|_{L^2}^2 = 2^{-k}$ and that
  \begin{equation}
    v_{kj}(t) = \int_0^t b_{kj}(s) ds =
    \begin{cases}
      t - j2^{-k} & \text{if}\ t \in \left[ j2^{-k}, \left( j + \frac{1}{2} \right)2^{-k} \right), \\
      - t + (j+1)2^{-k} & \text{if}\ t \in \left[ \left( j + \frac{1}{2} \right)2^{-k}, \left( j + 1 \right)2^{-k} \right), \\
      0 & \text{otherwise}, \\
    \end{cases}
  \end{equation}
  thus $\| v_{kj} \|_{L^\infty} = 2^{-k-1}$.
  Now, using integration by parts, and since $f \in BV([0,1],\R)$,
  \begin{equation}
    \left| \langle f, b_{kj} \rangle \right| 
    = \left| \int_{j2^{-k}}^{(j+1)2^{-k}} \dot{f}(t) v_{kj}(t) dt \right| 
    \leq 2^{-k-1} \int_{j2^{-k}}^{(j+1)2^{-k}} \big| \dot{f}(t) \big| dt
  \end{equation}
  
  Finally, Parseval's Identity for Hilbert spaces (Theorem 5.27 in \cite{Folland1999}) implies that
  \begin{equation}
    \label{eq:weak_conv_1}
	  \begin{aligned}
	    \big\| f - {\cal F}_N(f) \big\|_{L^2}^2 
      &= \sum_{k=N+1}^\infty \sum_{j=0}^{2^k-1} \frac{ | \langle f, b_{kj} \rangle |^2 }{ \| b_{kj} \|_{L^2}^2 } \\
	    &\leq \sum_{k = N+1}^\infty 2^{-k-2} \sum_{j=0}^{2^k-1} \left( \int_{j2^{-k}}^{(j+1)2^{-k}} \big| \dot{f}(t) \big| dt \right)^2 \\
	    &\leq 2^{-N-2} \| f \|_{BV}^2,
	  \end{aligned}
	\end{equation}
  as desired.
\end{proof}

The following lemma is fundamental to find a rate of convergence for the approximation of the solution of differential equations using relaxed inputs:
\begin{lemma}
  \label{lemma:weak_convergence}
  There exists $K > 0$ such that for each $d \in {\cal D}_r$ and $f \in L^2([0,1],\R^q) \cap BV([0,1], \R^q )$,
  \begin{equation}
    \left| \left\langle d - {\cal P}_N \big( {\cal F}_N(d) \big), f \right\rangle \right| \leq 
    K \left( \left( \frac{1}{\sqrt{2}} \right)^N \| f \|_{L^2} \| d \|_{BV} + \left( \frac{1}{2} \right)^N \| f \|_{BV} \right),
  \end{equation}
where ${\cal F}_N$ is as defined Equation \eqref{eq:wavelet_approx} and ${\cal P}_N$ is as defined in Equation \eqref{eq:pwm}.
\end{lemma}
\begin{proof}
  To simplify our notation, let $t_k = \frac{k}{2^N}$, $p_{ik} = [{\cal F}_N(d)]_i( t_k )$, $S_{ik} = \sum_{j=1}^i p_{jk}$, and 
  \begin{equation}
    A_{ik} = \left[ t_k + \frac{1}{2^N} S_{(i-1)k}, t_k + \frac{1}{2^N} S_{ik} \right).   
  \end{equation}
  Also let us denote the indicator function of the set $A_{ik}$ by $\mathds{1}_{A_{ik}}$. Consider
  \begin{equation}
    \left\langle {\cal F}_N(d) - {\cal P}_N \big( {\cal F}_N(d) \big), f \right\rangle =
    \sum_{k=0}^{2^N-1} \sum_{i=1}^q \int_{t_k}^{t_{k+1}} \big( p_{ik} - \mathds{1}_{A_{ik}}(t) \big) f_i(t) dt.
  \end{equation}
  Let $w_{ik}: [0,1] \to \R$ be defined by
  \begin{equation}
    w_{ik}(t) = \int_{t_k}^t p_{ik} - \mathds{1}_{A_{ik}}(s) ds =
    \begin{cases}
      p_{ik} ( t - t_k )
        & \text{if}\ t \in \left[ t_k, t_k + \frac{1}{2^N} S_{(i-1)k} \right), \\
      \frac{1}{2^N} p_{ik} S_{(i-1)k} + ( p_{ik} - 1 ) \left( t - t_k - \frac{1}{2^N} S_{(i-1)k} \right) 
        & \text{if}\ t \in A_{ik}, \\
      \frac{1}{2^N} p_{ik} (S_{ik} - 1) + p_{ik} \left( t - t_k - \frac{1}{2^N} S_{ik} \right) 
        & \text{if}\ t \in \left[ t_k + \frac{1}{2^N} S_{ik}, t_{k+1} \right),
    \end{cases}
  \end{equation}
  when $t \in \left[ t_k, t_{k+1} \right]$, and $w_{ik}(t) = 0$ otherwise. Note that $\| w_{ik} \|_{L^\infty} \leq \frac{p_{ik}}{2^N}$. 
  Thus, using integration by parts,
  \begin{equation}
    \left| \int_{t_k}^{t_{k+1}} \big( p_{ik} - \mathds{1}_{A_{ik}}(t) \big) f_i(t) dt \right| = 
    \left| \int_{t_k}^{t_{k+1}} w(t) \dot{f}_i(t) dt \right| \leq \frac{p_{ik}}{2^N} \int_{t_k}^{t_{k+1}} \big| \dot{f}_i(t) \big| dt,
  \end{equation}
  and
  \begin{equation}
    \label{eq:weak_conv_2}
    \begin{aligned}
      \left|\left\langle {\cal F}_N(d) - {\cal P}_N \big( {\cal F}_N(d) \big), f \right\rangle \right|
      &\leq \frac{1}{2^N} \sum_{k=0}^{2^N-1} \int_{t_k}^{t_{k+1}} \sum_{i=1}^q p_{ik} \big| \dot{f}_i(t) \big| dt \\
      &\leq \frac{1}{2^N} \| f \|_{BV}.
    \end{aligned}
  \end{equation}
  where the last inequality follows by H\"older's Inequality.
  
  Also, by Lemma \ref{lemma:wavelet_approximation} we have that
  \begin{equation}
    \big\| d_i - \left[ {\cal F}(d) \right]_i \big\|_{L^2} \leq \frac{1}{2} \left( \frac{1}{\sqrt{2}} \right)^N \| d_i \|_{BV}.
  \end{equation}
  Hence, using Cauchy-Schwartz's Inequality,
  \begin{equation}
    \left|\left\langle d - {\cal P}_N \big( {\cal F}_N(d) \big), f \right\rangle \right|
    \leq \| d - {\cal F}_N(d) \|_{L^2} \| f \|_{L^2} + 
    \left|\left\langle {\cal F}_N(d) - {\cal P}_N \big( {\cal F}_N(d) \big), f \right\rangle\right|,
  \end{equation}
  and the desired result follows from Equations \eqref{eq:weak_conv_1}, \eqref{eq:weak_conv_2}.
\end{proof}

Note that Lemma \ref{lemma:weak_convergence} does not prove convergence of ${\cal P}_N\big( {\cal F}_N(d) \big)$ to $d$ in the weak topology on ${\cal D}_r$. Such a result is indeed true, i.e. ${\cal P}_N\big( {\cal F}_N(d) \big)$ does converge in the weak topology to $d$, and it can be shown using an argument similar to the one used in Lemma 1 in \cite{Sussmann1972}. The reason we chose to prove a weaker result is because in this case we get an explicit rate of convergence, which is fundamental to the construction of our optimization algorithm because it allows us to bound the quality of approximation of the state trajectory.
\begin{theorem}
	\label{thm:quality_of_approximation}
  Let $\rho_N$ be defined as in Equation \eqref{eq:rho} and $\phi_t$ be defined as in Equation \eqref{eq:flow_xi}.
	Then there exists $K > 0$ such that for each $\xi = (u,d) \in {\cal X}_r$ and for each $t \in [0,1]$,
	\begin{equation}
    \big\| \phi_t\bigl( \rho_N(\xi) \bigr) - \phi_t(\xi) \big\|_2 
    \leq K \left( \frac{1}{\sqrt{2}} \right)^N \big( \|\xi\|_{BV} + 1 \big).
	\end{equation}
\end{theorem}
\begin{proof}
  To simplify our notation, let us denote $u_N = {\cal F}_N(u)$ and $d_N = {\cal P}_N \big( {\cal F}_N(d) \big)$, thus $\rho_N(\xi) = (u_N,d_N)$. Consider
  \begin{equation}
    \bigl\| x^{(u_N,d_N)}(t) - x^{(u,d)}(t) \bigr\|_2 \leq
    \bigl\| x^{(u_N,d_N)}(t) - x^{(u,d_N)}(t) \bigr\|_2 + \bigl\| x^{(u,d_N)}(t) - x^{(u,d)}(t) \bigr\|_2.
  \end{equation}
  The main result of the theorem will follow from upper bounds from each of these two parts.

  Note that
  \begin{equation}
    \begin{aligned}
      \bigl\| x^{(u_N,d_N)}(t) - x^{(u,d_N)}(t) \bigr\|_2 
      &\leq \int_0^1 \bigl\| f\bigl( s, x^{(u_N,d_N)}(s), u_N(s), d_N(s) \bigr) - f\bigl( s, x^{(u,d_N)}(s), u(s), d_N(s) \bigr) \bigr\|_2 ds \\
      &\leq L \int_0^1 \bigl\| x^{(u_N,d_N)}(s) - x^{(u,d_N)}(s) \bigr\|_2 + \bigl\| u_N(s) - u(s) \bigr\|_2 ds,
    \end{aligned}
  \end{equation}
  thus, using Bellman-Gronwall's Inequality (Lemma 5.6.4 in \cite{Polak1997}) together with the result in Lemma \ref{lemma:wavelet_approximation} we get
  \begin{equation}
    \label{eq:qualofapprox_ineq_1}
    \bigl\| x^{(u_N,d_N)}(t) - x^{(u,d_N)}(t) \bigr\|_2 \leq \frac{L e^L \sqrt{2}}{2} \left( \frac{1}{\sqrt{2}} \right)^N \|u\|_{BV}
  \end{equation}

	On the other hand,
  \begin{multline}
    x^{(u,d_N)}(t) - x^{(u,d)}(t)
    = \int_0^t \sum_{i=1}^{q} \bigl( [d_N]_i(s) - d_i(s) \bigr) f\big( s, x^{(u,d)}(s), u(s), e_i \big) ds + \\
    + \int_0^t \sum_{i=1}^{q} [d_N]_i(s) \left( f\big( s, x^{(u,d_N)}(s), u(s), e_i \big) - f\big( s, x^{(u,d)}(s), u(s), e_i \big) \right) ds, 
  \end{multline}
	thus,
	\begin{multline}
    \bigl\| x^{(u,d_N)}(t) - x^{(u,d)}(t) \bigr\|_2 
    \leq \left\| \int_0^1 \sum_{i=1}^q \big( [d_N]_i(s) - d_i(s) \big) f\big( s, x^{(u,d)}(s), u(s), e_i \big) ds \right\|_2 + \\
    + L \int_0^1 \bigl\| x^{(u,d_N)}(s) - x^{(u,d)}(s) \bigr\|_2 ds.
	\end{multline}
	Using Bellman-Gronwall's inequality we get
	\begin{equation}
    \big\| x^{(u,d_N)}(t) - x^{(u,d)}(t) \big\|_2 
    \leq e^{L} \left\| \int_0^1 \sum_{i=1}^q \big( [d_N]_i(s) - d_i(s) \big) f\big( t, x^{(u,d)}(s), u(s), e_i \big) ds \right\|_2. 
	\end{equation}
  Recall that $f$ maps to $\R^n$, so let us denote the $k$--th coordinate of $f$ by $f_k$. 
  Let $v_{ki}(t) = f_k\big( t, x^{(u,d)}(t), u(t), e_i \big)$ and $v_k = \left( v_{k1}, \ldots, v_{kq} \right)$, then $v_k$ is of bounded variation.
  Indeed, by Theorem \ref{thm:bounded_variation} and Condition \ref{corollary:f_bounded} in Corollary \ref{corollary:fns_bounded}, we have that $\bigl\| x^{(\xi)} \bigr\|_{BV} \leq C$.
  Thus, by Condition \ref{assump:f_lipschitz} in Assumption \ref{assump:fns_continuity} and again using Theorem \ref{thm:bounded_variation}, we get that, for each $i \in {\cal Q}$,
  \begin{equation}
    \left\| v_{ki} \right\|_{BV} \leq L \big( 1 + C + \| u \|_{BV} \big).
  \end{equation}
  Moreover, Condition \ref{corollary:f_bounded} in Corollary \ref{corollary:fns_bounded} directly imply that $\| v_{ki} \|_{L^2} \leq C$. 
  Hence, Lemma \ref{lemma:weak_convergence} implies that there exists $K > 0$ such that
  \begin{equation}
    \label{eq:qualofapprox_ineq_2}
    \left| \left\langle d - d_N, v_k \right\rangle \right| \leq K \left( \left( \frac{1}{\sqrt{2}} \right)^N C \|d\|_{BV} + q \left( \frac{1}{2} \right)^N \left( 1 + C + \|u\|_{BV} \right) \right).
  \end{equation}
  Since Equation \eqref{eq:qualofapprox_ineq_2} is satisfied for each $k \in \{1,\ldots,n\}$, then after ordering the constants and noting that $2^N \geq 2^{\frac{N}{2}}$ for each $N \in \N$, together with Equation \eqref{eq:qualofapprox_ineq_1} we get the desired result.
\end{proof}

\subsection{Convergence of the Algorithm}

To prove the convergence of our algorithm, we employ a technique similar to the one prescribed in Section 1.2 in \cite{Polak1997}. 
Summarizing the technique, one can think of an algorithm as discrete-time dynamical system, whose desired stable equilibria are characterized by the stationary points of its optimality function, i.e. points $\xi \in {\cal X}_p$ where $\theta(\xi) = 0$, since we know from Theorem \ref{thm:theta_optimality_function} that all local minimizers are stationary. Before applying this line of reasoning to our algorithm, we present a simplified version of this argument for a general unconstrained optimization problem. This is done in the interest of clarity.
Inspired by the stability analysis of dynamical systems, a sufficient condition for the convergence of our algorithm can be formulated by requiring that the cost function satisfy a notion of sufficient descent with respect to an optimality function:
\begin{definition}
	\label{def:sufficient_descent}
  Let ${\cal S}$ be a metric space, and consider the problem of minimizing the cost function $J: {\cal S} \to \R$.
  We say that a function $\Gamma: {\cal S} \to {\cal S}$ has the \emph{sufficient descent property} with respect to an optimality function $\theta: {\cal S} \to (-\infty,0]$ if for each $x \in {\cal S}$ with $\theta(x) < 0$, there exists a $\delta_x > 0$ and $O_x \subset {\cal S}$, a neighborhood of $x$, such that:
  \begin{equation}
    J\bigl( \Gamma(x') \bigr) - J(x') \leq - \delta_x, \quad \forall x' \in O_x.
  \end{equation}
\end{definition}

Importantly, a function satisfying the sufficient property can be proven to approach the zeros of the optimality function:
\begin{theorem}[Theorem 1.2.8 in \cite{Polak1997}]
  \label{thm:sufficient_descent}
  Consider the problem of minimizing a cost function $J: {\cal S} \to \R$.
  Suppose that ${\cal S}$ is a a metric space and a function $\Gamma: {\cal S} \to {\cal S}$ has the sufficient descent property with respect to an optimality function $\theta: {\cal S} \to (-\infty,0]$, as described in Definition \ref{def:sufficient_descent}.
  Let $\{ x_j \}_{j \in \N}$ be a sequence such that, for each $j \in \N$:
  \begin{equation}
    x_{j+1} =
    \begin{cases}
      \Gamma( x_j ) & \text{if}\ \theta(x_j) < 0, \\
      x_j & \text{if}\ \theta(x_j) = 0. \\
    \end{cases}
  \end{equation}
  Then every accumulation point of $\{ x_j \}_{j \in \N}$ belongs to the set of zeros of the optimality function $\theta$.
\end{theorem}

Theorem \ref{thm:sufficient_descent}, as originally stated in \cite{Polak1997}, requires ${\cal S}$ to be a Euclidean space, but the result as presented here can be proven without requiring this property using the same original argument. Though Theorem \ref{thm:sufficient_descent} proves that the accumulation point of a sequence generated by $\Gamma$ converges to a stationary point of the optimality function, it does not prove the existence of the accumulation point. 
This is in general not a problem for finite-dimensional optimization problems since the level sets of the cost function are usually compact, thus every sequence produced by a descent method has at least one accumulation point.
On the other hand, infinite-dimensional problems, such as optimal control problems, do not have this property, since bounded sets may not be compact in infinite-dimensional vector spaces.
Thus, even though Theorem \ref{thm:sufficient_descent} can be applied to both finite-dimensional and infinite-dimensional optimization problems, the result is much weaker in the latter case.

The issue mentioned above has been addressed several times in the literature \cite{Axelsson2008,Polak1984,Wardi2012,Wardi2012draft}, by formulating a stronger version of sufficient descent:
\begin{definition}[Definition 2.1 in \cite{Axelsson2008}]
	\label{def:uniform_sufficient_descent}
  Let ${\cal S}$ be a metric space, and consider the problem of minimizing the cost function $J: {\cal S} \to \R$.
  A function $\Gamma: {\cal S} \to {\cal S}$ has the \emph{uniform sufficient descent property} with respect to an optimality function $\theta: {\cal S} \to (-\infty,0]$ if for each $C > 0$ there exists a $\delta_C > 0$ such that, for every $x \in {\cal S}$ with $\theta(x) < 0$,
  \begin{equation}
    J\bigl( \Gamma(x) \bigr) - J(x) \leq - \delta_C.
  \end{equation}
\end{definition}

A sequence of points generated by an algorithm satisfying this property, under mild assumptions, can be shown to approach the zeros of the optimality function:
\begin{theorem}[Proposition 2.1 in \cite{Axelsson2008}]
  \label{thm:uniform_sufficient_descent}
  Consider the problem of minimizing a lower bounded cost function $J: {\cal S} \to [\alpha, \infty)$.
  Suppose that ${\cal S}$ is a a metric space and $\Gamma: {\cal S} \to {\cal S}$ satisfies the uniform sufficient descent property with respect to an optimality function $\theta: {\cal S} \to (-\infty,0]$, as stated in Definition \ref{def:uniform_sufficient_descent}.
  Let $\{ x_j \}_{j \in \N}$ be a sequence such that, for each $j \in \N$:
  \begin{equation}
    x_{j+1} =
    \begin{cases}
      \Gamma( x_j ) & \text{if}\ \theta(x_j) < 0, \\
      x_j & \text{if}\ \theta(x_j) = 0. \\
    \end{cases}
  \end{equation}
  Then,
  \begin{equation}
    \lim_{j \to \infty} \theta( x_j ) = 0.
  \end{equation}
\end{theorem}
\begin{proof}
	Suppose that $\liminf_{j \to \infty} \theta(x_j) = - 2 \epsilon < 0$.
  Then there exists a subsequence $\{ x_{j_k} \}_{k \in \N}$ such that $\theta( x_{j_k} ) < - \epsilon$ for each $k \in \N$.
  Definition \ref{def:uniform_sufficient_descent} implies that there exists $\delta_\epsilon$ such that
  \begin{equation}
    J( x_{j_k + 1} ) - J( x_{j_k} ) \leq - \delta_\epsilon, \quad \forall k \in \N.
  \end{equation}
  But this is a contradiction, since $J( x_{j+1} ) \leq J( x_j )$ for each $j \in \N$, thus $J( x_j ) \to - \infty$ as $j \to \infty$, contrary to the assumption that $J$ is lower bounded.
\end{proof}
Note that Theorem \ref{thm:uniform_sufficient_descent} does not assume the existence of accumulation points of the sequence $\{x_j\}_{j \in \N}$. 
Thus, this Theorem remains valid even when the sequence generated by $\Gamma$ does not have accumulation points. This becomes tremendously useful in infinite-dimensional problems where the level sets of the cost function may not be compact. Though we include these results for the sake of completeness of presentation, our proof of convergence of the sequence of points generated by Algorithm \ref{algo:main_algo} does not make explicit use of Theorem \ref{thm:uniform_sufficient_descent}. The line of argument is similar, but our approach, as described in Theorem \ref{thm:convergence}, requires special treatment due to the projection operation, $\rho_N$, as defined in Equation \eqref{eq:rho} and the existence of constraints.

Now, we begin the convergence proof of Algorithm \ref{algo:main_algo} by showing that the Armijo algorithm, as defined in Equation \eqref{eq:armijo_line_search}, terminates after a finite number of steps and its value is bounded.
\begin{lemma}
  \label{lemma:mu_upper_bound}
  Let $\alpha \in (0,1)$ and $\beta \in \left( 0, 1 \right)$.
For every $\delta > 0$ there exists an $M_\delta^* < \infty$ such that if $\theta(\xi) \leq -\delta$ for $\xi \in {\cal X}_p$, then $\mu(\xi) \leq M^*_\delta$, where $\theta$ is as defined in Equation \eqref{eq:theta_definition} and $\mu$ is as defined in Equation \eqref{eq:armijo_line_search}.
\end{lemma}
\begin{proof}
  Given $\xi' \in {\cal X}$ and $\lambda \in [0,1]$, using the Mean Value Theorem and Corollary \ref{corollary:DJ_lipschitz} we have that there exists $s \in (0,1)$ such that
  \begin{equation}
    \label{eq:muub_J_mvt}
    \begin{aligned}
      J\big( \xi + \lambda (\xi' - \xi) \big) - J(\xi) 
      &= \d{J}\bigl( \xi + s \lambda ( \xi' - \xi ); \lambda (\xi' - \xi) \bigr) \\
      &\leq \lambda \d{J}\bigl( \xi; \xi' - \xi \bigr) + L \lambda^2 \| \xi' - \xi \|^2_{\cal X}.
    \end{aligned}
  \end{equation}
  Letting ${\cal A}(\xi) = \left\{ (j,t) \in {\cal J} \times [0,1] \mid \Psi(\xi) = h_j\big( x^{(\xi)}(t) \big) \right\}$, then there exists a pair $(j,t) \in {\cal A}\bigl( \xi + \lambda (\xi' - \xi) \bigr)$ and $s \in (0,1)$ such that, using Corollary \ref{corollary:Dhj_lipschitz},
  \begin{equation}
    \label{eq:muub_psi_mvt}
    \begin{aligned}
      \Psi\big( \xi + \lambda (\xi' - \xi) \big) - \Psi(\xi) 
      &\leq \psi_{j,t}\bigl( \xi + \lambda ( \xi' - \xi ) \bigr) - \Psi(\xi) \\
      &\leq \psi_{j,t}\bigl( \xi + \lambda ( \xi' - \xi ) \bigr) - \psi_{j,t}(\xi) \\
      &= \d{\psi_{j,t}}\bigl( \xi + s \lambda ( \xi' - \xi ); \lambda ( \xi' - \xi ) \bigr) \\
      &\leq \lambda \d{\psi_{j,t}}\bigl( \xi; \xi' - \xi \bigr) + L \lambda^2 \| \xi' - \xi \|^2_{\cal X}.
    \end{aligned}
  \end{equation}

  Now let us assume that $\Psi(\xi) \leq 0$, and consider $g$ as defined in Equation \eqref{eq:theta_definition}. Then
  \begin{equation}
    \label{eq:muub_theta}
    \theta(\xi) = \max \left\{ \d{J}( \xi; g(\xi) - \xi ), 
    \max_{(j,t) \in {\cal J} \times [0,1]} \D{\psi_{j,t}}( \xi; g(\xi) - \xi ) + \gamma \Psi(\xi) \right\} 
    \leq -\delta,
  \end{equation}
  and using Equation \eqref{eq:muub_J_mvt},
  \begin{equation}
    J\big( \xi + \beta^k ( g(\xi) - \xi ) \big) - J(\xi) - \alpha \beta^k \theta(\xi) 
    \leq - ( 1 - \alpha ) \delta \beta^k + 4 A^2 L \beta^{2k},
  \end{equation}
  where $A = \max\big\{ \|u\|_2 + 1 \mid u \in U \big\}$. Hence, for each $k \in \N$ such that $\beta^k \leq \frac{ (1-\alpha) \delta }{ 4 A^2 L}$ we have that
  \begin{equation}
    \label{eq:muub_res_1}
    J\big( \xi + \beta^k ( g(\xi) - \xi ) \big) - J(\xi) \leq \alpha \beta^k \theta(\xi).
  \end{equation}
  Similarly, using Equations \eqref{eq:muub_psi_mvt} and \eqref{eq:muub_theta},
  \begin{equation}
    \Psi\big( \xi + \beta^k ( g(\xi) - \xi ) \big) - \Psi(\xi) + \beta^k \bigl( \gamma \Psi(\xi) - \alpha \theta(\xi) \bigr)
    \leq - \delta \beta^k + 4 A^2 L \beta^{2k},
  \end{equation}
  hence for each $k \in \N$ such that $\beta^k \leq \min\left\{\frac{ (1-\alpha) \delta}{4 A^2 L}, \frac{1}{\gamma} \right\}$ we have that
  \begin{equation}
    \label{eq:muub_res_2}
    \Psi\big( \xi + \beta^k ( g(\xi) - \xi ) \big) - \alpha \beta^k \theta(\xi) \leq \left( 1 - \beta^k \gamma \right) \Psi(\xi) \leq 0.
  \end{equation}

  If $\Psi(\xi) > 0$ then 
  \begin{equation}
    \max_{(j,t)\in{\cal J}\times[0,1]} \D{\psi_{j,t}}( \xi; g(\xi) - \xi ) \leq \theta(\xi) \leq -\delta,
  \end{equation}
  thus, from Equation \eqref{eq:muub_psi_mvt},
  \begin{equation}
    \Psi\big( \xi + \beta^k ( g(\xi) - \xi ) \big) - \Psi(\xi) - \alpha \beta^k \theta(\xi) 
    \leq - ( 1 - \alpha ) \delta \beta^k + 4 A^2 L \beta^{2k}.
  \end{equation}
  Hence, for each $k \in \N$ such that $\beta^k \leq \frac{ (1-\alpha) \delta}{4 A^2 L}$ we have that
  \begin{equation}
    \label{eq:muub_res_3}
    \Psi\big( \xi + \beta^k ( g(\xi) - \xi ) \big) - \Psi(\xi) \leq \alpha \beta^k \theta(\xi).    
  \end{equation}

  Finally, let 
  \begin{equation}
    M^*_\delta = 1 + \max \left\{ \log_{\beta} \left( \frac{ (1-\alpha) \delta}{4 A^2 L} \right), \log_{\beta} \left( \frac{1}{\gamma} \right) \right\},
  \end{equation}
  then from Equations \eqref{eq:muub_res_1}, \eqref{eq:muub_res_2}, and \eqref{eq:muub_res_3}, we get that $\mu(\xi) \leq M^*_\delta$ as desired.
\end{proof}

Next, we show that the determination of the frequency at which to perform pulse width modulation as defined in Equation \eqref{eq:armijo_pwm} terminates after a finite number of steps.
\begin{lemma}
  \label{lemma:nu_upper_bound}
  Let $\alpha \in (0,1)$, $\bar{\alpha} \in (0,\infty)$, $\beta \in (0,1)$, $\bar{\beta} \in \left( \frac{1}{\sqrt{2}}, 1 \right)$, and $\xi \in {\cal X}_p$. 
  If $\theta(\xi) < 0$, then $\nu(\xi) < \infty$, where $\theta$ is as defined in Equation \eqref{eq:theta_definition} and $\nu$ is as defined in Equation \eqref{eq:armijo_pwm}.
\end{lemma}
\begin{proof}
  Throughout the proof, we leave out the natural inclusion taking $\xi \in {\cal X}_p$ to $\xi \in {\cal X}_r$.
  To simplify our notation let us denote $M = \mu(\xi)$ and $\xi' = \xi + \beta^M \bigl( g(\xi) - \xi \bigr)$. 
  Theorem \ref{thm:quality_of_approximation} implies that there exists $K > 0$ such that
  \begin{equation}
    J\bigl( \rho_N( \xi' ) \bigr) - J( \xi' )
    \leq K L \left( \frac{1}{\sqrt{2}} \right)^N \bigl( \| \xi' \|_{BV} + 1 \bigr),
  \end{equation}
  where $L$ is the constant defined in Assumption \ref{assump:constraint_fns}. 

  Let ${\cal A}(\xi) = \left\{ (j,t) \in \{1,\ldots,N_c\} \times [0,1] \mid \Psi(\xi) = h_j\big( x^{(\xi)}(t) \big) \right\}$, then for each pair $(j,t) \in {\cal A}\bigl( \rho_N(\xi') \bigr)$ we have that
  \begin{equation}
    \begin{aligned}
      \Psi\bigl( \rho_N(\xi') \bigr) - \Psi( \xi' ) 
      &= \psi_{j,t}\bigl( \rho_N( \xi' ) \bigr) - \Psi(\xi') \\
      &\leq \psi_{j,t}\bigl( \rho_N( \xi' ) \bigr) - \psi_{j,t}(\xi') \\
      &\leq K L \left( \frac{1}{\sqrt{2}} \right)^N \bigl( \| \xi' \|_{BV} + 1 \bigr).
    \end{aligned}
  \end{equation}

  Recall that $\bar{\alpha} \in (0,\infty)$, $\bar{\beta} \in \left( \frac{1}{\sqrt{2}}, 1 \right)$, and $\omega \in (0,1)$, hence there exists $N_0 \in \N$ such that, for each $N \geq N_0$,
  \begin{equation}
    K L \left( \frac{1}{\sqrt{2}} \right)^N \bigl( \| \xi' \|_{BV} + 1 \bigr) \leq -\bar{\alpha} \bar{\beta}^N \theta(\xi).
  \end{equation}
  Also, there exists $N_1 \geq N_0$ such that, for each $N \geq N_1$,
  \begin{equation}
    \bar{\alpha} \bar{\beta}^N \leq ( 1 - \omega ) \alpha \beta^M.
  \end{equation}

  Now suppose that $\Psi(\xi) \leq 0$, then, for each $N \geq N_1$,
  \begin{equation}
    \label{eq:nuub_res_1}
    \begin{aligned}
      J\bigl( \rho_N( \xi' ) \bigr) - J(\xi) 
      &= J\bigl( \rho_N( \xi' ) \bigr) - J( \xi' ) + J( \xi' ) - J( \xi ) \\
      &\leq \left( \alpha \beta^M - \bar{\alpha} \bar{\beta}^N \right) \theta(\xi),
    \end{aligned}
  \end{equation}
  and 
  \begin{equation}
    \label{eq:nuub_res_2}
    \begin{aligned}
      \Psi\bigl( \rho_N( \xi' ) \bigr)
      &= \Psi\bigl( \rho_N( \xi' ) \bigr) - \Psi( \xi' ) + \Psi( \xi' ) \\
      &\leq \left( \alpha \beta^M - \bar{\alpha} \bar{\beta}^N \right) \theta(\xi) \\
      &\leq 0.
    \end{aligned}
  \end{equation}
  Similarly, if $\Psi(\xi) > 0$ then, using the same argument as above, we have that
  \begin{equation}
    \label{eq:nuub_res_3}
      \Psi\bigl( \rho_N( \xi' ) \bigr) - \Psi(\xi) 
      \leq \left( \alpha \beta^M - \bar{\alpha} \bar{\beta}^N \right) \theta(\xi).
  \end{equation}
  Therefore, from Equations \eqref{eq:nuub_res_1}, \eqref{eq:nuub_res_2}, and \eqref{eq:nuub_res_3}, it follows that $\nu(\xi) \leq N_1$ as desired.
\end{proof}

The following lemma proves that, once Algorithm \ref{algo:main_algo} finds a feasible point, every point generated afterwards is also feasible. We omit the proof since it follows directly from the definition of $\nu$ in Equation \eqref{eq:armijo_pwm}.
\begin{lemma}
  \label{lemma:phase12}
  	Let $\Gamma$ be defined as in Equation \eqref{eq:gamma_def} and let $\Psi$ be as defined in Equation \eqref{eq:super_const}.
	Let $\{\xi_i\}_{i \in \N}$ be a sequence generated by Algorithm \ref{algo:main_algo}. 
  	If there exists $i_0 \in \N$ such that $\Psi(\xi_{i_0}) \leq 0$, then $\Psi(\xi_i) \leq 0$ for each $i \geq i_0$.
\end{lemma}

Employing these preceding results, we can prove the convergence of Algorithm \ref{algo:main_algo} to a point that satisfies our optimality condition by employing an argument similar to the one used in the proof of Theorem \ref{thm:uniform_sufficient_descent}:
\begin{theorem}
  \label{thm:convergence}
  Let $\theta$ be defined as in Equation \eqref{eq:theta_definition}.
  If $\{\xi_i\}_{i \in \N}$ is a sequence generated by Algorithm \ref{algo:main_algo}, then $\lim_{i \to \infty} \theta(\xi_i) = 0$.
\end{theorem}
\begin{proof}
  If the sequence produced by Algorithm \ref{algo:main_algo} is finite, then the theorem is trivially satisfied, so we assume that the sequence is infinite.
  
  Suppose the theorem is not true, then $\liminf_{i \to \infty} \theta(\xi_i) = -2 \delta < 0$ and therefore there exists $k_0 \in \N$ and a subsequence $\{\xi_{i_k}\}_{k \in \N}$ such that $\theta( \xi_{i_k} ) \leq - \delta$ for each $k \geq k_0$. 
  Also, recall that $\nu(\xi)$ was chosen such that, given $\mu(\xi)$,
  \begin{equation}
    \alpha \beta^{\mu(\xi)} - \bar{\alpha} \bar{\beta}^{\nu(\xi)} \geq \omega \alpha \beta^{\mu(\xi)},
  \end{equation}
  where $\omega \in (0,1)$ is a parameter.

  From Lemma \ref{lemma:mu_upper_bound} we know that there exists $M^*_\delta$, which depends on $\delta$, such that $\beta^{\mu(\xi)} \geq \beta^{M^*_\delta}$.
  Suppose that the subsequence $\{\xi_{i_k}\}_{k \in \N}$ is eventually feasible, then, by Lemma \ref{lemma:phase12}, without loss of generality we can assume that the sequence is always feasible. Thus, given $\Gamma$ as defined in Equation \eqref{eq:gamma_def},
  \begin{equation}
    \begin{aligned}
      J\big( \Gamma( \xi_{i_k} ) \big) - J( \xi_{i_k} ) 
      &\leq \big( \alpha \beta^{\mu(\xi)} - \bar{\alpha} \bar{\beta}^{\nu(\xi)} \big) \theta( \xi_{i_k} ) \\
      &\leq - \omega \alpha \beta^{\mu(\xi)} \delta \\
      &\leq - \omega \alpha \beta^{M^*_\delta} \delta.
    \end{aligned}
  \end{equation}
  This inequality, together with the fact that $J(\xi_{i+1}) \leq J(\xi_i)$ for each $i \in \N$, implies that $\liminf_{k \to \infty} J(\xi_{i_k}) = -\infty$, but this is a contradiction since $J$ is lower bounded, which follows from Condition \ref{corollary:phixi_lipschitz} in Corollary \ref{corollary:constraints_lipschitz}.

  The case when the sequence is never feasible is analogous after noting that, since the subsequence is infeasible, then $\Psi(\xi_{i_k}) > 0$ for each $k \in \N$, establishing a similar contradiction.
\end{proof}

\section{Implementable Algorithm}
\label{sec:implementation}

In this section, we describe how to implement Algorithm \ref{algo:main_algo} given the various algorithmic components derived in the Section \ref{sec:algo_analysis}. 
Numerically computing a solution to the Switched System Optimal Control Problem defined as in Equation \eqref{eq:ssocp} demands employing some form of discretization. 
When numerical integration is introduced, the original infinite-dimensional optimization problem defined over function spaces is replaced by a finite-dimensional discrete-time optimal control problem. 
Changing the discretization precision results in an infinite sequence of such approximating problems. 

Our goal is the construction of an implementable algorithm that generates a sequence of points by recursive application that converge to a point that satisfies the optimality condition defined in Equation \eqref{eq:theta_definition}. Given a particular choice of discretization precision, at a high level, our algorithm solves a finite dimensional optimization problem and terminates its operation when a discretization improvement test is satisfied. At this point, a finer discretization precision is chosen, and the whole process is repeated, using the last iterate, obtained with the coarser discretization precision as a ``warm start.''

In this section, we begin by describing our discretization strategy, which allows us to define our discretized optimization spaces. Next, we describe how to construct discretized trajectories, cost, constraints, and optimal control problems. This allows us to define a discretized optimality function, and a notion of \emph{consistent approximation} between the optimality function and its discretized counterpart. We conclude by constructing our numerically implementable optimal control algorithm for constrained switched systems.

\subsection{Discretized Optimization Space}

To define our discretization strategy, for any positive integer $N$ we first define the \emph{$N$--th switching time space} as:
\begin{equation}
  \label{eq:switching_times}
  {\cal T}_N = 
  \left\{ (\tau_0,\ldots,\tau_k) \subset [0,1] \mid 
    0 = \tau_0 \leq \tau_1 \leq \cdots \leq \tau_k = 1,\ 
    | \tau_i - \tau_{i-1} | \leq \frac{1}{2^N}\ \forall i \in \{ 1, \ldots, k \} 
  \right\},
\end{equation}
i.e. ${\cal T}_N$ is the collection of finite partitions of $[0,1]$ whose samples have a maximum distance of $\frac{1}{2^N}$.
For notational convenience, given $\tau \in {\cal T}_N$, we define $\card{\tau}$ as the cardinality of $\tau$.
Importantly, notice that the sets ${\cal T}_N$ are nested, i.e. for each $N \in \N$, ${\cal T}_{N+1} \subset {\cal T}_N$.

We utilize the switching time spaces to define a sequence of finite dimensional subspaces of ${\cal X}_p$ and ${\cal X}_r$.
Given $N \in \N$, $\tau \in {\cal T}_N$, and $k \in \{ 0, \ldots, \card{\tau}-1 \}$, we define $\pi_{\tau,k}: [0,1] \to \R$ that scales the discretization:
\begin{equation}
  \label{eq:scaling_discretization}
  \pi_{\tau,k}(t) = 
  \begin{cases}
    1 & \text{if}\ t \in [\tau_{k}, \tau_{k+1} ), \\
    0 & \text{otherwise}.
  \end{cases}
\end{equation}
Using this definition, we define ${\cal D}_{\tau,p}$, a subspace of the discrete input space, as:
\begin{equation}
  \label{eq:discretized_pure_discrete_input_space}
  {\cal D}_{\tau,p} = \left\{ d \in {\cal D}_p \mid 
    d = \sum_{k=0}^{\card{\tau}-1} \bar{d}_k \pi_{\tau,k},\ \bar{d}_k \in \Sigma^q_p\ \forall k
  \right\}.
\end{equation}
Similarly, we define ${\cal D}_{\tau,r}$, a subspace of the relaxed discrete input space, as:
\begin{equation}
  \label{eq:discretized_relaxed_discrete_input_space}
  {\cal D}_{\tau,r} = \left\{ d \in {\cal D}_r \mid 
    d = \sum_{k=0}^{\card{\tau}-1} \bar{d}_k \pi_{\tau,k},\ \bar{d}_k \in \Sigma^q_r\ \forall k 
  \right\}.
\end{equation}
Finally, we define ${\cal U}_{\tau}$, a subspace of the continuous input space, as:
\begin{equation}
  \label{eq:discretized_input_space}
  {\cal U}_{\tau} = \left\{ u \in {\cal U} \mid 
    u = \sum_{k=0}^{\card{\tau}-1} \bar{u}_k \pi_{\tau,k},\ \bar{u}_k \in U\ \forall k 
  \right\}.
\end{equation}
Now, we can define the \emph{$N$--th discretized pure optimization space induced by switching vector $\tau$} as ${\cal X}_{\tau,p} = {\cal U}_{\tau} \times {\cal D}_{\tau,p}$, and the \emph{$N$--th discretized relaxed optimization space induced by switching vector $\tau$} as ${\cal X}_{\tau,r} =  {\cal U}_{\tau} \times {\cal D}_{\tau,r}$. Similarly, we define a subspace of ${\cal X}$:
\begin{equation}
  \label{eq:discretized_X}
  {\cal X}_\tau = \left\{ (u,d) \in {\cal X} \mid 
    u = \sum_{k=0}^{\card{\tau}-1} \bar{u}_k \pi_{\tau,k},\ \bar{u}_k \in \R^m\ \forall k,\ \text{and}\  
    d = \sum_{k=0}^{\card{\tau}-1} \bar{d}_k \pi_{\tau,k},\ \bar{d}_k \in \R^q\ \forall k
  \right\}.
\end{equation}

In order for these discretized optimization spaces to be useful, we need to know to show that we can use a sequence of functions belonging to these finite-dimensional subspaces to approximate any infinite dimensional function. The following lemma proves this result and validates our choice of discretized spaces:
\begin{lemma}
  \label{lemma:dense_discretized_spaces}
  Let $\{ \tau_k \}_{k \in \N}$ with $\tau_k \in {\cal T}_k$.
  \begin{enumerate_parentesis}
  \item \label{lemma:dense_pure_discretized_space} For each $\xi \in {\cal X}_p$ there exists a sequence $\{ \xi_k \}_{k \in \N}$, with $\xi_k \in {\cal X}_{\tau_k,p}$, such that $\xi_k \to \xi$ as $k \to \infty$.
  \item \label{lemma:dense_relaxed_discretized_space} For each $\xi \in {\cal X}_r$ there exists a sequence $\{ \xi_k \}_{k \in \N}$, with $\xi_k \in {\cal X}_{\tau_k,r}$, such that $\xi_k \to \xi$ as $k \to \infty$.
  \end{enumerate_parentesis}
\end{lemma}
\begin{proof}
  We only present an outline of the proof, since the argument is outside the scope of this paper.
  First, every Lebesgue measurable set in $[0,1]$ can be arbitrarily approximated by intervals (Theorem 2.40 in \cite{Folland1999}).
  Second, the sequence of partitions $\{ \tau_k \}_{k \in \N}$ can clearly approximate any interval.
  Finally, the result follows since every measurable function can be approximated in the $L^2$--norm by integrable simple functions, which are the finite linear combination of indicator functions defined on Borel sets (Theorem 2.10 in \cite{Folland1999}).
\end{proof}

\subsection{Discretized Trajectories, Cost, Constraint, and Optimal Control Problem}

For a positive integer $N$, given a switching vector, $\tau \in {\cal T}_{N}$, a relaxed control $\xi = (u,d) \in {\cal X}_{\tau,r}$, and an initial condition $x_0 \in \R^n$, the discrete dynamics, denoted by $\bigl\{ z^{(\xi)}_\tau( \tau_k ) \bigr\}_{k=0}^{\card{\tau}} \subset \R^n$, are computed via the Forward Euler Integration Formula:
\begin{equation}
  \label{eq:euler_traj_xi}
  z^{(\xi)}_\tau ( \tau_{k+1} ) = z^{(\xi)}_\tau ( \tau_k ) + 
  ( \tau_{k+1} - \tau_{k} ) f\bigl( \tau_k, z^{(\xi)}_\tau ( \tau_k ), u( \tau_k ), d( \tau_k ) \bigr), 
  \quad \forall k \in \{0,\ldots,\card{\tau}-1\},
  \quad z^{(\xi)}_\tau (0) = x_0.
\end{equation}
Employing these discrete dynamics we can define the \emph{discretized trajectory}, $z^{(\xi)}_\tau: [0,1] \to \R^n$, by performing linear interpolation over the discrete dynamics:
\begin{equation}
  \label{eq:discretized_traj_xi}
  z^{(\xi)}_\tau(t) = \sum_{k=0}^{\card{\tau}-1} \left( z^{(\xi)}_\tau (\tau_k) + 
    \frac{t - \tau_k}{\tau_{k+1} - \tau_{k}} \bigl( z^{(\xi)}_\tau (\tau_{k+1}) - z^{(\xi)}_\tau(\tau_{k}) \bigr) 
  \right) \pi_{\tau,k}(t),
\end{equation}
where $\pi_{\tau,k}$ are as defined in Equation \eqref{eq:scaling_discretization}.
Note that the definition in Equation \eqref{eq:discretized_traj_xi} is valid even if $\tau_k = \tau_{k+1}$ for some $k \in \{ 0, \ldots, \card{\tau} \}$, which becomes clear after replacing Equation \eqref{eq:euler_traj_xi} in Equation \eqref{eq:discretized_traj_xi}.
For notational convenience, we suppress the dependence on $\tau$ in $z^{(\xi)}_\tau$ when it is clear in context.

Employing the trajectory computed via Euler integration, we define the \emph{discretized cost function}, $J_{\tau}: {\cal X}_{\tau,r} \to \R$:
\begin{equation}
  \label{eq:discretized_cost}
  J_{\tau}(\xi) = h_0\bigl( z^{(\xi)}(1) \bigr).
\end{equation}
Similarly, we define the \emph{discretized constraint function}, $\psi_{\tau}: {\cal X}_{\tau,r} \to \R$:
\begin{equation}
  \label{eq:discretized_constraint}
  \Psi_{\tau}(\xi) = \max_{j \in {\cal J},\; k \in \{0,\ldots,\card{\tau}\}} h_j\bigl( z^{(\xi)}( \tau_k ) \bigr).
\end{equation}
Note that these definitions extend easily to points belonging to ${\cal X}_{\tau,p}$. 

As we did in Section \ref{sec:trajs_cost_cons_ocp}, we now introduce some additional notation to ensure the clarity of the ensuing analysis. 
First, for any positive integer $N$ and $\tau \in {\cal T}_N$, we define the \emph{discretized flow of the system}, $\phi_{\tau,t}: {\cal X}_r \to \R^n$ for each $t \in [0,1]$ as:
\begin{equation}
  \label{eq:discretized_flow_xi}
  \phi_{\tau,t}(\xi) = z^{(\xi)}_\tau(t).
\end{equation}
Second, for any positive integer $N$ and $\tau \in {\cal T}_N$, we define \emph{component constraint functions}, $\psi_{\tau,j,t}: {\cal X}_r \to \R$ for each $t \in [0,1]$ and each $j \in {\cal J}$ as:
\begin{equation}
  \label{eq:discretized_component_constraints}
  \psi_{\tau,j,t}(\xi) = h_j\bigl( \phi_{\tau,t}(\xi) \bigr).
\end{equation}
Notice that the discretized cost function and the discretized constraint function become
\begin{equation}
  J_{\tau}(\xi) = h_0\bigl( \phi_{\tau,1}(\xi) \bigr),\quad \text{and}\quad 
  \Psi_\tau(\xi) = \max_{ j \in {\cal J},\; k \in \{0,\ldots,\card{\tau}\} } \psi_{\tau,j,\tau_k}(\xi),
\end{equation}
respectively.
This notation change is made to emphasize the dependence on $\xi$.

\subsection{Local Minimizers and a Discretized Optimality Condition}

Before proceeding further, we make an observation that dictates the construction of our implementable algorithm. Recall how we employ directional derivatives and Theorem \ref{thm:bangbang_result} in order to construct a necessary condition for optimality for the Switched System Optimal Control Problem. In particular, if at a particular point belonging to the pure optimization space the appropriate directional derivatives are negative, then the point is not a local minimizer of the Relaxed Switched System Optimal Control Problem. An application of Theorem \ref{thm:bangbang_result} to this point proves that it is not a local minimizer of the Switched System Optimal Control Problem. 

Proceeding in a similar fashion, for any positive integer $N \in \N$ and $\tau \in {\cal T}_N$, we can define a Discretized Relaxed Switched System Optimal Control Problem:
\begin{DRSSOCP}
  \begin{equation}
	\label{eq:drssocp}
    \min_{\xi \in {\cal X}_{\tau,r}} \left\{ J_{\tau}(\xi) \mid \Psi_{\tau}(\xi) \leq 0 \right\}.
  \end{equation}
\end{DRSSOCP}
\noindent The local minimizers of this problem are then defined as follows:
\begin{definition}
  \label{def:discretized_local_minimizer_Xr}
  Fix $N \in \N$, and $\tau \in {\cal T}_N$. Let us denote an $\epsilon$--ball in the ${\cal X}$--norm centered at $\xi$ induced by switching vector $\tau$ by:
  \begin{equation}
    \label{eq:discretized_nbhd_X}
    {\cal N}_{\tau,\cal X}(\xi,\epsilon) = \left\{ \bar{\xi} \in {\cal X}_{\tau,r} \mid \bigl\lVert \xi - \bar{\xi} \bigr\rVert_{\cal X} < \epsilon \right\}.
  \end{equation}
	We say that a point $\xi \in {\cal X}_{\tau,r}$ is a \emph{local minimizer of the Relaxed Switched System Optimal Control Problem Induced by Switching Vector $\tau$} defined in Equation \eqref{eq:drssocp} if $\Psi_{\tau}(\xi) \leq 0$ and there exists $\epsilon > 0$ such that $J_{\tau}( \hat{\xi} ) \geq J_{\tau}( \xi )$ for each $\hat{\xi} \in {\cal N}_{\tau,\cal X}(\xi,\epsilon) \cap \left\{ \bar{\xi} \in {\cal X}_{\tau,r} \mid \Psi_{\tau}( \bar{\xi} ) \leq 0 \right\}$.
\end{definition}
Given this definition, a first order numerical optimal control scheme can exploit the vector space structure of the discretized relaxed optimization space in order to define discretized directional derivatives that find local minimizers for this Discretized Relaxed Switched System Optimal Control Problem. Just as in Section \ref{subsec:directional_derivatives}, we can employ a first order approximation argument and the existence of the directional derivative of the cost, $\D{J}_{\tau}$ (proven in Lemma \ref{lemma:discretized_DJ_definition}), and of each of the component constraints, $\D{\psi}_{\tau,j,\tau_k}$ (proven in Lemma \ref{lemma:Dhj_discrete_definition}), for each $j \in {\cal J}$ and $k \in \{0,\ldots,|\tau|\}$ in order to elucidate this fact.

Employing these directional derivatives, we can define a discretized optimality function. Fixing a positive integer $N$ and $\tau \in {\cal T}_N$, we define a \emph{discretized optimality function}, $\theta_{\tau}: {\cal X}_{\tau,p} \to (-\infty,0]$ and a corresponding \emph{discretized descent direction}, $g_\tau: {\cal X}_{\tau,p} \to {\cal X}_{\tau,r}$:
\begin{equation}
  \label{eq:discrete_theta}
  \theta_{\tau}(\xi) = \min_{\xi' \in {\cal X}_{\tau,r}} \zeta_{\tau}( \xi, \xi' ), \qquad 
  g_{\tau}(\xi) = \argmin_{\xi' \in {\cal X}_{\tau,r}} \zeta_{\tau}( \xi, \xi' ),
\end{equation}
where 
\begin{equation}
  \label{eq:discrete_zeta}
  \zeta_{\tau}(\xi,\xi') =
  \begin{cases}
    \max \left\{ 
      \d{J_{\tau}}(\xi; \xi' - \xi), 
      \max_{ j \in {\cal J},\; k \in \{0,\ldots,\card{\tau}\} }\limits \D{\psi_{\tau,j,\tau_k}}(\xi;\xi'-\xi) + \gamma \Psi_{\tau}(\xi) 
    \right\} + \| \xi' - \xi \|_{\cal X}
    & \text{if}\ \Psi_{\tau}(\xi) \leq 0, \\
    \max \left\{ 
      \d{J_{\tau}}(\xi; \xi' - \xi) - \Psi_{\tau}(\xi), 
      \max_{ j \in {\cal J},\; k \in \{0,\ldots,\card{\tau}\} }\limits \D{\psi_{\tau,j,\tau_k}}(\xi;\xi'-\xi) 
    \right\} + \| \xi' - \xi \|_{\cal X}
    & \text{if}\ \Psi_{\tau}(\xi) > 0, \\
  \end{cases}
\end{equation}
and $\gamma > 0$ is a design parameter as in the original optimality function $\theta$, defined in Equation \eqref{eq:theta_definition}. 
Before proceeding, we make two observations. First, note that $\theta_{\tau}(\xi) \leq 0$ for each $\xi \in {\cal X}_{\tau,p}$, since we can always choose $\xi' = \xi$ which leaves the trajectory unmodified. Second, note that at a point $\xi \in {\cal X}_{\tau,p}$ the directional derivatives in the optimality function consider directions $\xi' - \xi$ with $\xi' \in {\cal X}_{\tau,r}$ in order to ensure that first order approximations belong to the discretized relaxed optimization space ${\cal X}_{\tau,r}$ which is convex (e.g. for $0 < \lambda \ll 1$, $J_{\tau}(\xi) + \lambda \d{J}_{\tau}(\xi;\xi' - \xi) \approx J_{\tau}( (1 - \lambda) \xi + \lambda \xi' )$  where $(1 - \lambda) \xi + \lambda \xi' \in  {\cal X}_{\tau,r}$).

Just as we argued in the infinite dimensional case, we can prove, as we do in Theorem \ref{thm:discretized_theta_optimality_function}, that if $\theta_{\tau}(\xi) < 0$ for some $\xi \in {\cal X}_{\tau,p}$, then $\xi$ is not a local minimizer of the Discretized Relaxed Switched System Optimal Control Problem. Proceeding as we did in Section \ref{sec:optimization_algorithm}, we can attempt to apply Theorem \ref{thm:bangbang_result} to prove that $\theta$ encodes local minimizers by employing the weak topology over the discretized pure optimization space. Unfortunately, Theorem \ref{thm:bangbang_result} does not prove that the element in the pure optimization space, $\xi_p \in {\cal X}_p$, that approximates a particular relaxed control $\xi_r \in {\cal X}_{\tau,r} \subset {\cal X}_r$ at a particular quality of approximation $\epsilon > 0$ with respect to the trajectory of the switched system, belongs to ${\cal X}_{\tau,p}$. Though the point in the pure optimization space that approximates a particular discretized relaxed control at a particular quality of approximation exists, it may exist at a different discretization precision.

This deficiency of Theorem \ref{thm:bangbang_result} which is shared by our extension to it, Theorem \ref{thm:quality_of_approximation}, means that our computationally tractable algorithm, in contrast to our conceptual algorithm, requires an additional step where the discretization precision is allowed to improve. Nevertheless, if we prove that the Discretized Switched System Optimal Control Problem consistently approximates the Switched System Optimal Control Problem in a manner that is formalized next, then an algorithm that generates a sequence of points by recursive application that converge to a point that is a zero of the discretized optimality function is also converging to a point that is a zero of the original optimality function.

Formally, motivated by the approach taken in \cite{Polak1997}, we define consistent approximation as:
\begin{definition}[Definition 3.3.6 \cite{Polak1997}]
	\label{def:consistent_approximation}
	The Discretized Relaxed Switched System Optimal Control Problem as defined in Equation \eqref{eq:drssocp} is a \emph{consistent approximation} of the Switched System Optimal Control Problem as defined in Equation \eqref{eq:ssocp} if for any infinite sequence $\{\tau_i\}_{i \in \N}$ and $\{\xi_i\}_{i \in \N}$ such that $\tau_i \in {\cal T}_i$ and $\xi_i \in {\cal X}_{\tau_i,p}$ for each $i \in \N$, then $\lim_{i \to \infty} \left\lvert \theta_{\tau_i}(\xi_i) - \theta(\xi_i) \right\rvert = 0$, where $\theta$ is as defined in Equation \eqref{eq:theta_definition} and $\theta_{\tau}$ is as defined in Equation \eqref{eq:discrete_theta}.
\end{definition}
Importantly, if this notion of consistent approximation is satisfied, then a critical result follows:
\begin{theorem}
	\label{thm:consistent_approximation_means_convergence}
	Suppose the Discretized Relaxed Switched System Optimal Control Problem, as defined in Equation \eqref{eq:drssocp}, is a consistent approximation, as in Definition \ref{def:consistent_approximation}, of the Switched System Optimal Control Problem, as defined in Equation \eqref{eq:ssocp}. 
  Let $\{\tau_i\}_{i \in \N}$ and $\{\xi_i\}_{i \in \N}$ be such that $\tau_i \in {\cal T}_i$ and $\xi_i \in {\cal X}_{\tau_i,p}$ for each $i \in \N$. 
  In this case, if $\lim_{i\to\infty} \theta_{\tau_i}(\xi_i) = 0$, then $\lim_{i\to\infty}\theta(\xi_i) = 0$.
\end{theorem}
\begin{proof}
	Arguing by contradiction, suppose there exists a $\delta > 0$ such that $\liminf_{i\to\infty} \theta(\xi_i) < -\delta$.
	  Then by the super-additivity of the $\liminf$,
	  \begin{equation}
	    \liminf_{i \to \infty} \theta_{\tau_i}(\xi_i) - \liminf_{i\to\infty} \theta(\xi_i)
	    \leq \liminf_{i\to\infty} \theta_{\tau_i}(\xi_i) - \theta(\xi_i).
	  \end{equation}
	  Rearranging terms and applying Definition \ref{def:consistent_approximation}, we have that:
	  \begin{equation}
	    \liminf_{i \to \infty} \theta_{\tau_i}(\xi_i)
	    \leq \liminf_{i \to \infty} \left(\theta_{\tau_i}(\xi_i) - \theta(\xi_i)\right) + \liminf_{i\to\infty} \theta(\xi_i) < -\delta,
	  \end{equation}
	  which contradicts the fact that $\lim_{i\to\infty} \theta_{\tau_i}(\xi_i) = 0$. Since by Condition \ref{thm:theta_negative} in Theorem \ref{thm:theta_optimality_function}, $\liminf_{i \to \infty} \theta(\xi_i) \leq \limsup_{i \to \infty} \theta(\xi_i) \leq 0$, we have our result.
\end{proof}

To appreciate the importance of this result, observe that if we prove that the Discretized Relaxed Switched System Optimal Control Problem is a consistent approximation of the Switched System Optimal Control Problem, as we do in Theorem \ref{thm:consistent_approximation}, and devise an algorithm for the Discretized Relaxed Switched System Optimal Control Problem that generates a sequence of discretized points that converge to a point that is a zero of the discretized optimality function, then the sequence of points generated actually converges to a point that also satisfies the necessary condition for optimality for the Switched System Optimal Control Problem. 


\subsection{Choosing a Discretized Step Size and Projecting the Discretized Relaxed Discrete Input}

Before describing the step in our algorithm where the discretization precision is allowed to increase, we describe how the descent direction can be exploited in order to construct a point in the discretized relaxed optimization space that either reduces the cost (if the original point is feasible) or the infeasibility (if the original point is infeasible). Just as we did in Section \ref{subsec:project_relaxed_discrete_input}, we employ a line search algorithm similar to the traditional Armijo algorithm used during finite dimensional optimization in order to choose a step size (Algorithm Model 1.2.23 in \cite{Polak1997}). Given $N \in \N$, $\tau \in {\cal T}_N$, $\alpha \in (0,1)$, and $\beta \in (0,1)$, a step size for a point $\xi \in {\cal X}_{\tau,p}$, is chosen by solving the following optimization problem:
\begin{equation}
  \label{eq:discrete_mu}
  \mu_\tau(\xi) = 
  \begin{cases}
    \min \bigl\{ k \in \N \mid J_\tau\big( \xi + \beta^k ( g_\tau(\xi) - \xi ) \big) - J_\tau(\xi) \leq \alpha \beta^k \theta_\tau(\xi), \\
    \phantom{\min \bigl\{ k \in \N \mid {}} \Psi_\tau \big( \xi + \beta^k( g_\tau(\xi) - \xi ) \big) \leq \alpha \beta^k \theta_\tau(\xi) \bigr\} 
    & \text{if}\ \Psi_\tau(\xi) \leq 0, \\
    \min \bigl\{ k \in \N \mid \Psi_\tau\big( \xi + \beta^k ( g_\tau(\xi) - \xi ) \big) - \Psi_\tau(\xi) \leq \alpha \beta^k \theta_\tau(\xi) \bigr\}
    & \text{if}\ \Psi_\tau(\xi) > 0.
  \end{cases}
\end{equation}

Continuing as we did in Section \ref{subsec:project_relaxed_discrete_input}, given $N \in \N$ we can apply ${\cal F}_N$ defined in Equation \eqref{eq:wavelet_approx} and ${\cal P}_N$ defined in Equation \eqref{eq:pwm} to the constructed discretized relaxed discrete input. The pulse width modulation at a particular frequency induces a partition in ${\cal T}_N$ according to the times at which the constructed pure discrete input switched. That is, let $\sigma_N: {\cal X}_r \to {\cal T}_N$ be defined by
\begin{equation}
  \label{eq:sigmaN_def}
  \sigma_N( u, d ) = \bigl\{ 0 \bigr\} \cup \left\{ \frac{k}{2^N} + 
    \frac{1}{2^N} \sum_{j=1}^{i} \left[ {\cal F}_N(d) \right]_j \left( \frac{k}{2^N} \right) 
  \right\}_{
    \begin{subarray}{l}
      i \in \{ 1, \ldots, q \} \\
      k \in \{ 0, \ldots, 2^N - 1\}
    \end{subarray} 
  }.
\end{equation}
Employing this induced partition, we can be more explicit about the range of $\rho_N$ by stating that $\rho_N(\xi) \in {\cal X}_{\sigma_N(\xi),p}$ for each $\xi \in {\cal X}_r$.

Now, given given $N \in \N$, $\tau \in {\cal T}_N$, $\bar{\alpha} \in (0,\infty)$, $\bar{\beta} \in \left( \frac{1}{\sqrt{2}}, 1 \right)$, $\omega \in (0,1)$, and $k_{\text{max}} \in \N$, a frequency at which to perform pulse width modulation for a point $\xi \in {\cal X}_{\tau,p}$ is computed by solving the following optimization problem::
\begin{equation}
  \label{eq:discrete_nu}
  \nu_\tau(\xi,k_{\text{max}}) = 
  \begin{cases}
    \min \bigl\{ k \leq k_{\text{max}} \mid 
    \xi' = \xi + \beta^{\mu_\tau(\xi)} \bigl( g_\tau(\xi) - \xi \bigr),\ 
    \bar{\alpha} \bar{\beta}^k \leq ( 1 - \omega ) \alpha \beta^{\mu_\tau(\xi)}, \\
    \phantom{\min \bigl\{k \leq k_{\text{max}} \mid {}} 
    J_{\sigma_k(\xi')} \bigl( \rho_k( \xi' ) \bigr) - J_\tau(\xi) 
    \leq \bigl( \alpha \beta^{\mu_\tau(\xi)} - \bar{\alpha} \bar{\beta}^k \bigr) \theta_\tau(\xi), \\
    \phantom{\min \bigl\{k \leq k_{\text{max}} \mid {}} 
    \Psi_{\sigma_k(\xi')} \bigl( \rho_k( \xi' ) \bigr) 
    \leq \bigl( \alpha \beta^{\mu_\tau(\xi)} - \bar{\alpha} \bar{\beta}^k \bigr) \theta_\tau(\xi)
    \bigr\} 
    & \text{if}\ \Psi_\tau(\xi) \leq 0, \\

    \min \bigl\{ k \leq k_{\text{max}} \mid 
    \xi' = \xi + \beta^{\mu_\tau(\xi)} \bigl( g_\tau(\xi) - \xi \bigr),\ 
    \bar{\alpha} \bar{\beta}^k \leq ( 1 - \omega ) \alpha \beta^{\mu_\tau(\xi)}, \\
    \phantom{\min \bigl\{k \leq k_{\text{max}} \mid {}} \Psi_{\sigma_k(\xi')} \bigl( \rho_k( \xi' ) \bigr) - \Psi_\tau(\xi) \leq \bigl( \alpha \beta^{\mu_\tau(\xi)} - \bar{\alpha} \bar{\beta}^k \bigr) \theta_\tau(\xi) \bigr\} 
    & \text{if}\ \Psi_\tau(\xi) > 0.
  \end{cases}  
\end{equation}
In the discrete case, as opposed to the original infinite dimensional algorithm, due to the aforementioned shortcomings of Theorem \ref{thm:bangbang_result} and \ref{thm:quality_of_approximation}, there is no guarantee that the optimization problem solved in order to $\nu_\tau$ is feasible. Without loss of generality, we say that $\nu_\tau(\xi) = \infty$ for each $\xi \in {\cal X}_{\tau,r}$ when there is no feasible solution. Importantly letting $N_0 \in \N$, $\tau_0 \in {\cal T}_{N_0}$, and $\xi \in {\cal X}_{\tau,r}$, we prove, in Lemma \ref{lemma:discrete_nu_upper_bound}, that if $\theta(\xi) < 0$ then for each $\eta \in \N$ there exists a finite $N \geq N_0$ such that $\nu_{\sigma_N(\xi)}(\xi,N+\eta)$ is finite. That is, if $\theta(\xi) < 0$, then $\nu_\tau$ is always finite after a certain discretization quality is reached.

\subsection{An Implementable Switched System Optimal Control Algorithm}

Consolidating our definitions, Algorithm \ref{algo:discrete_algo} describes our numerical method to solve the Switched System Optimal Control Problem. 
Notice that note that at each step of Algorithm \ref{algo:discrete_algo}, $\xi_j \in {\cal X}_{\tau_j,p}$. 
Also, observe the two principal differences between Algorithm \ref{algo:main_algo} and Algorithm \ref{algo:discrete_algo}. 

First, as discussed earlier, $\nu_\tau$ maybe infinite as is checked in Line \ref{algo:discrete_step10} of Algorithm \ref{algo:discrete_algo}, at which point the discretization precision is increased since we know that if $\theta(\xi) < 0$, then $\nu_\tau$ is always finite after a certain discretization quality is reached. 
Second, notice that if $\theta_{\tau}$ comes close to zero as is checked in Line \ref{algo:discrete_step3} of Algorithm \ref{algo:discrete_algo}, the discretization precision is increased. 
To understand why this additional check is required, remember that our goal in this paper is the construction of an implementable algorithm that constructs a sequence of points by recursive application that converges to a point that satisfies the optimality condition. 
In particular, $\theta_{\tau}$ may come arbitrarily close to zero due to a particular discretization precision that limits the potential descent directions to search amongst, rather than because it is actually close to a local minimizer of the Switched System Optimal Control Problem. 
This additional step that improves the discretization precision is included in Algorithm \ref{algo:discrete_algo} to guard against this possibility. 

With regards to actual numerical implementation, we make two additional comments. 
First, a stopping criterion is chosen that terminates the operation of the algorithm if $\theta_{\tau}$ is too large. 
We describe our selection of this parameter in Section \ref{sec:examples}. 
Second, due to the definitions of $\D{J}_{\tau}$ and $\D{\psi}_{\tau,j,\tau_k}$ for each $j \in {\cal J}$ and $k \in \{0,\ldots,|\tau|\}$, the optimization required to solve $\theta_{\tau}$ is a quadratic program.

For analysis purposes, we define $\Gamma_\tau: \left\{ \xi \in {\cal X}_{\tau,p} \mid \nu_\tau(\xi,k_{\text{max}}) < \infty \right\} \to {\cal X}_p$ by:
\begin{equation}
  \label{eq:discrete_gamma}
  \Gamma_\tau(\xi) = \rho_{ \nu_\tau( \xi, k_{\text{max}} ) } \bigl( \xi + \beta^{\mu_\tau(\xi)} ( g_\tau(\xi) - \xi ) \bigr).
\end{equation}
We say $\{\xi_j\}_{j \in \N}$ is \emph{a sequence generated by Algorithm \ref{algo:discrete_algo}} if $\xi_{j+1} = \Gamma_{\tau_j}(\xi_j)$ for each $j \in \N$. 
We can prove several important properties about the sequence generated by Algorithm \ref{algo:discrete_algo}. 
First, letting $\{ N_i \}_{i \in \N}$, $\{ \tau_i \}_{i \in \N}$, and $\{ \xi_i \}_{i \in \N}$ be the sequences produced by Algorithm \ref{algo:discrete_algo}, then, as we prove in Lemma \ref{lemma:discrete_phase12}, there exists $i_0 \in \N$ such that, if $\Psi_{\tau_{i_0}}( \xi_{i_0} ) \leq 0$, then $\Psi( \xi_i ) \leq 0$ and $\Psi_{\tau_i}( \xi_i ) \leq 0$ for each $i \geq i_0$. 
That is, once Algorithm \ref{algo:discrete_algo} finds a feasible point, every point generated after it remains feasible. 
Second, as we prove in Theorem \ref{thm:discrete_convergence}, $\lim_{j\to\infty} \theta(\xi_j) = 0$ for a sequence $\{\xi_j\}_{j \in \N}$ generated by Algorithm \ref{algo:discrete_algo}, or Algorithm \ref{algo:discrete_algo} converges to a point that satisfies the optimality condition.

\begin{algorithm}[!ht]
  \begin{algorithmic}[1]
    \REQUIRE
    $N_0 \in \N$, 
    $\tau_0 \in {\cal T}_{N_0}$, 
    $\xi_0 \in {\cal X}_{\tau_0,p}$,
    $\alpha \in (0,1)$, 
    $\bar{\alpha} \in (0,\infty)$, 
    $\beta \in ( 0, 1 )$,
    $\bar{\beta} \in \left( \frac{1}{\sqrt{2}}, 1 \right)$,
    $\gamma \in (0,\infty)$, 
    $\eta \in \N$,
    $\Lambda \in (0,\infty)$,
    $\chi \in \left( 0, \frac{1}{2} \right)$,
    $\omega \in (0,1)$.
    \STATE Set $j = 0$.
    \STATE \label{algo:discrete_step1} Compute $\theta_{\tau_j}(\xi_j)$ as defined in Equation \eqref{eq:discrete_theta}. 
    \IF{\label{algo:discrete_step3} $\theta_{\tau_j}(\xi_j) > - \Lambda 2^{- \chi N_j}$} 
    \STATE Set $\xi_{j+1} = \xi_j$, 
    $N_{j+1} = N_j + 1$, 
    $\tau_{j+1} = \sigma_{N_{j+1}}(\xi_j)$. 
    \STATE Replace $j$ by $j+1$ and go to Line \ref{algo:discrete_step1}.
    \ENDIF
    \STATE Compute $g_{\tau_j}( \xi_j )$ as defined in Equation \eqref{eq:discrete_theta}.
    \STATE Compute $\mu_{\tau_j}( \xi_j )$ as defined in Equation \eqref{eq:discrete_mu}.
    \STATE Compute $\nu_{\tau_j}( \xi_j, N_j + \eta )$ as defined in Equation \eqref{eq:discrete_nu}.
    \IF{\label{algo:discrete_step10} $\nu_{\tau_j}( \xi_j, N_j + \eta ) = \infty$}
    \STATE Set $\xi_{j+1} = \xi_j$, 
    $N_{j+1} = N_j + 1$, 
    $\tau_{j+1} = \sigma_{N_{j+1}}(\xi_{j+1})$. 
    \STATE Replace $j$ by $j+1$ and go to Line \ref{algo:discrete_step1}.
    \ENDIF
    \STATE Set $\xi_{j+1} = \rho_{\nu_{\tau_j}( \xi_j, N_j + \eta )} \bigl( \xi_j + \beta^{\mu_{\tau_j}(\xi_j)} ( g_{\tau_j}(\xi_j) - \xi_j ) \bigr)$,
    $N_{j+1} = \max\left\{ N_j, \nu_{\tau_j}( \xi_j, N_j + \eta ) \right\}$, 
    $\tau_{j+1} = \sigma_{ N_{j+1} }( \xi_{j+1} )$.
    \STATE Replace $j$ by $j+1$ and go to Line \ref{algo:discrete_step1}.
  \end{algorithmic}
  \caption{Numerically Tractable Algorithm for the Switched System Optimal Control Problem}
  \label{algo:discrete_algo}
\end{algorithm}
\section{Implementable Algorithm Analysis}
\label{sec:imp_algo_analysis}

In this section, we derive the various components of Algorithm \ref{algo:discrete_algo} and prove that Algorithm \ref{algo:discrete_algo} converges to a point that satisfies our optimality condition. Our argument proceeds as follows: first, we prove the continuity and convergence of the discretized state, cost, and constraints to their infinite dimensional analogues; second, we construct the components of the optimality function and prove the convergence of these discretized components to their infinite dimensional analogues; finally, we prove the convergence of our algorithm. 

\subsection{Continuity and Convergence of the Discretized Components}

In this subsection, we prove the continuity and convergence of the discretized state, cost, and constraint. We begin by proving the boundedness of the linear interpolation of the Euler Integration scheme:
\begin{lemma}
  \label{lemma:z_bound}
  There exists a constant $C > 0$ such that for each $N \in \N$, $\tau \in {\cal T}_N$, $\xi \in {\cal X}_{\tau,r}$, and $t \in [0,1]$,
  \begin{equation}
    \bigl\| z^{(\xi)}(t) \bigr\|_2 \leq C
  \end{equation}
\end{lemma}
\begin{proof}
  We begin by showing the result for each $t \in \tau$. 
  By Condition \ref{assump:f_lipschitz} in Assumption \ref{assump:fns_continuity}, together with the boundedness of $\| f( 0, x_0, 0, e_i ) \|_2$ for each $i \in {\cal Q}$, there exists a constant $K > 0$ such that, for each $N \in \N$, $\tau \in {\cal T}_N$, $\xi \in {\cal X}_{\tau,r}$, $i \in {\cal Q}$, and $k \in \{ 0, \ldots, \card{\tau} \}$, 
  \begin{equation}
    \bigl\| f\bigl( \tau_k, z^{(\xi)}(\tau_k), u(\tau_k), e_i \bigr) \bigr\|_2 
    \leq K \bigl( \bigl\| z^{(\xi)}(\tau_k) \bigr\|_2 + 1 \bigr).
  \end{equation}
  Employing Equation \eqref{eq:euler_traj_xi} and the Discrete Bellman-Gronwall Inequality (Exercise 5.6.14 in \cite{Polak1997}), we have:
  \begin{equation}
    \begin{aligned}
      \bigl\| z^{(\xi)}(\tau_k) \bigr\|_2 
      &\leq \| x_0 \|_2 + \frac{1}{2^N} \sum_{j=0}^k \sum_{i=1}^q \bigl\| f\bigl( \tau_j, z^{(\xi)}(\tau_j), u(\tau_j), e_i \bigr) \bigr\|_2 \\
      &\leq \left( \| x_0 \|_2 + 1 \right) \left( 1 + \frac{qK}{2^N} \right)^{2^N} \\
      &\leq e^{qK} \left( \| x_0 \|_2 + 1 \right),
    \end{aligned}
  \end{equation}
  thus obtaining the desired result for each $t \in \tau$.

  The result for each $t \in [0,1]$ follows after observing that, in Equation \eqref{eq:discretized_traj_xi}, $\left(\frac{t - \tau_k}{\tau_{k+1} - \tau_{k}}\right) \leq 1$ for each $t \in [\tau_k,\tau_{k+1})$ and $k \in \{ 0, \ldots, \card{\tau} \}$. 
\end{proof}

In fact, this implies that the dynamics, cost, constraints, and their derivatives are all bounded:
\begin{corollary}
  \label{corollary:discretized_fns_bounded}
  There exists a constant $C > 0$ such that for each $N \in \N$, $\tau \in {\cal T}_N$, $j \in {\cal J}$, and $\xi = (u,d) \in {\cal X}_{\tau,r}$:
  \begin{enumerate_parentesis}
  \item \label{corollary:discretized_f_bounded} 
    $\bigl\| f\bigl( t, z^{(\xi)}(t), u(t), d(t) \bigr) \bigr\|_2 \leq C$,\;
    $\left\| \frac{\partial f}{\partial x}\bigl( t, z^{(\xi)}(t), u(t), d(t) \bigr) \right\|_{i,2} \leq C$,\; 
    $\left\| \frac{\partial f}{\partial u}\bigl( t, z^{(\xi)}(t), u(t), d(t) \bigr) \right\|_{i,2} \leq C$.
  \item \label{corollary:discretized_phi_bounded} 
    $\bigl| h_0\bigl( z^{(\xi)}(t) \bigr) \bigr| \leq C$,\; 
    $\Bigl\| \frac{\partial h_0}{\partial x}\bigl( z^{(\xi)}(t) \bigr) \Bigr\|_2 \leq C$,
  \item \label{corollary:discretized_hj_bounded} 
    $\left| h_j\bigl( z^{(\xi)}(t) \bigr) \right| \leq C$,\; 
    $\left\| \frac{\partial h_j}{\partial x}\bigl( z^{(\xi)}(t) \bigr) \right\|_2 \leq C$.
  \end{enumerate_parentesis}
\end{corollary}
\begin{proof}
  The result follows immediately from the continuity of $f$, $\frac{\partial f}{\partial x}$, $\frac{\partial f}{\partial u}$, $h_0$, $\frac{\partial h_0}{\partial x}$, $h_j$, and $\frac{\partial h_j}{\partial x}$ for each $j \in {\cal J}$, as stated in Assumptions \ref{assump:fns_continuity} and \ref{assump:constraint_fns}, and the fact that the arguments of these functions can be constrained to a compact domain, which follows from Lemma \ref{lemma:z_bound} and the compactness of $U$ and $\Sigma_r^q$.
\end{proof}

Next, we prove that the mapping from the discretized relaxed optimization space to the discretized trajectory is Lipschitz:
\begin{lemma}
  \label{lemma:discretized_x_lipschitz}
  There exists a constant $L > 0$ such that, for each $N \in \N$, $\tau \in {\cal T}_N$, $\xi_1,\xi_2 \in {\cal X}_{\tau,r}$ and $t \in [0,1]$:
  \begin{equation}
    \| \phi_{\tau,t}(\xi_1) - \phi_{\tau,t}(\xi_2) \|_2 \leq L \|\xi_1 - \xi_2 \|_{\cal X},
  \end{equation}
  where $\phi_{\tau,t}(\xi)$ is as defined in Equation \eqref{eq:discretized_flow_xi}.
\end{lemma}
\begin{proof}
  We first prove this result for each $t \in \tau$. 
  For notational convenience we will define $\Delta \tau_j = \tau_{j+1} - \tau_j$.
  Letting $\xi_1 = (u_1,d_1)$ and $\xi_2 = (u_2,d_2)$, notice that for $j \in \{0,\ldots,\card{\tau}-1\}$, by Equation \eqref{eq:euler_traj_xi} and rearranging the terms, there exists $L' > 0$ such that:
  \begin{equation}
    \begin{aligned}
      \label{eq:dxl_ineq1}
      \bigl\| \phi_{\tau,\tau_{j+1}}(\xi_1)& - \phi_{\tau,\tau_{j+1}}(\xi_2) \bigr\|_2 
      - \bigl\| \phi_{\tau,\tau_j}(\xi_1) - \phi_{\tau,\tau_j}(\xi_2) \bigr\|_2 \leq \\
      &\leq \Delta \tau_{j} \bigl\| f\bigl(\tau_{j},\phi_{\tau,\tau_{j}}(\xi_1), u_1(\tau_j), d_1(\tau_j) \bigr)
      - f\bigl( \tau_{j}, \phi_{\tau,\tau_{j}}(\xi_2), u_2(\tau_j), d_2(\tau_j) \bigr) \bigr\|_2 \\
      &\leq \frac{L'}{2^N} \bigl\| \phi_{\tau,\tau_j}(\xi_1) - \phi_{\tau,\tau_j}(\xi_2) \bigr\|_2 
      + L' \Delta \tau_j \bigl( \| u_1(\tau_j) - u_2(\tau_j) \|_2 + \| d_1(\tau_j) - d_2(\tau_j) \|_2 \bigr),
    \end{aligned}
  \end{equation}
  where the last inequality holds since the vector field $f$ is Lipschitz in all of its arguments, as shown in the proof of Theorem \ref{thm:existence_and_uniqueness}, and $\Delta \tau_j \leq \frac{1}{2^N}$ by definition of ${\cal T}_N$.

  Summing the inequality in Equation \eqref{eq:dxl_ineq1} for $j \in \{ 0, \ldots, k-1 \}$ and noting that $\phi_{\tau,\tau_0}(\xi_1) = \phi_{\tau,\tau_0}(\xi_2)$:
  \begin{multline}
    \bigl\| \phi_{\tau,\tau_{k}}(\xi_1) - \phi_{\tau,\tau_{k}}(\xi_2) \bigr\|_2 
    \leq \frac{L'}{2^N} \sum_{j=0}^{k-1} \left\| \phi_{\tau,\tau_j}(\xi_1) - \phi_{\tau,\tau_j}(\xi_2) \right\|_2 
    + L' \sum_{j=0}^{k-1} \Delta \tau_j \left\|u_1(\tau_j) - u_2(\tau_j) \right\|_2 + \\ 
    + L' \sum_{j=0}^{k-1} \Delta \tau_j \left\|d_1(\tau_j) - d_2(\tau_j) \right\|_2.
  \end{multline}
  Using the Discrete Bellman-Gronwall Inequality (Exercise 5.6.14 in \cite{Polak1997}) and the fact that $\left(1 + \frac{L'}{2^N}\right)^\frac{L'}{2^{N}} \leq e^{L'}$,
  \begin{equation}
    \begin{aligned}
      \left\| \phi_{\tau,\tau_{k}}(\xi_1) - \phi_{\tau,\tau_{k}}(\xi_2)\right\|_2 
      &\leq L' e^{L'} \left( \sum_{j=0}^{\card{\tau}-1} \Delta \tau_j \left\| u_1(\tau_j) - u_2(\tau_j) \right\|_2 
        + \sum_{j=0}^{\card{\tau}-1} \Delta \tau_{j} \left\| d_1(\tau_j) - d_2(\tau_j) \right\|_2 \right) \\
      &\leq L' e^{L'} \left( \sqrt{ \sum_{j=0}^{\card{\tau}-1} \Delta \tau_{j} \left\| u_1(\tau_j) - u_2(\tau_j) \right\|_2^2 } 
      + \sqrt{ \sum_{j=0}^{\card{\tau}-1} \Delta \tau_{j} \left\| d_1(\tau_j) - d_2(\tau_j) \right\|_2^2 } \right) \\
      &= L \| \xi_1 - \xi_2 \|_{\cal X},
    \end{aligned}
  \end{equation}
  where $L = L' e^{L'}$, and we employed Jensen's Inequality (Equation A.2 in \cite{Mallat1999}) together the fact that the ${\cal X}$--norm of $\xi \in {\cal X}_{\tau,r}$ can be written as a finite sum.

  The result for any $t \in [0,1]$ follows by noting that $\phi_{\tau,t}(\xi)$ is a convex combination of $\phi_{\tau,\tau_k}(\xi)$ and $\phi_{\tau,\tau_{k+1}}(\xi)$ for some $k \in \{ 0, \ldots, \card{\tau}-1 \}$.
\end{proof}

As a consequence, we immediately have the following results:
\begin{corollary}
  \label{corollary:discretized_vf_lipschitz}
  There exists a constant $L > 0$ such that for each $N \in \N$, $\tau \in {\cal T}_N$, $\xi_1 = (u_1,d_1) \in {\cal X}_{\tau,r}$, $\xi_2=(u_2,d_2) \in {\cal X}_{\tau,r}$ and $t \in [0,1]$:
  \begin{enumerate_parentesis}
    \item \label{corollary:discretized_fxi_lipschitz} 
      \begin{math}
        \begin{aligned}[t]
          \bigl\| f\bigl( t, \phi_{\tau,t}(\xi_1), u_1(t), d_1(t) \bigr) 
            - f\bigl( t, \phi_{\tau,t}(\xi_2), &u_2(t), d_2(t) \bigr) \bigr\|_2 \leq \\
          &\leq L \bigl( \left\| \xi_1 - \xi_2 \right\|_{\cal X} + \left\| u_1(t) - u_2(t) \right\|_2 + \left\| d_1(t) - d_2(t) \right\|_2 \bigr),
        \end{aligned}
      \end{math}
    \item \label{corollary:discretized_dfdxxi_lipschitz} 
      \begin{math}
        \begin{aligned}[t]
          \biggl\| \frac{\partial f}{\partial x}\bigl( t, \phi_{\tau,t}(\xi_1), u_1(t), d_1(t) \bigr) 
          - \frac{\partial f}{\partial x}\bigl( t, \phi_{\tau,t}(\xi_2), &u_2(t), d_2(t) \bigr) \biggr\|_{i,2} \leq \\
          &\leq L \bigl( \left\| \xi_1 - \xi_2 \right\|_{\cal X} + \left\| u_1(t) - u_2(t) \right\|_2 + \left\| d_1(t) - d_2(t) \right\|_2 \bigr),          
        \end{aligned}
      \end{math}
    \item \label{corollary:discretized_dfduxi_lipschitz} 
      \begin{math}
        \begin{aligned}[t]
          \biggl\| \frac{\partial f}{\partial u}\bigl( t, \phi_{\tau,t}(\xi_1), u_1(t), d_1(t) \bigr) 
            - \frac{\partial f}{\partial u}\bigl( t, \phi_{\tau,t}(\xi_2), &u_2(t), d_2(t) \bigr) \biggr\|_{i,2} \leq \\
          &\leq L \bigl( \left\| \xi_1 - \xi_2 \right\|_{\cal X} + \left\| u_1(t) - u_2(t) \right\|_2 + \left\| d_1(t) - d_2(t) \right\|_2 \bigr),          
        \end{aligned}
      \end{math}
  \end{enumerate_parentesis}
  where $\phi_{\tau,t}(\xi)$ is as defined in Equation \eqref{eq:discretized_flow_xi}.
\end{corollary}
\begin{proof}
  The proof of Condition \ref{corollary:discretized_fxi_lipschitz} follows by the fact that the vector field $f$ is Lipschitz in all its arguments, as shown in the proof of Theorem \ref{thm:existence_and_uniqueness}, and applying Lemma \ref{lemma:discretized_x_lipschitz}.
  The remaining conditions follow in a similar fashion.
\end{proof}

\begin{corollary}
  \label{corollary:discretized_constraints_lipschitz}
  There exists a constant $L > 0$ such that for each $N \in \N$, $\tau \in {\cal T}_N$, $\xi_1 = (u_1,d_1) \in {\cal X}_{r,\tau}$, $\xi_2=(u_2,d_2) \in {\cal X}_{r,\tau}$, $j \in {\cal J}$, and $t \in [0,1]$:
\begin{enumerate_parentesis}
  \item \label{corollary:discretized_phixi_lipschitz} 
    $\bigl| h_0\bigl( \phi_{\tau,1}(\xi_1) \bigr) - h_0\bigl( \phi_{\tau,1}(\xi_2) \bigr) \bigr| 
    \leq L \left\| \xi_1 - \xi_2 \right\|_{\cal X}$,
  \item \label{corollary:discretized_dphidxxi_lipschitz} 
    $\Bigl\| \frac{\partial h_0}{\partial x}\bigl( \phi_{\tau,1}(\xi_1) \bigr) - \frac{\partial h_0}{\partial x}\bigl(\phi_{\tau,1}(\xi_2) \bigr) \Bigr\|_2 
    \leq L \left\| \xi_1 - \xi_2 \right\|_{\cal X}$,
  \item \label{corollary:discretized_hjxi_lipschitz} 
    $\bigl| h_j\bigl( \phi_{\tau,t}(\xi_1) \bigr) - h_j\bigl( \phi_{\tau,t}(\xi_2) \bigr) \bigr| 
    \leq L \left\| \xi_1 - \xi_2 \right\|_{\cal X}$,
  \item \label{corollary:discretized_dhjdxxi_lipschitz} 
    $\Bigl\| \frac{\partial h_j}{\partial x}\bigl( \phi_{\tau,t}(\xi_1) \bigr) - \frac{\partial h_j}{\partial x}\bigl( \phi_{\tau,t}(\xi_2) \bigr) \Bigr\|_2 
    \leq L \left\| \xi_1 - \xi_2 \right\|_{\cal X}$,
\end{enumerate_parentesis}
where $\phi_{\tau,t}(\xi)$ is as defined in Equation \eqref{eq:discretized_flow_xi}.
\end{corollary}
\begin{proof}
  This result follows by Assumption \ref{assump:constraint_fns} and Lemma \ref{lemma:discretized_x_lipschitz}.
\end{proof}

Even though it is a straightforward consequence of Condition \ref{corollary:discretized_phixi_lipschitz} in Corollary \ref{corollary:discretized_constraints_lipschitz}, we write the following result to stress its importance.
\begin{corollary}
	\label{corollary:discretized_J_continuous}
	Let $N \in \N$ and $\tau \in {\cal T}_N$, then there exists a constant $L > 0$ such that, for each $\xi_1, \xi_2 \in {\cal X}_{\tau,r}$:
  \begin{equation}
    \left\lvert J_{\tau}( \xi_1 ) - J_{\tau}( \xi_2 ) \right\rvert \leq L \left\lVert \xi_1 - \xi_2 \right\rVert_{\cal X}
  \end{equation}
  where $J_{\tau}$ is as defined in Equation \eqref{eq:discretized_cost}.
\end{corollary}

In fact, $\Psi_{\tau}$ is also Lipschitz continuous:
\begin{lemma}
	\label{lemma:discretized_psi_continuous}
  Let $N \in \N$ and $\tau \in {\cal T}_N$, then there exists a constant $L > 0$ such that, for each $\xi_1, \xi_2 \in {\cal X}_r$:
  \begin{equation}
    \left\lvert \Psi_{\tau}( \xi_1 ) - \Psi_{\tau}( \xi_2 ) \right\rvert \leq L \left\lVert \xi_1 - \xi_2 \right\rVert_{\cal X}
  \end{equation}
  where $\Psi_{\tau}$ is as defined in Equation \eqref{eq:discretized_constraint}.
\end{lemma}
\begin{proof}
  Since the maximum in $\Psi_{\tau}$ is taken over ${\cal J} \times k \in \{0,\ldots,|\tau|\}$, which is compact, and the maps $(j,k) \mapsto \psi_{\tau,j,\tau_k}(\xi)$ are continuous for each $\xi \in {\cal X}_{\tau}$, we know from Condition \ref{corollary:discretized_hjxi_lipschitz} in Corollary \ref{corollary:discretized_constraints_lipschitz} that there exists $L > 0$ such that,
  \begin{equation}
    \begin{aligned}
      \Psi_{\tau}(\xi_1) - \Psi_{\tau}(\xi_2) &= \max_{(j,k) \in {\cal J} \times  \{0,\ldots,|\tau|\}} \psi_{\tau,j,\tau_k}(\xi_1) - \max_{(j,k) \in {\cal J} \times  \{0,\ldots,|\tau|\}} \psi_{\tau,j,\tau_k}(\xi_2) \\
      &\leq \max_{(j,k) \in {\cal J} \times  \{0,\ldots,|\tau|\}} \psi_{\tau,j,\tau_k}(\xi_1) - \psi_{\tau,j,\tau_k}(\xi_2) \\
      &\leq L \left\lVert \xi_1 - \xi_2 \right\rVert_{\cal X}.
    \end{aligned}
  \end{equation}
  By reversing $\xi_1$ and $\xi_2$, and applying the same argument we get the desired result.
\end{proof}

We can now show the rate of convergence of the Euler Integration scheme:
\begin{lemma}
  \label{lemma:convergence_discretized_traj_xi}
  There exists a constant $B > 0$ such that for each $N \in \N$, $\tau \in {\cal T}_N$, $\xi \in {\cal X}_{\tau,r}$, and $t \in [0,1]$:
  \begin{equation}
    \bigl\| z^{(\xi)}_\tau(t) - x^{(\xi)}(t) \bigr\|_2 \leq \frac{B}{2^N},
  \end{equation}
  where $x^{(\xi)}$ is the solution to Differential Equation \eqref{eq:traj_xi} and $z^{(\xi)}_\tau$ is as defined in Difference Equation \eqref{eq:discretized_traj_xi}.
\end{lemma}
\begin{proof}
  Let $\xi = (u,d)$, and recall that the vector field $f$ is Lipschitz continuous in all its arguments, as shown in the proof of Theorem \ref{thm:existence_and_uniqueness}.
  By applying Picard's Lemma (Lemma 5.6.3 in \cite{Polak1997}), we have:
  \begin{equation}
    \begin{aligned}
      \bigl\| z^{(\xi)}(t) - x^{(\xi)}(t) \bigr\|_2 
      &\leq e^{L} \int_0^1 \left\| \frac{dz^{(\xi)}}{ds}(s) - f\bigl( s, z^{(\xi)}(s), u(s), d(s) \bigr) \right\|_2 ds \\
      &= e^{L} \sum_{k=0}^{\card{\tau}-1} \int_{\tau_k}^{\tau_{k+1}} \biggl\| 
      f\bigl( \tau_k, z^{(\xi)}(\tau_k), u(\tau_k), d(\tau_k) \bigr) + \\
      &\phantom{= e^{L} \sum_{k=0}^{\card{\tau}-1} \int_{\tau_k}^{\tau_{k+1}} \biggl\|{}} 
      - f\biggl( s, z^{(\xi)}(\tau_k) + \frac{s - \tau_k}{\tau_{k+1} - \tau_k} \bigl( z^{(\xi)}(\tau_{k+1}) - z^{(\xi)}(\tau_{k}) \bigr), u(\tau_k), d(\tau_k) \biggr) \biggr\|_2 ds \\
      &\leq L e^{L} \sum_{k=0}^{\card{\tau}-1} \left( 1 + \bigl\| f\bigl( \tau_k, z^{(\xi)}(\tau_k), u(\tau_k), d(\tau_k) \bigr) \bigr\|_2 \right) 
      \left( \int_{\tau_k}^{\tau_{k+1}} \lvert s - \tau_k \rvert ds \right) \\
      &\leq \frac{1}{2^N} L e^L ( 1 + C ) \sum_{k=0}^{\card{\tau}-1} ( \tau_{k+1} - \tau_k ) = \frac{B}{2^N},
    \end{aligned}
  \end{equation}
  where $C > 0$ is as defined in Condition \ref{corollary:discretized_f_bounded} in Corollary \ref{corollary:discretized_fns_bounded} and, $B = ( 1 + C ) L e^L$, and we used the fact that $\tau_{k+1} - \tau_k \leq \frac{1}{2^N}$ by definition of ${\cal T}_N$ in Equation \eqref{eq:switching_times}.
\end{proof}

Importantly we can show that we can bound the rate of convergence of this discretized cost function. We omit the proof since if follows easily using Assumption \ref{assump:constraint_fns} and Lemma \ref{lemma:convergence_discretized_traj_xi}.
\begin{lemma}
  \label{lemma:roc_cost}
  There exists a constant $B > 0$ such that for each $N \in \N$, $\tau \in {\cal T}_N$, and $\xi \in {\cal X}_{\tau,r}$:
  \begin{equation}
    \left\lvert J_{\tau}(\xi) - J(\xi) \right\rvert \leq \frac{B}{2^N},
  \end{equation}
  where $J$ is as defined in Equation \eqref{eq:cost} and $J_{\tau}$ is as defined in Equation \eqref{eq:discretized_cost}.
\end{lemma}

Similarly, we can bound the rate of convergence of this discretized constraint function.
\begin{lemma}
  \label{lemma:roc_constraint}
  There exists a constant $B > 0$ such that for each $N \in \N$, $\tau \in {\cal T}_N$, and $\xi \in {\cal X}_{\tau,r}$:
  \begin{equation}
    \left\lvert \Psi_{\tau}(\xi) - \Psi(\xi) \right\rvert \leq \frac{B}{2^N},
  \end{equation}
  where $\Psi$ is as defined in Equation \eqref{eq:super_const} and $\Psi_{\tau}$ is as defined in Equation \eqref{eq:discretized_constraint}.
\end{lemma}
\begin{proof}
  Let $C > 0$ be as defined in Condition \ref{corollary:f_bounded} in Corollary \ref{corollary:fns_bounded}, and let $L > 0$ be the Lipschitz constant as specified in Assumption \ref{assump:constraint_fns}.
  Then, using the definition of ${\cal T}_N$ in Equation \eqref{eq:switching_times}, for each $k \in \{0,\ldots,\card{\tau}-1\}$ and $t \in [\tau_k, \tau_{k+1}]$,
  \begin{equation}
    \bigl\lvert h_j\bigl( x^{(\xi)}(t) \bigr) - h_j\bigl( x^{(\xi)}( \tau_k ) \bigr) \bigr\rvert 
    \leq L \int_{\tau_k}^t \bigl\lVert f\bigl( s, x^{(\xi)}(s), u(s), d(s) \bigr) \bigr\rVert_2 ds 
    \leq \frac{LC}{2^N}.
  \end{equation}
  Moreover, Condition \ref{assump:hj_lipschitz} in Assumption \ref{assump:constraint_fns} together Lemma \ref{lemma:convergence_discretized_traj_xi} imply the existence of a constant $K > 0$ such that:
  \begin{equation}
    \bigl\lvert h_j\bigl( x^{(\xi)}( \tau_k ) \bigr) - h_j\bigl( z^{(\xi)}( \tau_k ) \bigr) \bigr\rvert 
    \leq \frac{K}{2^N}.
  \end{equation}
  Employing the Triangular Inequality on the two previous inequalities, we know there exists a constant $B > 0$ such that, for each $t \in [\tau_k, \tau_{k+1}]$,
  \begin{equation}
    \bigl\lvert h_j\bigl( x^{(\xi)}(t) \bigr) - h_j\bigl( z^{(\xi)}( \tau_k ) \bigr) \bigr\lvert \leq \frac{B}{2^N}.
  \end{equation}

  Let $t' \in \argmax_{t \in [0,1]} h_j\bigl( x^{(\xi)}(t) \bigr)$, and let $\kappa(t') \in \{ 0, \ldots, \card{\tau}-1\}$ such that $t' \in \bigl[ \tau_{\kappa(t')}, \tau_{\kappa(t')+1} \bigr]$.
  Then,
  \begin{equation}
    \max_{t \in [0,1]} h_j\bigl( x^{(\xi)}(t) \bigr) - \max_{k \in \{0,\ldots,\card{\tau}\}} h_j\bigl( z^{(\xi)}(\tau_k)\bigr) 
    \leq h_j\bigl( x^{(\xi)}(t') \bigr) - h_j\bigl( z^{(\xi)}(\tau_{\kappa(t')}) \bigr) \leq \frac{B}{2^N}.
  \end{equation}
  Similarly if $k' \in \argmax_{ k \in \{ 0, \ldots, \card{\tau} \} } h_j\bigl( z^{(\xi)}(\tau_k) \bigr)$, then
  \begin{equation}
    \max_{k \in \{0,\ldots,\card{\tau}\}} h_j\bigl( z^{(\xi)}(\tau_k) \bigr) - \max_{t \in [0,1]} h_j\bigl( x^{(\xi)}(t) \bigr) 
    \leq h_j\bigl( z^{(\xi)}(\tau_{k'}) \bigr) - h_j\bigl( x^{(\xi)}( \tau_{k'} ) \bigr) \leq \frac{B}{2^N}.
  \end{equation}
  This implies that:
  \begin{align}
    \Psi(\xi) - \Psi_{\tau}(\xi) 
    &\leq \max_{j\in {\cal J}} \left( \max_{t \in [0,1]} h_j\bigl( x^{(\xi)}(t) \bigr) 
      - \max_{k \in \{0,\ldots,\card{\tau}\}} h_j\bigl( z^{(\xi)}(\tau_k)\bigr) \right) \leq \frac{B}{2^N}, \\
    \Psi_{\tau}(\xi) - \Psi(\xi) 
    &\leq \max_{j\in {\cal J}} \left( \max_{k \in \{0,\ldots,\card{\tau}\}} h_j\bigl( z^{(\xi)}(\tau_k) \bigr) 
      - \max_{t \in [0,1]} h_j\bigl( x^{(\xi)}(t) \bigr) \right) \leq \frac{B}{2^N},
  \end{align}
  which proves the desired result.
\end{proof}

\subsection{Derivation of the Implementable Algorithm Terms}

Next, we formally derive the components of the discretized optimality function, prove the well-posedness of the discretized optimality function, and prove the convergence of the discretized optimality function to the optimality function. We begin by deriving the equivalent of Lemma \ref{lemma:dxt_definition} for our discretized formulation.

\begin{lemma}
  \label{lemma:discretized_dxt_definition}
  Let $N \in \N$, $\tau \in {\cal T}_N$, $\xi = (u,d) \in {\cal X}_{\tau,r}$, $\xi' = (u',d') \in {\cal X}_\tau$, and $\phi_{\tau,t}$ be as defined in Equation \eqref{eq:discretized_flow_xi}.
  Then, for each $k \in \{ 0, \ldots, \card{\tau} \}$, the directional derivative of $\phi_{\tau,\tau_k}$, as defined in Equation \eqref{eq:operator_D_definition}, is given by
  \begin{multline}
    \label{eq:discretized_dxt_definition}
    \D{\phi_{\tau,\tau_k}}(\xi;\xi') = 
    \sum_{j=0}^{k-1} (\tau_{j+1} - \tau_j) 
      \Phi^{(\xi)}_{\tau}(\tau_k,\tau_{j+1}) \biggl( 
      \frac{\partial f}{\partial u}\bigl( \tau_j, \phi_{\tau,\tau_j}(\xi), u(\tau_j), d(\tau_j) \bigr) u'(\tau_j) + \\
    + \sum_{i=1}^{q} f\bigl( \tau_j, \phi_{\tau,\tau_j}(\xi), u(\tau_j), e_i \bigl) d_i'(\tau_j) \biggr),
  \end{multline}
  where $\Phi_\tau^{(\xi)}(\tau_k,\tau_j)$ is the unique solution of the following matrix difference equation:
  \begin{equation}
    \label{eq:discretized_stm}
    \Phi^{(\xi)}_\tau( \tau_{k+1}, \tau_j ) = \Phi^{(\xi)}_\tau( \tau_k, \tau_j ) 
    + ( \tau_{k+1} - \tau_k ) \frac{\partial f}{\partial x}\bigl( \tau_k, \phi_{\tau,\tau_k}(\xi), u(\tau_k), d(\tau_k) \bigr)
    \Phi^{(\xi)}_\tau( \tau_k, \tau_j), \quad \Phi^{(\xi)}_\tau( \tau_j, \tau_j ) = I,
  \end{equation}
  for each $k \in \{ 0, \ldots, \card{\tau}-1 \}$.
  
\end{lemma}
\begin{proof}
  For notational convenience, let $z^{(\lambda)} = z^{(\xi + \lambda \xi')}$, $u^{(\lambda)} = u + \lambda u'$, and $d^{(\lambda)} = d + \lambda d'$. Also, let us define $\Delta z^{(\lambda)} = z^{(\lambda)} - z^{(\xi)}$, thus, for each $k \in \{0,\ldots,\card{\tau}\}$,
  \begin{equation}
    \begin{aligned}
      \Delta z^{(\lambda)}(\tau_k) 
      &= \sum_{j=0}^{k-1} (\tau_{j+1} - \tau_j) 
      \Bigl( f\bigl( \tau_j, z^{(\lambda)}(\tau_j), u^{(\lambda)}(\tau_j), d^{(\lambda)}(\tau_j) \bigr) - 
        f\bigl( \tau_j, z^{(\xi)}(\tau_j), u(\tau_j), d(\tau_j) \bigr) \Bigr) \\
      &= \sum_{j=0}^{k-1} (\tau_{j+1} - \tau_j) \Biggl( 
      \lambda \sum_{i=1}^q d'_i(\tau_j) f\bigl( \tau_j, z^{(\lambda)}(\tau_j), u^{(\lambda)}(\tau_j), e_i \bigr) + \\
      &\phantom{= \sum_{j=0}^{k-1} (\tau_{j+1} - \tau_j) \Biggl( }
      + \frac{\partial f}{\partial x}\bigl( \tau_j, z^{(\xi)}(\tau_j) + \nu_{x,j} \Delta z^{(\lambda)}(\tau_j), u^{(\lambda)}(\tau_j), d(\tau_j) \bigr) \Delta z^{(\lambda)}(\tau_j) + \\
      &\phantom{= \sum_{j=0}^{k-1} (\tau_{j+1} - \tau_j) \Biggl( }\qquad \qquad
      + \lambda \frac{\partial f}{\partial u}\bigl( \tau_j, z^{(\xi)}(\tau_j), u(\tau_j) + \nu_{u,j} \lambda u'(\tau_j), d(\tau_j) \bigr) u'(\tau_j) 
      \Biggr),
    \end{aligned}
  \end{equation}
  where $\{ \nu_{x,j} \}_{j=0}^\card{\tau} \subset [0,1]$ and $\{ \nu_{u,j} \}_{j=0}^\card{\tau} \subset [0,1]$.

  Let $\{ y(\tau_k) \}_{k=0}^\card{\tau}$ be recursively defined as follows:
  \begin{multline}
    y(\tau_{k+1}) = y(\tau_k) + ( \tau_{k+1} - \tau_k ) \Biggl( 
    \frac{\partial f}{\partial x}\bigl( \tau_k, z^{(\xi)}(\tau_k), u(\tau_k), d(\tau_k) \bigr) y(\tau_k)
    + \frac{\partial f}{\partial u}\bigl( \tau_k, z^{(\xi)}(\tau_k), u(\tau_k), d(\tau_k) \bigr) u'(\tau_k) + \\
    + \sum_{i=1}^q d'_i(\tau_k) f\bigl( \tau_k, z^{(\xi)}(\tau_k), u(\tau_k), e_i \bigr)
    \Biggr), \quad y(\tau_0) = 0.
  \end{multline}
  We want to show that $\frac{\Delta z^{(\lambda)}(\tau_k)}{\lambda} \to y(\tau_k)$ as $\lambda \downarrow 0$. Consider:
  \begin{multline}
    \left\| \frac{\partial f}{\partial x}\bigl( \tau_k, z^{(\xi)}(\tau_k), u(\tau_k), d(\tau_k) \bigr) y(\tau_k) - 
      \frac{\partial f}{\partial x}\bigl( \tau_k, z^{(\xi)}(\tau_k) + \nu_{x,k} \Delta z^{(\lambda)}(\tau_k), u^{(\lambda)}(\tau_k), d(\tau_k) \bigr) \frac{\Delta z^{(\lambda)}(\tau_k)}{\lambda} \right\|_2 \leq \\
    \leq L \left\| y(\tau_k) - \frac{ \Delta z^{(\lambda)}(\tau_k) }{\lambda} \right\|_2 
    + L \left( \bigl\| \Delta z^{(\lambda)}(\tau_k) \bigr\|_2 + \lambda \left\| u'(\tau_k) \right\|_2 \right) \left\| y(\tau_k) \right\|_2,
  \end{multline}
  which follows by Assumption \ref{assump:fns_continuity} and the Triangular Inequality. Also,
  \begin{equation}
    \left\| \left( \frac{\partial f}{\partial u}\bigl( \tau_k, z^{(\xi)}(\tau_k), u(\tau_k), d(\tau_k) \bigr) - 
      \frac{\partial f}{\partial u}\bigl( \tau_k, z^{(\xi)}(\tau_k), u(\tau_k) + \nu_{u,k} \lambda u'(\tau_k), d(\tau_k) \bigr) \right) u'(\tau_k) \right\|_2 \leq L \lambda \left\lVert u'(\tau_k) \right\rVert_2^2,
  \end{equation}
  and
  \begin{equation}
    \left\| \sum_{i=1}^q d'_i(\tau_k) \Bigl( f\bigl( \tau_k, z^{(\xi)}(\tau_k), u(\tau_k), e_i \bigr) -
       f\bigl( \tau_k, z^{(\lambda)}(\tau_k), u^{(\lambda)}(\tau_k), e_i \bigr) \Bigr) \right\|_2 \leq
     L \bigl\| \Delta z^{(\lambda)} (\tau_k) \bigr\|_2 + L \lambda \left\| u'(\tau_k) \right\|_2.
  \end{equation}
  Hence, using the Discrete Bellman-Gronwall Inequality (Lemma 5.6.14 in \cite{Polak1997}) and the inequalities above,
  \begin{multline}
    \label{eq:ddxtd_ineq1}
    \left\| y(\tau_k) - \frac{\Delta z^{(\lambda)}(\tau_k)}{\lambda} \right\|_2 
    \leq L e^L \sum_{j=0}^{k-1} ( \tau_{j+1} - \tau_j ) \biggl( 
    \left( \bigl\| \Delta z^{(\lambda)}(\tau_j) \bigr\|_2 + \lambda \left\| u'(\tau_j) \right\|_2 \right) \left\| y(\tau_j) \right\|_2 + \\
    + \lambda \left\lVert u'(\tau_j) \right\rVert_2^2 + 
    \bigl\| \Delta z^{(\lambda)} (\tau_j) \bigr\|_2 + \lambda \left\| u'(\tau_j) \right\|_2
    \biggr)
  \end{multline}
  where we used the fact that $\left( 1 + \frac{L}{2^N} \right)^\frac{L}{2^N} \leq e^L$. 
  But, by Lemma \ref{lemma:discretized_x_lipschitz}, the right-hand side of Equation \eqref{eq:ddxtd_ineq1} goes to zero as $\lambda \downarrow 0$, thus obtaining that
  \begin{equation}
    \lim_{\lambda \downarrow 0} \frac{ \Delta z^{(\lambda)}(\tau_k) }{\lambda} = y(\tau_k).
  \end{equation}
  The result of the first part of the Lemma is obtained after noting that $\D{\phi_{\tau,\tau_k}}(\xi;\xi')$ is equal to $y(\tau_k)$ for each $k \in \{ 0, \ldots, \card{\tau} \}$.

\end{proof}

Next, we prove that $\D{\phi}_{\tau,\tau_k}$ is bounded by proving that $\Phi^{\xi}$ is bounded:
\begin{corollary}
  \label{corollary:discretized_stm_bounded}
  There exists a constant $C > 0$ such that for each $N \in \N$, $\tau \in {\cal T}_N$, $\xi \in {\cal X}_{\tau,r}$, and $k,l \in \{0,\ldots,\card{\tau}\}$:
  \begin{equation}
    \bigl\| \Phi^{(\xi)}_{\tau}(\tau_k,\tau_l) \bigr\|_{i,2} \leq C,
  \end{equation}
  where $\Phi^{(\xi)}_{\tau}(\tau_k,\tau_l)$ is the solution to the Difference Equation \eqref{eq:discretized_stm}.
\end{corollary}
\begin{proof}
	This follows directly from Equation \eqref{eq:discretized_stm} and Condition \ref{corollary:discretized_f_bounded} in Corollary \ref{corollary:discretized_fns_bounded}.
\end{proof}

\begin{corollary}
  \label{corollary:discretized_dxt_bounded}
  There exists a constant $C > 0$ such that for each $N \in \N$, $\tau \in {\cal T}_N$, $\xi \in {\cal X}_{\tau,r}$, $\xi' \in {\cal X}_\tau$, and $k \in \{0,\ldots,\card{\tau}\}$:
  \begin{equation}
    \label{eq:discretized_dxt_bounded}
    \left\|\D{\phi_{\tau,\tau_k}}(\xi;\xi') \right\|_2 \leq C \left\| \xi' \right\|_{\cal X},
  \end{equation}
where $\D{\phi_{\tau,\tau_k}}(\xi;\xi')$ is as defined in Equation \eqref{eq:discretized_dxt_definition}.
\end{corollary}
\begin{proof}
	This follows by the Cauchy-Schwartz Inequality together with Corollary \ref{corollary:discretized_fns_bounded} and Corollary \ref{corollary:discretized_stm_bounded}.
\end{proof}

We now show that $\Phi^{(\xi)}_\tau$ is in fact Lipschitz continuous.
\begin{lemma}
  \label{lemma:discretized_stm_lipschitz}
  There exists a constant $L > 0$ such that for each $N \in \N$, $\tau \in {\cal T}_N$, $\xi_1,\xi_2 \in {\cal X}_{\tau,r}$, and $k,l \in \{0,\ldots,\card{\tau}\}$:
  \begin{equation}
    \label{eq:discretized_stm_lipschitz}
    \left\| \Phi^{(\xi_1)}_{\tau}(\tau_k,\tau_l) - \Phi^{(\xi_2)}_{\tau}(\tau_k,\tau_l) \right\|_{i,2} 
    \leq L \left\| \xi_1 - \xi_2 \right\|_{\cal X},
  \end{equation}
  where $\Phi^{(\xi)}_{\tau}$ is the solution to Difference Equation \eqref{eq:discretized_stm}.
\end{lemma}
\begin{proof}
  Let $\xi_1 = (u_1,d_1)$ and $\xi_2 = (u_2,d_2)$. Then, using the Triangular Inequality:
  \begin{equation}
    \begin{aligned}
      \Bigl\| \Phi^{(\xi_1)}_{\tau}(\tau_k,&\tau_l) - \Phi^{(\xi_2)}_{\tau}(\tau_k,\tau_l) \Bigr\|_{i,2} \leq \\
      &\leq \sum_{i=0}^{k-1} ( \tau_{i+1} - \tau_i ) \Biggl(
      \left\| \frac{\partial f}{\partial x}\bigl( \tau_i, z^{(\xi_2)}(\tau_i), u_2(\tau_i), d_2(\tau_i) \bigr) \right\|_{i,2}
      \left\| \Phi^{(\xi_1)}_\tau(\tau_i,\tau_j) - \Phi^{(\xi_2)}_\tau(\tau_i,\tau_j) \right\|_{i,2} + \\
      &\phantom{\leq \sum_{i=0}^{k-1}}\quad
      + \left\| \frac{\partial f}{\partial x}\bigl( \tau_i, z^{(\xi_1)}(\tau_i), u_1(\tau_i), d_1(\tau_i) \bigr) -
        \frac{\partial f}{\partial x}\bigl( \tau_i, z^{(\xi_2)}(\tau_i), u_2(\tau_i), d_2(\tau_i) \bigr) \right\|_{i,2}
      \left\| \Phi^{(\xi_1)}_\tau(\tau_i,\tau_j) \right\|_{i,2}
      \Biggr).
    \end{aligned}
  \end{equation}
  The result follows by applying Condition \ref{corollary:discretized_f_bounded} in Corollary \ref{corollary:discretized_fns_bounded}, Condition \ref{corollary:discretized_dfdxxi_lipschitz} in Corollary \ref{corollary:discretized_vf_lipschitz}, the same argument used in Equation \eqref{eq:1to2_norm_holder_trick}, and the Discrete Bellman-Gronwall Inequality (Exercise 5.6.14 in \cite{Polak1997}).
\end{proof}

A simple extension of our previous argument shows that $\D{\phi}_{\tau,\tau_k}(\xi,\cdot)$ is Lipschitz continuous with respect to its point of evaluation, $\xi$.
\begin{lemma}
  \label{lemma:discretized_dxt_lipschitz}
  There exists a constant $L > 0$ such that for each $N \in \N$, $\tau \in {\cal T}_N$, $\xi_1,\xi_2 \in {\cal X}_{\tau,r}$, $\xi' \in {\cal X}_\tau$, and $k \in \{0,\ldots,\card{\tau}\}$:
  \begin{equation}
    \label{eq:discretized_dxt_lipschitz}
    \left\| \D{\phi_{\tau,\tau_k}}(\xi_1;\xi') - \D{\phi_{\tau,\tau_k}}(\xi_2;\xi') \right\|_2 
    \leq L \left\| \xi_1 - \xi_2 \right\|_{\cal X} \left\| \xi' \right\|_{\cal X},
  \end{equation}
  where $\D{\phi_{\tau,\tau_k}}$ is as defined in Equation \eqref{eq:discretized_dxt_definition}.
\end{lemma}
\begin{proof}
  Let $\xi_1 = (u_1,d_1)$, $\xi_2 = (u_2,d_2)$, and $\xi' = (u',d')$.
  Then, applying the Triangular Inequality:
  \begin{equation}
    \begin{aligned}
      \bigl\| &\D{\phi_{\tau,\tau_k}}(\xi_1;\xi') - \D{\phi_{\tau,\tau_k}}(\xi_2;\xi') \bigr\|_2 \leq \\
      &\leq \sum_{j=0}^{k-1} ( \tau_{j+1} - \tau_j ) \Biggl(
      \left\| \Phi^{(\xi_1)}_\tau(\tau_k,\tau_{j+1}) - \Phi^{(\xi_2)}_\tau(\tau_k,\tau_{j+1}) \right\|_{i,2}
      \left\| \frac{\partial f}{\partial u}\bigl( \tau_j, z^{(\xi_1)}(\tau_j), u_1(\tau_j), d_1(\tau_j) \bigr) \right\|_{i,2} 
      \left\| u'( \tau_j ) \right\|_2 + \\
      &\phantom{\leq} + \left\| \Phi^{(\xi_2)}_\tau(\tau_k,\tau_{j+1}) \right\|_{i,2}
      \left\| \frac{\partial f}{\partial u}\bigl( \tau_j, z^{(\xi_1)}(\tau_j), u_1(\tau_j), d_1(\tau_j) \bigr) -
        \frac{\partial f}{\partial u}\bigl( \tau_j, z^{(\xi_2)}(\tau_j), u_2(\tau_j), d_2(\tau_j) \bigr) \right\|_{i,2}
      \left\| u'( \tau_j ) \right\|_2 + \\
      &\phantom{\leq} + \sum_{i=1}^q \left\| \Phi^{(\xi_1)}_\tau(\tau_k,\tau_{j+1}) - \Phi^{(\xi_2)}_\tau(\tau_k,\tau_{j+1}) \right\|_{i,2}
      \left\| f\bigl( \tau_j, z^{(\xi_1)}(\tau_j), u_1(\tau_j), e_i \bigr) \right\|_{i,2} \left\lvert d'_i( \tau_j ) \right\rvert + \\
      &\phantom{\leq + \sum_{i=1}^q} + \left\| \Phi^{(\xi_2)}_\tau(\tau_k,\tau_{j+1}) \right\|_{i,2}
      \left\| f\bigl( \tau_j, z^{(\xi_1)}(\tau_j), u_1(\tau_j), e_i \bigr) -
        f\bigl( \tau_j, z^{(\xi_2)}(\tau_j), u_2(\tau_j), e_i \bigr) \right\|_{i,2} \left\lvert d'_i( \tau_j ) \right\rvert
      \Biggr)
    \end{aligned}
  \end{equation}
  The result follows by applying Lemma \ref{lemma:discretized_stm_lipschitz}, Corollary \ref{corollary:discretized_stm_bounded}, Condition \ref{corollary:f_bounded} in Corollary \ref{corollary:fns_bounded}, Conditions \ref{corollary:discretized_fxi_lipschitz} and \ref{corollary:discretized_dfduxi_lipschitz} in Corollary \ref{corollary:discretized_vf_lipschitz}, and the same argument used in Equation \eqref{eq:1to2_norm_holder_trick}.
\end{proof}

Employing these results, we can prove that $\D{\phi_{\tau,\tau_k}}(\xi;\xi')$ converges to $\D{\phi_{\tau_k}}(\xi;\xi')$ as the discretization is increased:
\begin{lemma}
  \label{lemma:roc_dxt}
  There exists $B > 0$ such that for each $N \in \N$, $\tau \in {\cal T}_N$, $\xi \in {\cal X}_{\tau,r}$, $\xi' \in {\cal X}_\tau$ and $k \in \{0,\ldots,\card{\tau}\}$:
  \begin{equation}
    \left\| \D{\phi_{\tau_k}}(\xi;\xi') - \D{\phi_{\tau,\tau_k}}(\xi;\xi') \right\|_2 \leq \frac{B}{2^N},
  \end{equation}
  where $D{\phi_{\tau_k}}$ and $\D{\phi_{\tau,\tau_k}}$ are as defined in Equations \eqref{eq:dxt_definition} and \eqref{eq:discretized_dxt_definition}, respectively. 
\end{lemma}
\begin{proof}
  Let $\xi = (u,d)$, $\xi' = (u',d')$. 
  First, by applying the triangle inequality and noticing that the induced matrix norm is compatible, we have:
  \begin{equation}
    \begin{aligned}
      \label{eq:roc_dxt_triple_triangle}
      \bigl\| &\D{\phi_{\tau_k}}(\xi;\xi') - \D{\phi_{\tau,\tau_k}}(\xi;\xi') \bigr\|_2 \leq \\
      &\leq \sum_{j=0}^{k-1} \int_{\tau_j}^{\tau_{j+1}} 
      \Biggl( \left\| \Phi^{(\xi)}(\tau_k,s) - \Phi^{(\xi)}_\tau(\tau_k, \tau_{j+1}) \right\|_{i,2} 
      \left\| \frac{\partial f}{\partial u}\bigl( \tau_j, z^{(\xi)}(\tau_j), u(\tau_j), d(\tau_j) \bigr) \right\|_{i,2} + \\ 
      &\phantom{\leq \sum_{j=0}^{k-1}} \quad
      + \left\| \Phi^{(\xi)}(\tau_k,s) \right\|_{i,2} 
      \left\| \frac{\partial f}{\partial u}\bigl( s, x^{(\xi)}(s), u(\tau_j), d(\tau_j) \bigr) - 
        \frac{\partial f}{\partial u}\bigl( \tau_j, z^{(\xi)}(\tau_j), u(\tau_j), d(\tau_j) \bigr) \right\|_{i,2} \Biggr) 
      \left\| u'(\tau_j) \right\|_2 ds + \\
      &\phantom{\leq}
      + \sum_{j=0}^{k-1} \int_{\tau_{j}}^{\tau_{j+1}} \sum_{i=1}^{q} 
      \Biggl( \left\| \Phi^{(\xi)}(\tau_k,s) \right\|_{i,2} 
      \bigl\| f\bigl( s, x^{(\xi)}(s), u(\tau_j), e_i \bigr) - f\bigl( \tau_j, z^{(\xi)}(\tau_j), u(\tau_j), e_i \bigr) \bigr\|_2 +\\ 
      &\phantom{\leq + \sum_{j=0}^{k-1} \int_{\tau_{j}}^{\tau_{j+1}} \sum_{i=1}^{q}}
      + \left\| \Phi^{(\xi)}(\tau_k,s) - \Phi^{(\xi)}_{\tau}(\tau_k, \tau_{j+1}) \right\|_{i,2} 
      \bigl\| f\bigl( \tau_j, z^{(\xi)}(\tau_j), u(\tau_j), e_i \bigr) \bigr\|_2 \Biggr) \left| d'_i(\tau_l) \right| ds .
    \end{aligned}
  \end{equation}

  Second, let $\kappa(t) \in \{ 0, \ldots, \card{\tau} \}$ such that $t \in [ \tau_{\kappa(t)}, \tau_{\kappa(t)+1} ]$ for each $t \in [0,1]$.
  Then, there exists $K > 0$ such that
  \begin{equation}
    \label{eq:two_diff_times}
    \begin{aligned}
    \bigl\lVert x^{(\xi)}(s) - z^{(\xi)}\bigl( \tau_{\kappa(s)} \bigr) \bigr\rVert
    &\leq \bigl\lVert x^{(\xi)}(s) - z^{(\xi)}(s) \bigr\rVert + \bigl\lVert z^{(\xi)}(s) - z^{(\xi)}\bigl (\tau_{\kappa(s)} \bigr) \bigr\rVert \\
    &\leq \bigl\lVert x^{(\xi)}(s) - z^{(\xi)}(s) \bigr\rVert + \bigl( s - \tau_{\kappa(s)} \bigr) C \\
    &\leq \frac{K}{2^N},
    \end{aligned}
  \end{equation}
  where $C > 0$ is as in Condition \ref{corollary:discretized_f_bounded} in Corollary \ref{corollary:discretized_fns_bounded}, and we applied Lemma \ref{lemma:convergence_discretized_traj_xi} and the definition of ${\cal T}_N$ in Equation \eqref{eq:switching_times}.

  Third, in a fashion similar to how we defined our discretized trajectory in Equation \eqref{eq:discretized_traj_xi}, we can define a discretized state transition matrix, $\widetilde{\Phi}^{(\xi)}_{\tau}$ for each $k \in \{0,\ldots,\card{\tau}\}$ via linear interpolation on the second argument:
  \begin{equation}
    \widetilde{\Phi}^{(\xi)}_{\tau}(\tau_k,t) = 
    \sum_{j=0}^{\card{\tau}-1} \left(\Phi^{(\xi)}_{\tau}(\tau_k,\tau_j) + \frac{t - \tau_j}{\tau_{j+1}-\tau_j} 
      \left(\Phi^{(\xi)}_{\tau}(\tau_k,\tau_{j+1}) - \Phi^{(\xi)}_{\tau}(\tau_k,\tau_j) \right) \right) \pi_{\tau,j}(t).
  \end{equation}
  where $\tau_{\tau,j}$ is as defined in Equation \eqref{eq:scaling_discretization}.
  Then there exists a constant $K' > 0$ such that for each $t \in [0,1]$:
  \begin{equation}
    \label{eq:stm_two_diff_times}
    \begin{aligned}
      \left\| \Phi^{(\xi)}( \tau_k, t ) - \Phi^{(\xi)}_{\tau}\bigl( \tau_k, \tau_{\kappa(t)} \bigr) \right\|_{i,2} 
      &\leq \left\| \Phi^{(\xi)}( \tau_k, t ) - \widetilde{\Phi}^{(\xi)}_{\tau}( \tau_k, t ) \right\|_{i,2} + \left\| \widetilde{\Phi}^{(\xi)}_{\tau}( \tau_k, t ) - \Phi^{(\xi)}_{\tau}\bigl( \tau_k, \tau_{\kappa(t)} \bigr) \right\|_{i,2} \\ 
      &\leq \frac{K'}{2^N},
    \end{aligned}
  \end{equation}
  where the last inequality follows by an argument identical to the one used in the proof of Lemma \ref{lemma:convergence_discretized_traj_xi}, together with an argument identical to the one used in Equation \eqref{eq:two_diff_times}.

  Finally, the result follows from Equation \eqref{eq:roc_dxt_triple_triangle} after applying Condition \ref{corollary:discretized_f_bounded} in Corollary \ref{corollary:discretized_fns_bounded}, Corollary \ref{corollary:stm_bounded}, Conditions \ref{corollary:discretized_fxi_lipschitz} and \ref{corollary:discretized_dfduxi_lipschitz} in Corollary \ref{corollary:discretized_vf_lipschitz}, Equations \eqref{eq:two_diff_times} and \eqref{eq:stm_two_diff_times}, and the same argument as in Equation \eqref{eq:1to2_norm_holder_trick}.
\end{proof}

Next, we construct the expression for the directional derivative of the discretized cost function and prove that it is Lipschitz continuous.
\begin{lemma}
  \label{lemma:discretized_DJ_definition}
  Let $N \in \N$, $\tau \in {\cal T}_N$, $\xi \in {\cal X}_{\tau,r}$, $\xi' \in {\cal X}_\tau$, and $J_\tau$ be defined as in Equation \eqref{eq:discretized_cost}. 
  Then the directional derivative of the discretized cost $J_{\tau}$ in the $\xi'$ direction is:
  \begin{equation}
    \label{eq:discretized_DJ_definition}
    \D{J_{\tau}}(\xi;\xi') = \frac{\partial h_0}{\partial x}\bigl( \phi_{\tau,1}(\xi) \bigr) \D{\phi_{\tau,1}}(\xi;\xi').
  \end{equation}
\end{lemma}
\begin{proof}
	The result follows using the Chain Rule and Lemma \ref{lemma:discretized_dxt_definition}.
\end{proof}

\begin{corollary}
  \label{corollary:discretized_DJ_lipschitz}
  There exists a constant $L > 0$ such that for each $N \in \N$, $\tau \in {\cal T}_N$, $\xi_1,\xi_2 \in {\cal X}_{\tau,r}$, and $\xi' \in {\cal X}_\tau$:
  \begin{equation}
    \label{eq:discretized_DJ_lipschitz}
    \left| \D{J}_{\tau}(\xi_1;\xi') - \D{J}_{\tau}(\xi_2;\xi') \right| \leq L \left\|\xi_1 - \xi_2 \right\|_{\cal X}\left\|\eta\right\|_{\cal X},
  \end{equation}
  where $\D{J_{\tau}}$ is as defined in Equation \eqref{eq:discretized_DJ_definition}.
\end{corollary}
\begin{proof}
  Notice by the Triangular Inequality and the Cauchy Schwartz Inequality: 
  \begin{multline}
    \left| \D{J}_{\tau}(\xi_1;\xi') - \D{J}_{\tau}(\xi_2;\xi') \right| 
    \leq \left\| \frac{\partial h_0}{\partial x}\bigl( \phi_{\tau,1}(\xi_1) \bigr) \right\|_2 
    \left\| \D{\phi_{\tau,1}}(\xi_1;\eta) - \D{\phi_{\tau,1}}(\xi_2;\eta) \right\|_2 + \\ 
    + \left\| \frac{\partial h_0}{\partial x} \bigl( \phi_{\tau,1}(\xi_1) \bigr) 
      - \frac{\partial h_0}{\partial x}\bigl( \phi_{\tau,1}(\xi_2) \bigr) \right\|_2 
    \left\| \D{\phi_{\tau,1}}(\xi_2; \eta) \right\|_2.
  \end{multline}
  The result then follows by applying Condition \ref{corollary:discretized_phi_bounded} in Corollary \ref{corollary:discretized_fns_bounded}, Condition \ref{corollary:discretized_dphidxxi_lipschitz} in Corollary \ref{corollary:discretized_constraints_lipschitz}, Corollary \ref{corollary:discretized_dxt_bounded}, and Lemma \ref{lemma:discretized_dxt_lipschitz}.
\end{proof}

In fact, the discretized cost function converges to the original cost function as the discretization is increased:
\begin{lemma}
  \label{lemma:roc_discretized_DJ}
  There exists a constant $B > 0$ such that for each $N \in \N$, $\tau \in {\cal T}_N$, $\xi \in {\cal X}_{\tau,r}$, and $\xi' \in {\cal X}_\tau$:
  \begin{equation}
    \left| \D{J_{\tau}}(\xi;\xi') - \D{J}(\xi;\xi') \right| \leq \frac{B}{2^N},
  \end{equation}
  where $\D{J}$ is as defined in Equation \eqref{eq:DJ_definition} and $\D{J_{\tau}}$ is as defined in Equation \eqref{eq:discretized_DJ_definition}.
\end{lemma} 
\begin{proof}
  Notice by the Triangular Inequality and the Cauchy Schwartz Inequality: 
  \begin{multline}
    \left| \D{J_{\tau}}(\xi;\xi') - \D{J}(\xi;\xi') \right| 
    \leq \left\| \frac{\partial h_0}{\partial x}\bigl( \phi_1(\xi) \bigr) \right\|_2 
    \left\| \D{\phi_1}(\xi;\xi') - \D{\phi_{\tau,1}}(\xi;\xi') \right\|_2 + \\   
    + \left\| \frac{\partial h_0}{\partial x}\bigl( \phi_1(\xi) \bigr) 
      - \frac{\partial h_0}{\partial x}\bigl( \phi_{\tau,1}(\xi) \bigr) \right\|_2 
    \left\| \D{\phi_{\tau,1}}(\xi;\xi') \right\|_2.
  \end{multline}
  Then the result follows by applying Condition \ref{assump:dphidx_lipschitz} in Assumption \ref{assump:constraint_fns}, Condition \ref{corollary:phi_bounded} in Corollary \ref{corollary:fns_bounded}, Lemma \ref{lemma:convergence_discretized_traj_xi}, Lemma \ref{lemma:roc_dxt}, and Corollary \ref{corollary:discretized_dxt_bounded}.
\end{proof}

Next, we construct the expression for the directional derivative of the discretized component functions and prove that they are Lipschitz continuous.
\begin{lemma}
  \label{lemma:Dhj_discrete_definition}
  Let $N \in \N$, $\tau \in {\cal T}_N$, $\xi \in {\cal X}_{\tau,r}$, $\xi' \in {\cal X}_\tau$, $j \in {\cal J}$, and $\psi_{\tau,j,\tau_k}$ be defined as in Equation \eqref{eq:discretized_component_constraints}. 
  Then the directional derivative of each of the discretized component constraints $\psi_{\tau,j,\tau_k}$ for each $k \in \{0,\ldots,\card{\tau}\}$ in the $\xi'$ direction is:
  \begin{equation}
    \label{eq:discretized_Dhj_definition}
    \D{\psi_{\tau,j,\tau_k}}(\xi;\xi') = \frac{\partial h_j}{\partial x}\bigl( \phi_{\tau,\tau_k}(\xi) \bigr) \D{\phi_{\tau,\tau_k}}(\xi;\xi').
  \end{equation}
\end{lemma}
\begin{proof}
	The result is a direct consequence of the Chain Rule and Lemma \ref{lemma:discretized_dxt_definition}.
\end{proof}

\begin{corollary}
  \label{corollary:discretized_Dhj_lipschitz}
  There exists a constant $L > 0$ such that for each $N \in \N$, $\tau \in {\cal T}_N$, $\xi_1,\xi_2 \in {\cal X}_{\tau,r}$, $\xi' \in {\cal X}_\tau$, and $k \in \{0,\ldots,\card{\tau}\}$:
  \begin{equation}
    \label{eq:discretized_Dhj_lipschitz}
    \left\lvert \D{\psi_{\tau,j,\tau_k}}(\xi_1;\xi') - \D{\psi_{\tau,j,\tau_k}}(\xi_2;\xi') \right\rvert
    \leq L \left\| \xi_1 - \xi_2 \right\|_{\cal X} \left\| \xi' \right\|_{\cal X},
  \end{equation}
  where $\D{\psi_{\tau,j,\tau_k}}$ is as defined in Equation \eqref{eq:discretized_Dhj_definition}.
\end{corollary}
\begin{proof}
  Notice by the Triangular Inequality and the Cauchy Schwartz Inequality: 
  \begin{multline}
    \left\vert \D{\psi_{\tau,j,\tau_k}}(\xi_1;\xi') - \D{\psi_{\tau,j,\tau_k}}(\xi_2;\xi') \right\rvert 
    \leq \left\| \frac{\partial h_j}{\partial x}\bigl( \phi_{\tau,\tau_k}(\xi_1) \bigr) \right\|_2 
    \left\| \D{\phi_{\tau,\tau_k}}(\xi_1;\xi') - \D{\phi_{\tau,\tau_k}}(\xi_2;\xi') \right\|_2 + \\ 
    + \left\| \frac{\partial h_j}{\partial x}\bigl( \phi_{\tau,\tau_k}(\xi_1) \bigr) 
      - \frac{\partial h_j}{\partial x}\bigl( \phi_{\tau,\tau_k}(\xi_2) \bigr) \right\|_2 
    \left\| \D{\phi_{\tau,\tau_k}}(\xi_2;\xi') \right\|_2.
  \end{multline}
  The result then follows by applying Condition \ref{corollary:discretized_hj_bounded} in Corollary \ref{corollary:discretized_fns_bounded}, Condition \ref{corollary:discretized_dhjdxxi_lipschitz} in Corollary \ref{corollary:discretized_constraints_lipschitz}, Corollary \ref{corollary:discretized_dxt_bounded}, and Lemma \ref{lemma:discretized_dxt_lipschitz}.
\end{proof}

In fact, the discretized component constraint functions converge to the original component constraint function as the discretization is increased:
\begin{lemma}
  \label{lemma:roc_discretized_Dhj}
  There exists a constant $B > 0$ such that for each $N \in \N$, $\tau \in {\cal T}_N$, $\xi \in {\cal X}_{\tau,r}$, $\xi' \in {\cal X}_\tau$, $j \in {\cal J}$, and $k \in \{0,\ldots,\card{\tau}\}$:
  \begin{equation}
    \left\lvert \D{\psi_{\tau,j,\tau_k}}(\xi;\xi') - \D{\psi_{j,\tau_k}}(\xi;\xi') \right\rvert \leq \frac{B}{2^N},
  \end{equation}
  where $\D{\psi_{j,\tau_k}}$ is as defined in Equation \eqref{eq:Dhj_definition} and $\D{\psi_{\tau,j,\tau_k}}$ is as defined in Equation \eqref{eq:discretized_Dhj_definition}.
\end{lemma} 
\begin{proof}
  Notice by the Triangular Inequality and the Cauchy Schwartz Inequality: 
  \begin{multline}
    \left\vert \D{\psi_{\tau,j,\tau_k}}(\xi;\xi') - \D{\psi_{j,\tau_k}}(\xi;\xi') \right\rvert
    \leq \left\| \frac{\partial h_j}{\partial x}\bigl( \phi_{\tau_k}(\xi) \bigr) \right\|_2 
    \left\| \D{\phi_{\tau_k}}(\xi;\xi') - \D{\phi_{\tau,\tau_k}}(\xi;\xi') \right\|_2 + \\  
    + \left\| \frac{\partial h_j}{\partial x}\bigl( \phi_{\tau_k}(\xi) \bigr) 
      - \frac{\partial h_j}{\partial x}\bigl( \phi_{\tau,\tau_k}(\xi) \bigr) \right\|_2 
    \left\| \D{\phi_{\tau,\tau_k}}(\xi;\xi') \right\|_2.
  \end{multline}
  The result follows by applying Condition \ref{assump:dhjdx_lipschitz} in Assumption \ref{assump:constraint_fns}, Condition \ref{corollary:hj_bounded} in Corollary \ref{corollary:fns_bounded}, Lemma \ref{lemma:convergence_discretized_traj_xi}, Lemma \ref{lemma:roc_dxt}, and Corollary \ref{corollary:discretized_dxt_bounded}.
\end{proof}

Given these results, we can begin describing the properties satisfied by the discretized optimality function:
\begin{lemma}
  \label{lemma:discretized_zeta_lipschitz}
Let $N \in \N$, $\tau \in {\cal T}_N$, and $\zeta_{\tau}$ be defined as in Equation \eqref{eq:discrete_theta}.
  Then there exists a constant $L > 0$ such that, for each $\xi_1,\xi_2,\xi' \in {\cal X}_{\tau,r}$,
	\begin{equation}
		\left| \zeta_{\tau}(\xi_1,\xi') - \zeta_{\tau}(\xi_2,\xi') \right| \leq L \left\| \xi_1-\xi_2 \right\|_{\cal X}.
	\end{equation}
\end{lemma}
\begin{proof}
  Letting $\Psi^+_{\tau}(\xi) = \max\{0,\Psi_{\tau}(\xi)\}$ and $\Psi_{\tau}^-(\xi)=\max\{0,-\Psi_{\tau}(\xi)\}$, observe:
  \begin{equation}
    \zeta_{\tau}(\xi,\xi') =
    \max \left\{ \d{J}_{\tau}(\xi; \xi' - \xi) - \Psi^{+}_{\tau}(\xi), 
      \max_{ j\in{\cal J},\; k\in\{0,\ldots,|\tau|\} }\D{\psi_{\tau,j,\tau_k}}(\xi;\xi'-\xi) - \gamma \Psi^{-}_{\tau}(\xi) 
    \right\} + \| \xi' - \xi \|_{\cal X}.
  \end{equation}
  Employing Equation \eqref{eq:max_norm_trick}:
  \begin{multline}
    \bigl| \zeta_{\tau}(\xi_1,\xi') - \zeta_{\tau}(\xi_2,\xi') \bigr| 
    \leq \max \biggl\{ \bigl| \d{J}_{\tau}(\xi_1; \xi' - \xi_1) - \d{J}_{\tau}(\xi_2; \xi' - \xi_2) \bigr| 
    + \bigl| \Psi^{+}_{\tau}(\xi_2) - \Psi^{+}_{\tau}(\xi_1) \bigr|, \\ 
    \max_{ j\in{\cal J},\; k\in\{0,\ldots,|\tau|\} } \bigl| \D{\psi_{\tau,j,\tau_k}}(\xi_1;\xi'-\xi_1) - \D{\psi_{\tau,j,\tau_k}}(\xi_2;\xi'-\xi_2) \bigr| 
    + \gamma \bigl| \Psi^{-}_{\tau}(\xi_2) - \Psi^{-}_{\tau}(\xi_1) \bigr| \biggr\} 
    + \bigl| \| \xi' - \xi_1 \|_{\cal X} - \| \xi' - \xi_2 \|_{\cal X} \bigr|.
  \end{multline}
  
  We show three results that taken together with the Triangular Inequality prove the desired result. First, by applying the reverse triangle inequality:
  \begin{equation}
    \bigl| \| \xi' - \xi_1 \|_{\cal X} - \| \xi' - \xi_2 \|_{\cal X} \bigr| \leq \| \xi_1 - \xi_2 \|_{\cal X}.
  \end{equation}
  Second, 
  \begin{equation}
    \label{eq:discretized_DJ_triangle_lipschitz}
    \begin{aligned}
      \bigl| \D{J}_{\tau}(\xi_1;\xi'-\xi_1) - \D{J}_{\tau}(\xi_2;\xi'-\xi_2) \bigr| 
      &= \bigl| \D{J}_{\tau}(\xi_1;\xi'-\xi_1) - \D{J}_{\tau}(\xi_2;\xi'-\xi_1) + \D{J}_{\tau}(\xi_2;\xi_2-\xi_1) \bigr| \\
      &\leq \bigl| \D{J}_{\tau}(\xi_1;\xi') - \D{J}_{\tau}(\xi_2;\xi') \bigr| + \bigl| \D{J}_{\tau}(\xi_1;\xi_1) - \D{J}(\xi_2;\xi_1) \bigr| + \\ 
      &\qquad + \left| \frac{\partial h_0}{\partial x}\bigl( \phi_{\tau,1}(\xi_2) \bigr) \D{\phi_{\tau,1}}(\xi_2;\xi_2-\xi_1) \right| \\
      &\leq L \left\| \xi_1 - \xi_2 \right\|_{\cal X},
    \end{aligned}
  \end{equation}
  where $L > 0$ and we employed the linearity of $\D{J}_{\tau}$, Corollary \ref{corollary:discretized_DJ_lipschitz}, the fact that $\xi'$ and $\xi_1$ are bounded since $\xi',\xi_1 \in {\cal X}_{\tau,r}$, the Cauchy-Schwartz Inequality, Condition \ref{corollary:discretized_phi_bounded} in Corollary \ref{corollary:discretized_fns_bounded}, and Corollary \ref{corollary:discretized_dxt_bounded}. 
  Notice that by employing an argument identical to Equation \eqref{eq:discretized_DJ_triangle_lipschitz} and Corollary \ref{corollary:discretized_Dhj_lipschitz}, we can assume without loss of generality that $\bigl| \D{\psi_{\tau,j,\tau_k}}(\xi_1;\xi'-\xi_1) - \D{\psi_{\tau,j,\tau_k}}(\xi_2;\xi'-\xi_2) \bigr| \leq L \left\| \xi_1 - \xi_2 \right\|_{\cal X}$. 
  Finally, notice that by applying Lemma \ref{lemma:discretized_psi_continuous}, $\Psi^{+}_{\tau}(\xi)$ and $\Psi^{-}_{\tau}(\xi)$ are Lipschitz continuous.
\end{proof}

Employing these results, we can prove that $\zeta_{\tau}(\xi;\xi')$ converges to $\zeta(\xi;\xi')$ as the discretization is increased:
\begin{lemma}
  \label{lemma:roc_zeta}
  There exists a constant $B > 0$ such that for each $N \in \N$, $\tau \in {\cal T}_N$, and $\xi,\xi' \in {\cal X}_{\tau,r}$:
  \begin{equation}
    \left\lvert \zeta_{\tau}(\xi,\xi') - \zeta(\xi,\xi') \right\rvert \leq \frac{B}{2^N},
  \end{equation}
  where $\zeta$ is as defined in Equation \eqref{eq:zeta_definition} and $\zeta_\tau$ is as defined in Equation \eqref{eq:discrete_zeta}.
\end{lemma}
\begin{proof}
  Let $\Psi^{+}(\xi) = \max\{0,\Psi(\xi)\}$, $\Psi^{+}_{\tau}(\xi) = \max\{0,\Psi_{\tau}(\xi)\}$, $\Psi^{-}(\xi) = \max\{0,-\Psi(\xi)\}$, and  $\Psi^{-}_{\tau}(\xi) = \max\{0,-\Psi_{\tau}(\xi)\}$. 
  Notice that we can then write: 
  \begin{equation}
    \zeta(\xi,\xi') =
    \max \left\{ \d{J}(\xi; \xi' - \xi) - \Psi^{+}(\xi), \max_{j\in {\cal J},\; t \in [0,1]}\d{\psi_{j,t}}(\xi; \xi' - \xi) - \gamma \Psi^{-}(\xi) \right\} + \| \xi' - \xi \|_{\cal X},
  \end{equation}
  and similarly for $\zeta_{\tau}(\xi,\xi')$. 
  Employing this redefinition, notice first that by employing an argument identical to the one used in the proof of Lemma \ref{lemma:roc_constraint} we can show that there exists a $K > 0$ such that for any positive integer $N$, $\tau \in {\cal T}_N$ and $\xi \in {\cal X}_{\tau,r}$:
  \begin{equation}
    \left| \Psi^{+}_{\tau}(\xi) - \Psi^{+}(\xi) \right| \leq \frac{K}{2^{N}},\quad \text{and} \quad
    \left| \Psi^{-}_{\tau}(\xi) - \Psi^{-}(\xi) \right| \leq \frac{K}{2^{N}}. 
  \end{equation}

  Let $\kappa(t) \in \{ 0, \ldots, \card{\tau} \}$ such that $t \in [ \tau_{\kappa(t)}, \tau_{\kappa(t)+1} ]$ for each $t \in [0,1]$.
  Then there exists $K' > 0$ such that,
  \begin{equation}
    \begin{aligned}
      \label{eq:roc_zeta_time_diff_bound}
      \left\lvert \D{\psi_{j,t}}(\xi;\xi'-\xi) - \D{\psi_{j,\tau_{\kappa(t)}}}(\xi;\xi'-\xi) \right\rvert 
      &\leq \left\lVert \frac{\partial h_j}{\partial x}\bigl( \phi_t(\xi) \bigr) - 
        \frac{\partial h_j}{\partial x}\bigl( \phi_{\tau_{\kappa(t)}}(\xi) \bigr) \right\rVert_2 
      \left\lVert \D{\phi_t}(\xi;\xi'-\xi) \right\rVert_2 + \\ 
      &\phantom{\leq}\quad + \left\lVert \frac{\partial h_j}{\partial x}\bigl( \phi_{\tau_{\kappa(t)}}(\xi) \bigr) \right\rVert_2
      \left\lVert \D{\phi_t}(\xi;\xi'-\xi) - \D{\phi_{\tau_{\kappa(t)}}}(\xi;\xi'-\xi) \right\rVert_2 \\
      &\leq C' \left( \left\lVert \phi_t(\xi) - \phi_{\tau_{\kappa(t)}}(\xi) \right\rVert_2 + 
        \left\lVert \D{\phi_t}(\xi;\xi'-\xi) - \D{\phi_{\tau_{\kappa(t)}}}(\xi;\xi'-\xi) \right\rVert_2 \right) \\
      &\leq \frac{K'}{2^N},
    \end{aligned}
  \end{equation}
  where $C' > 0$ is a constant obtained after applying Corollary \ref{corollary:dxt_bounded}, Condition \ref{assump:dhjdx_lipschitz} in Assumption \ref{assump:constraint_fns}, and Condition \ref{corollary:hj_bounded} in Corollary \ref{corollary:fns_bounded}, and the last inequality follows after noting that both terms can be written as the integral of uniformly bounded functions over an interval of length smaller than $2^{-N}$.
  Thus, by the Triangular Inequality, Lemma \ref{lemma:roc_discretized_Dhj}, and Equation \eqref{eq:roc_zeta_time_diff_bound}, we know there exists $B > 0$ such that for each $t \in [0,1]$:
  \begin{equation}
    \left\lvert \D{\psi_{j,t}}(\xi;\xi'-\xi) - \D{\psi_{\tau,j,\tau_{\kappa(t)}}}(\xi;\xi'-\xi) \right\rvert \leq \frac{B}{2^N}.
  \end{equation}

  Moreover, if $t' \in \argmax_{ t \in [0,1] } \D{\psi_{j,t}}(\xi;\xi'-\xi)$, then
  \begin{equation}
    \label{eq:roc_zeta_ineq1}
    \max_{t \in [0,1]} \D{\psi_{j,t}}(\xi;\xi'-\xi) - \max_{k \in \{0,\ldots,\card{\tau}\}} \D{\psi_{\tau,j,\tau_k}}(\xi;\xi'-\xi)
    \leq \D{\psi_{j,t'}}(\xi;\xi'-\xi) - \D{\psi_{\tau,j,\tau_{\kappa(t')}}}(\xi;\xi'-\xi) \leq \frac{B}{2^N}.
  \end{equation}
  Similarly if $k' \in \argmax_{ k \in \{ 0, \ldots, \card{\tau} \} } \D{\psi_{\tau,j,\tau_k}}(\xi;\xi'-\xi)$, then
  \begin{equation}
    \label{eq:roc_zeta_ineq2}
    \max_{k \in \{0,\ldots,\card{\tau}\}} \D{\psi_{\tau,j,\tau_k}}(\xi;\xi'-\xi) - \max_{t \in [0,1]} \D{\psi_{j,t}}(\xi;\xi'-\xi)  
    \leq  \D{\psi_{\tau,j,\tau_{k'}}}(\xi;\xi'-\xi) - \D{\psi_{j,\tau_{k'}}}(\xi;\xi'-\xi) \leq \frac{B}{2^N}.
  \end{equation}
  Therefore, by Equation \eqref{eq:roc_zeta_ineq1},
  \begin{multline}
    \max_{ j\in{\cal J},\; t\in[0,1] }\D{\psi_{j,t}}(\xi;\xi'-\xi) 
    - \max_{j\in{\cal J},\; k \in \{0,\ldots,\card{\tau}\}} \D{\psi_{\tau,j,\tau_k}}(\xi;\xi'-\xi) \leq \\
    \leq \max_{j\in{\cal J}} \left(\max_{t\in[0,1]}\D{\psi_{j,t}}(\xi;\xi'-\xi) 
      -  \max_{k \in \{0,\ldots,\card{\tau}\}} \D{\psi_{\tau,j,\tau_k}}(\xi;\xi'-\xi) \right)
    \leq \frac{B}{2^N}.
  \end{multline}
  and similarly, by Equation \eqref{eq:roc_zeta_ineq2},
  \begin{multline}
    \max_{ j\in{\cal J},\; k \in \{0,\ldots,\card{\tau}\}} \D{\psi_{\tau,j,\tau_k}}(\xi;\xi'-\xi) 
    - \max_{ j\in{\cal J},\; t\in[0,1]} \D{\psi_{j,t}}(\xi;\xi'-\xi) \leq \\
    \leq \max_{ j\in{\cal J}} \left( \max_{k \in \{0,\ldots,\card{\tau}\}} \D{\psi_{\tau,j,\tau_k}}(\xi;\xi'-\xi) 
      - \max_{t\in[0,1]}\D{\psi_{j,t}}(\xi;\xi'-\xi) \right)
    \leq \frac{B}{2^N},
  \end{multline}

  Employing these results and Equation \eqref{eq:max_norm_trick}, observe that:
  \begin{multline}
    \left\lvert \zeta_{\tau}(\xi,\xi') - \zeta(\xi,\xi') \right\rvert 
    \leq \max \biggl\{ \left\lvert \D{J_{\tau}}(\xi;\xi'-\xi) - \D{J}(\xi;\xi'-\xi) \right\rvert 
    + \left\vert \Psi^+(\xi) - \Psi^+_{\tau}(\xi) \right\rvert, \\ 
    \biggl\vert \max_{ j\in{\cal J},\; k \in \{0,\ldots,\card{\tau}\} } \D{\psi_{\tau,j,\tau_k}}(\xi;\xi'-\xi) 
      - \max_{ j\in{\cal J},\; t\in[0,1] } \D{\psi_{j,t}}(\xi;\xi'-\xi) \biggr\rvert 
    + \gamma \left\lvert \Psi^{-}(\xi) - \Psi^{-}_{\tau}(\xi) \right\rvert \biggr\}.
  \end{multline}
  Finaly, applying Lemma \ref{lemma:roc_discretized_DJ} and the inequatlities above, we get our desired result.
\end{proof}

$\zeta_{\tau}$ is in fact strictly convex just like its infinite dimensional analogue, and its proof is similar to the proof of Lemma \ref{lemma:zeta_strictly_convex}, hence we omit its details.
\begin{lemma}
  Let $N \in \N$, $\tau \in {\cal T}_N$, and $\xi \in {\cal X}_{\tau,p}$. 
  Then the map $\xi' \mapsto \zeta_\tau(\xi,\xi')$, as defined in Equation \eqref{eq:discrete_zeta}, is strictly convex.
\end{lemma}

Theorem \ref{thm:discretized_zeta_unique_minimizer} is very important since it proves that $g_\tau$, as defined in Equation \eqref{eq:discrete_theta}, is a well-defined function. Its proof is a consequence of the well-known result that strictly-convex functions in finite-dimensional spaces have unique minimizers.
\begin{theorem}
  \label{thm:discretized_zeta_unique_minimizer}
  Let $N \in \N$, $\tau \in {\cal T}_N$, and $\xi \in {\cal X}_{\tau,p}$. 
  Then the map $\xi' \mapsto \zeta_\tau(\xi,\xi')$, as defined in Equation \eqref{eq:discrete_zeta}, has a unique minimizer.
\end{theorem}

Employing these results we can prove the continuity of the discretized optimality function. 
This result is not strictly required in order to prove the convergence of Algorithm \ref{algo:discrete_algo} or in order to prove that the discretized optimality function encodes local minimizers of the Discretized Relaxed Switched System Optimal Control Problem. 
However, this is a fundamental result from an implementation point of view, since in practice, a computer only produces approximate results, and continuity gives a guarantee that these approximations are at least valid in a neighborhood of the evaluation point.

\begin{lemma}
  \label{lemma:discretized_theta_cont}
  Let $N \in \N$ and $\tau \in {\cal T}_N$, then the function $\theta_{\tau}$, as defined in Equation \eqref{eq:discrete_theta}, is continuous.
\end{lemma}
\begin{proof}
  First, we show that $\theta_{\tau}$ is upper semi-continuous. 
  Consider a sequence $\{\xi_i\}_{i \in \N} \subset {\cal X}_{\tau,r}$ converging to $\xi \in {\cal X}_{\tau,r}$, and $\xi' \in {\cal X}_{\tau,r}$, such that $\theta_{\tau}(\xi) = \zeta_{\tau}(\xi,\xi')$, i.e. $\xi' = g_{\tau}(\xi)$, where $g$ is defined as in Equation \eqref{eq:discrete_theta}.
Since $\theta_{\tau}(\xi_i) \leq \zeta_{\tau}(\xi_i,\xi')$ for all $i \in \N$, 
  \begin{equation}
    \limsup_{i\to\infty}\theta_{\tau}(\xi_i)\leq\limsup_{i\to\infty}\zeta_{\tau}(\xi_i,\xi')= \zeta_{\tau}(\xi,\xi') = \theta_{\tau}(\xi),
  \end{equation}
  which proves the upper semi-continuity of $\theta_{\tau}$.  
  
  Second, we show that $\theta_{\tau}$ is lower semi-continuous. 
  Let $\{\xi'_i\}_{i\in\N} \subset {\cal X}_{\tau,r}$ such that $\theta_{\tau}(\xi_i) = \zeta_{\tau}(\xi_i,\xi'_i)$, i.e. $\xi'_i = g_{\tau}(\xi_i)$. 
  From Lemma \ref{lemma:discretized_zeta_lipschitz}, we know there exists a Lipschitz constant $L > 0$ such that for each $i \in \N$, 
  $\left| \zeta_{\tau}(\xi,\xi'_i) - \zeta_{\tau}(\xi_i,\xi'_i)\right| \leq L \left\|\xi - \xi_i \right\|_{\cal X}$. 
  Consequently,
  \begin{equation}
    \theta_{\tau}(\xi) 
    \leq \bigl( \zeta_{\tau}(\xi,\xi'_i) - \zeta_{\tau}(\xi_i,\xi'_i) \bigr) + \zeta_{\tau}(\xi_i,\xi'_i) 
    \leq L \| \xi - \xi_i \|_{\cal X} + \theta_{\tau}(\xi_i).
  \end{equation}
  Taking limits we conclude that
  \begin{equation}
    \theta_{\tau}(\xi) \leq \liminf_{i\to\infty}\theta_{\tau}(\xi_i),
  \end{equation}
  which proves the lower semi-continuity of $\theta_{\tau}$, and our desired result.
\end{proof}

Next, we prove that the local minimizers of the Discretized Relaxed Switched System Optimal Control Problem are in fact zeros of the discretized optimality function.
\begin{theorem}
  \label{thm:discretized_theta_optimality_function}
  Let $N \in \N$, $\tau \in {\cal T}_N$, and $\theta_{\tau}$ be defined as in Equation \eqref{eq:discrete_theta}, then:
  \begin{enumerate_parentesis}
  \item \label{thm:discretized_theta_negative} $\theta_{\tau}$ is non-positive valued, and
  \item \label{thm:discretized_local_min_theta_zero} If $\xi \in {\cal X}_{\tau,p}$ is a local minimizer of the Discretized Relaxed Switched System Optimal Control Problem as in Definition \ref{def:discretized_local_minimizer_Xr}, then $\theta_{\tau}(\xi) = 0$.
  \end{enumerate_parentesis}
\end{theorem}
\begin{proof}
  Notice $\zeta_{\tau}(\xi,\xi) = 0$, therefore $\theta_{\tau}(\xi) = \min_{\xi' \in {\cal X}_{\tau,r}} \zeta_{\tau}(\xi,\xi') \leq \zeta_{\tau}(\xi,\xi) = 0$. 
  This proves Condition \ref{thm:discretized_theta_negative}.
  
  To prove Condition \ref{thm:discretized_local_min_theta_zero}, we begin by making several observations. 
  Given $\xi' \in {\cal X}_{\tau,r}$ and $\lambda \in [0,1]$, using the Mean Value Theorem and Corollary \ref{corollary:discretized_DJ_lipschitz} we have that there exists $s \in (0,1)$ and $L > 0$ such that
  \begin{equation}
    \label{eq:discretized_theta_J_mvt}
    \begin{aligned}
      J_{\tau}\big( \xi + \lambda (\xi' - \xi) \big) - J_{\tau}(\xi) 
      &= \d{J_{\tau}}\bigl( \xi + s \lambda ( \xi' - \xi ); \lambda (\xi' - \xi) \bigr) \\
      &\leq \lambda \d{J_{\tau}}\bigl( \xi; \xi' - \xi \bigr) + L \lambda^2 \| \xi' - \xi \|^2_{\cal X}.
    \end{aligned}
  \end{equation}
  Letting ${\cal A}_{\tau}(\xi) = \left\{ (j,k) \in {\cal J} \times \{0,\ldots,\left|\tau\right|\} \mid \Psi_{\tau}(\xi) = h_j\big( z^{(\xi)}(\tau_k) \big) \right\}$, similar to the equation above, there exists a pair $(j,k) \in {\cal A}\bigl( \xi + \lambda (\xi' - \xi) \bigr)$ and $s \in (0,1)$ such that, using Corollary \ref{corollary:discretized_Dhj_lipschitz},
  \begin{equation}
    \label{eq:discretized_theta_psi_mvt}
    \begin{aligned}
      \Psi_{\tau}\big( \xi + \lambda (\xi' - \xi) \big) - \Psi_{\tau}(\xi) 
      &\leq \psi_{\tau,j,\tau_k}\bigl( \xi + \lambda ( \xi' - \xi ) \bigr) - \Psi(\xi) \\
      &\leq \psi_{\tau,j,\tau_k}\bigl( \xi + \lambda ( \xi' - \xi ) \bigr) - \psi_{\tau,j,\tau_k}(\xi) \\
      &= \d{\psi_{\tau,j,\tau_k}}\bigl( \xi + s \lambda ( \xi' - \xi ); \lambda ( \xi' - \xi ) \bigr) \\
      &\leq \lambda \d{\psi_{\tau,j,\tau_k}}\bigl( \xi; \xi' - \xi \bigr) + L \lambda^2 \| \xi' - \xi \|^2_{\cal X}.
    \end{aligned}
  \end{equation}
  
  We prove Condition \ref{thm:discretized_local_min_theta_zero} by contradiction. 
  That is we assume $\xi \in {\cal X}_{\tau,p}$ is a local minimizer of the Discretized Relaxed Switched System Optimal Control Problem and $\theta_{\tau}(\xi) < 0$ and show that for each $\epsilon > 0$ there exists $\hat{\xi} \in \{\bar{\xi} \in {\cal X}_{\tau,r} \mid \Psi_{\tau}(\bar{\xi}) \leq 0 \} \cap { \cal N }_{\tau,\cal X}(\xi,\epsilon)$ such that $J_{\tau}(\hat{\xi}) < J_{\tau}(\xi)$, where ${\cal N}_{\tau,\cal X}(\xi,\epsilon)$ is as defined in Equation \eqref{eq:discretized_nbhd_X}, hence arriving at a contradiction.

  Before arriving at this contradiction, we make two more observations. 
  First, notice that since $\xi \in {\cal X}_p$ is a local minimizer of the Discretized Relaxed Switched System Optimal Control Problem, $\Psi_{\tau}(\xi) \leq 0$. 
  Second, consider $g_{\tau}$ as defined in Equation \eqref{eq:discrete_theta}, which exists by Theorem \ref{thm:discretized_zeta_unique_minimizer} and notice that since $\theta_{\tau}(\xi) < 0$, $g_{\tau}(\xi) \neq \xi$. 

  Next, observe that:
  \begin{equation}
    \label{eq:discretized_theta_lm}
    \theta_{\tau}(\xi) = \max \left\{ \d{J_{\tau}}( \xi; g_{\tau}(\xi) - \xi ), 
      \max_{(j,k) \in {\cal J} \times \{0,\ldots,|\tau|\}} \D{\psi_{\tau,j,\tau_k}}( \xi; g_{\tau}(\xi) - \xi ) 
      + \gamma \Psi_{\tau}(\xi) \right\} 
    + \left\|g_{\tau}(\xi) - \xi \right\|_{\cal X} < 0.
  \end{equation}
  For each $\lambda > 0$ by using Equations \eqref{eq:discretized_theta_J_mvt} and \eqref{eq:discretized_theta_lm} we have:
  \begin{equation}
    \label{eq:discretized_theta_Jless1}
    J_{\tau}( \xi + \lambda ( g_{\tau}(\xi) - \xi ) ) - J_{\tau}(\xi) \leq \theta_{\tau}(\xi)\lambda + 4 A^2 L \lambda^{2},
  \end{equation}
  where $A = \max\big\{ \|u\|_2 + 1 \mid u \in U \big\}$ and we used the fact that $\D{J_\tau}(\xi;g(\xi)-\xi) \leq \theta_\tau(\xi)$. 
  Hence for each $\lambda \in \left( 0, \frac{-\theta_{\tau}(\xi)}{4A^2L} \right)$:
  \begin{equation}
    \label{eq:discretized_theta_Jless}
    J_{\tau}(\xi + \lambda ( g_{\tau}(\xi) - \xi )) - J_{\tau}(\xi) < 0.
  \end{equation}

  Similarly, for each $\lambda > 0$ by using Equations \eqref{eq:discretized_theta_psi_mvt} and \eqref{eq:discretized_theta_lm} we have:
  \begin{equation}
    \Psi_{\tau}( \xi + \lambda ( g_{\tau}(\xi) - \xi ) ) 
    \leq \Psi_{\tau}(\xi)  + \left(\theta_{\tau}(\xi) - \gamma \Psi_{\tau}(\xi) \right)\lambda + 4 A^2 L \lambda^{2},
  \end{equation}
  where, as in Equation \eqref{eq:discretized_theta_Jless1}, $A = \max\big\{ \|u\|_2 + 1 \mid u \in U \big\}$ and we used the fact that $\D{\psi_{\tau,j,\tau_k}}( \xi; g(\xi) - \xi ) \leq \theta_\tau(\xi)$. 
  Hence for each $\lambda \in \left( 0, \min \left\{ \frac{-\theta_{\tau}(\xi)}{4A^2L}, \frac{1}{\gamma} \right\} \right)$:
  \begin{equation}
    \label{eq:discretized_theta_psisatisfied}
    \Psi_{\tau}(\xi + \lambda ( g_{\tau}(\xi) - \xi ) ) \leq ( 1 - \gamma \lambda) \Psi_{\tau}(\xi) \leq 0.
  \end{equation}
  
  Summarizing, suppose $\xi \in {\cal X}_{\tau,p}$ is a local minimizer of the Discretized Relaxed Switched System Optimal Control Problem and $\theta_{\tau}(\xi) < 0$. 
  For each $\epsilon > 0$, by choosing any 
  \begin{equation}
    \lambda \in \left( 0, \min\left\{ \frac{-\theta_{\tau}(\xi)}{4A^2L}, \frac{1}{\gamma}, 
        \frac{\epsilon}{\left\| g_{\tau}(\xi) - \xi \right\|_{\cal X} }\right\} \right), 
  \end{equation}
  we can construct a new point $\hat{\xi} = \bigl(\xi + \lambda ( g_{\tau}(\xi) - \xi ) \bigr) \in {\cal X}_{\tau,r}$ such that $\hat{\xi} \in {\cal N}_{\tau,\cal X}(\xi,\epsilon)$ by our choice of $\lambda$, $J_{\tau}( \hat{\xi} ) < J_{\tau}(\xi)$ by Equation \eqref{eq:discretized_theta_Jless}, and $\Psi_{\tau}(\hat{\xi}) \leq 0$ by Equation \eqref{eq:discretized_theta_psisatisfied}. 
  Therefore, $\xi$ is not a local minimizer of the Discretized Relaxed Switched System Optimal Control Problem, which is a contradiction and proves Condition \ref{thm:discretized_local_min_theta_zero}.
\end{proof}

Finally, we prove that the Discretized Relaxed Switched System Optimal Control Problem consistently approximates the Switched System Optimal Control Problem:
\begin{theorem}
  \label{thm:consistent_approximation}
  Let $\{\tau_i\}_{i \in \N}$ and $\{\xi_i\}_{i \in \N}$ such that $\tau_i \in {\cal T}_i$ and $\xi_i \in {\cal X}_{\tau_i,p}$ for each $i \in \N$.
  Then 
  \begin{equation}
    \lim_{i \to \infty} \left\lvert \theta_{\tau_i}(\xi_i) - \theta(\xi_i) \right\rvert = 0,
  \end{equation}
  where $\theta$ is as defined in Equation \eqref{eq:theta_definition} and $\theta_{\tau}$ is as defined in Equation \eqref{eq:discrete_theta}. That is, the Discretized Relaxed Switched System Optimal Control Problem as defined in Equation \eqref{eq:drssocp} is a \emph{consistent approximation} of the Switched System Optimal Control Problem as defined in Equation \eqref{eq:ssocp}, where consistent approximation is defined as in Definition \ref{def:consistent_approximation}.
\end{theorem}
\begin{proof}
  First, by Lemma \ref{lemma:roc_zeta},
  \begin{equation}
    \limsup_{ i \to \infty } \theta(\xi_i) - \theta_{\tau_i}(\xi_i) 
    \leq \limsup_{ i \to \infty } \zeta\bigl( \xi_i, g(\xi_i) \bigr) - \zeta_{\tau_i}\bigr( \xi_i, g_{\tau_i}(\xi_i) \bigr)
    \leq \limsup_{ i \to \infty } \frac{B}{2^i} = 0,
  \end{equation}
  where $g$ is as defined in Equation \eqref{eq:theta_definition} and $g_\tau$ is as defined in Equation \eqref{eq:discrete_theta}.
  
  Now, by Condition \ref{lemma:dense_relaxed_discretized_space} in Lemma \ref{lemma:dense_discretized_spaces}, we know there exists a sequence $\{ \xi'_i \}_{i\in\N}$, with $\xi'_i \in {\cal X}_{\tau_i,r}$ for each $i \in \N$, such that $\lim_{i \to \infty} \xi'_i = g(\xi)$.
  Then, by Lemma \ref{lemma:roc_zeta},
  \begin{equation}
  \label{eq:theta_convergence_sup}
    \begin{aligned}
      \limsup_{i \to \infty} \theta_{\tau_i}(\xi_i) - \theta(\xi_i)
      &\leq \limsup_{i \to \infty} \zeta_{\tau_i}( \xi_i, \xi'_i ) - \zeta\bigl( \xi_i, g(\xi) \bigr) \\
      &\leq \limsup_{i \to \infty} \bigl( \zeta_{\tau_i}( \xi_i, \xi'_i) - \zeta( \xi_i, \xi'_i ) \bigr) 
      + \bigl( \zeta( \xi_i, \xi'_i) - \zeta\bigl( \xi_i, g(\xi) \bigr) \bigr) \\
      &\leq \limsup_{i \to \infty} \frac{B}{2^i} + \zeta( \xi_i, \xi'_i ) - \zeta\bigl( \xi_i, g(\xi) \bigr).
    \end{aligned}
  \end{equation}
  Employing Equation \eqref{eq:max_norm_trick}:
  \begin{multline}
    \bigl| \zeta( \xi_i, \xi'_i ) - \zeta\bigl( \xi_i, g(\xi) \bigr) \bigr| 
    \leq \max \biggl\{ \bigl| \d{J}(\xi_i; \xi'_i - \xi_i) - \d{J}(\xi_i; g(\xi) - \xi_i) \bigr|, \\ 
    \max_{ j\in{\cal J},\; t\in[0,1] } \bigl| \D{\psi_{j,t}}(\xi_i;\xi'_i - \xi_i) - \D{\psi_{j,t}}(\xi_i;g(\xi) - \xi_i) \bigr| \biggr\} 
    + \bigl| \| \xi'_i - \xi_i \|_{\cal X} - \| g(\xi) - \xi_i \|_{\cal X} \bigr|.
  \end{multline}

  Notice, that by applying the reverse triangle inequality:
  \begin{equation}
     \bigl| \| \xi'_i - \xi_i \|_{\cal X} - \| g(\xi) - \xi_i \|_{\cal X} \bigr| \leq \| \xi'_i - g(\xi) \|_{\cal X}.
  \end{equation}
  Next, notice:
  \begin{equation}
    \label{eq:theta_DJ_triangle_lipschitz}
    \begin{aligned}
      \bigl| \d{J}(\xi_i; \xi'_i - \xi_i) - \d{J}(\xi_i; g(\xi) - \xi_i) \bigr| 
      &= \bigl| \d{J}(\xi_i; \xi'_i - g(\xi)) \bigr| \\
      &= \left| \frac{\partial h_0}{\partial x}\bigl(\phi_1(\xi_i) \bigr) \D{\phi_1}(\xi_i;\xi'_i - g(\xi)) \right| \\
      &\leq L \left\| \xi'_i - g(\xi) \right\|_{\cal X},
    \end{aligned}
  \end{equation}
  where $L > 0$ and we employed the linearity of $\D{J}$, Condition \ref{corollary:phi_bounded} in Corollary \ref{corollary:fns_bounded}, and Corollary \ref{corollary:dxt_bounded}. 
  Notice that by employing an argument identical to Equation \eqref{eq:theta_DJ_triangle_lipschitz}, we can assume without loss of generality that $\bigl| \D{\psi_{j,t}}(\xi_i;\xi'_i - \xi_i) - \D{\psi_{j,t}}(\xi_i;g(\xi) - \xi_i) \bigr| \leq L \left\| \xi'_i - g(\xi) \right\|_{\cal X}$.
  Therefore: 
  \begin{equation}
    \limsup_{i\to\infty} \bigl| \zeta( \xi_i, \xi'_i ) - \zeta\bigl( \xi_i, g(\xi) \bigr) \bigr| 
    \leq 0.
  \end{equation}
  From Equation \eqref{eq:theta_convergence_sup}, we have $\limsup_{i\to\infty}\left(\theta_{\tau_i}(\xi_i) - \theta(\xi_i) \right) \leq 0$. 
  Notice that 
  \begin{equation}
    \limsup_{i\to\infty}\left| \theta_{\tau_i}(\xi_i) - \theta(\xi_i) \right| \geq \liminf_{i\to\infty}\left| \theta_{\tau_i}(\xi_i) - \theta(\xi_i) \right| \geq 0.
  \end{equation}
  Therefore combining our results, we have $\lim_{i\to\infty}\left|\theta_{\tau_i}(\xi_i) - \theta(\xi_i) \right| = 0$.
\end{proof}

\subsection{Convergence of the Implementable Algorithm}

In this subsection, we prove that the sequence of points generated by Algorithm \ref{algo:discrete_algo} converge to a point that satisfies the optimality condition. We begin by proving that the Armijo algorithm as defined in Equation \eqref{eq:discrete_mu} terminates after a finite number of steps.

\begin{lemma}
  \label{lemma:discrete_mu_upper_bound}
	Let $\alpha \in (0,1)$ and $\beta \in (0,1)$.
	For every $\delta > 0$, there exists an $M^*_\delta < \infty$ such that if  $\theta_{\tau}(\xi) \leq -\delta$ for $N \in \N$, $\tau \in {\cal T}_N$, and $\xi \in {\cal X}_{\tau,p}$, then $\mu_\tau(\xi) \leq M^*_\delta$, where $\theta_{\tau}$ is as defined in Equation \eqref{eq:discrete_theta} and $\mu_\tau$ is as defined in Equation \eqref{eq:discrete_mu}.
\end{lemma}
\begin{proof}
  Given $\xi' \in {\cal X}$ and $\lambda \in [0,1]$, using the Mean Value Theorem and Corollary \ref{corollary:discretized_DJ_lipschitz} we have that there exists $s \in (0,1)$ such that
  \begin{equation}
    \label{eq:dmuub_J_mvt}
    \begin{aligned}
      J_\tau \bigl( \xi + \lambda (\xi' - \xi) \bigr) - J_\tau(\xi) 
      &= \d{J_\tau}\bigl( \xi + s \lambda ( \xi' - \xi ); \lambda (\xi' - \xi) \bigr) \\
      &\leq \lambda \d{J_\tau}\bigl( \xi; \xi' - \xi \bigr) + L \lambda^2 \| \xi' - \xi \|^2_{\cal X}.
    \end{aligned}
  \end{equation}
  Letting ${\cal A}_\tau(\xi) = \left\{ (j,i) \in {\cal J} \times \{ 0, \ldots, \card{\tau} \} \mid \Psi_\tau(\xi) = \psi_{\tau,j,\tau_i}(\xi) \right\}$, then there exists a pair $(j,i) \in {\cal A}_\tau\bigl( \xi + \lambda (\xi' - \xi) \bigr)$ and $s \in (0,1)$ such that, using Corollary \ref{corollary:Dhj_lipschitz},
  \begin{equation}
    \label{eq:dmuub_psi_mvt}
    \begin{aligned}
      \Psi_\tau\bigl( \xi + \lambda (\xi' - \xi) \bigr) - \Psi_\tau(\xi) 
      &\leq \psi_{\tau,j,\tau_i}\bigl( \xi + \lambda ( \xi' - \xi ) \bigr) - \Psi_\tau(\xi) \\
      &\leq \psi_{\tau,j,\tau_i}\bigl( \xi + \lambda ( \xi' - \xi ) \bigr) - \psi_{\tau,j,\tau_k}(\xi) \\
      &= \d{\psi_{\tau,j,\tau_i}}\bigl( \xi + s \lambda ( \xi' - \xi ); \lambda ( \xi' - \xi ) \bigr) \\
      &\leq \lambda \d{\psi_{\tau,j,\tau_i}}\bigl( \xi; \xi' - \xi \bigr) + L \lambda^2 \| \xi' - \xi \|^2_{\cal X}.
    \end{aligned}
  \end{equation}

  Now let us assume that $\Psi_\tau(\xi) \leq 0$, and consider $g_\tau$ as defined in Equation \eqref{eq:discrete_theta}. Then
  \begin{equation}
    \label{eq:dmuub_theta}
    \theta_\tau(\xi) = \max \left\{ \d{J_\tau}( \xi; g_\tau(\xi) - \xi ), 
    \max_{(j,i) \in {\cal J} \times \{ 0, \ldots, \card{\tau} \}} \D{\psi_{\tau,j,\tau_i}}( \xi; g_\tau(\xi) - \xi ) 
    + \gamma \Psi_\tau(\xi) \right\} 
    \leq -\delta,
  \end{equation}
  and using Equation \eqref{eq:dmuub_J_mvt},
  \begin{equation}
    J_\tau\bigl( \xi + \beta^k ( g_\tau(\xi) - \xi ) \bigr) - J_\tau(\xi) - \alpha \beta^k \theta_\tau(\xi) 
    \leq - ( 1 - \alpha ) \delta \beta^k + 4 A^2 L \beta^{2k},
  \end{equation}
  where $A = \max\big\{ \|u\|_2 + 1 \mid u \in U \big\}$. 
  Hence, for each $k \in \N$ such that $\beta^k \leq \frac{ (1-\alpha) \delta }{ 4 A^2 L}$ we have that
  \begin{equation}
    \label{eq:dmuub_res_1}
    J_\tau\bigl( \xi + \beta^k ( g(\xi) - \xi ) \bigr) - J_\tau(\xi) \leq \alpha \beta^k \theta_\tau(\xi).
  \end{equation}
  Similarly, using Equations \eqref{eq:dmuub_psi_mvt} and \eqref{eq:dmuub_theta},
  \begin{equation}
    \Psi_\tau\bigl( \xi + \beta^k ( g(\xi) - \xi ) \bigr) - \Psi_\tau(\xi) 
    + \beta^k \bigl( \gamma \Psi_\tau(\xi) - \alpha \theta_\tau(\xi) \bigr)
    \leq - \delta \beta^k + 4 A^2 L \beta^{2k},
  \end{equation}
  hence for each $k \in \N$ such that $\beta^k \leq \min\left\{\frac{ (1-\alpha) \delta}{4 A^2 L}, \frac{1}{\gamma} \right\}$ we have that
  \begin{equation}
    \label{eq:dmuub_res_2}
    \Psi_\tau\bigl( \xi + \beta^k ( g(\xi) - \xi ) \bigr) - \alpha \beta^k \theta_\tau(\xi) 
    \leq \left( 1 - \beta^k \gamma \right) \Psi_\tau(\xi) \leq 0.
  \end{equation}

  If $\Psi_\tau(\xi) > 0$ then 
  \begin{equation}
    \max_{(j,i) \in {\cal J} \times \{ 0, \ldots, \card{\tau}\} } \D{\psi_{\tau,j,\tau_i}}( \xi; g_\tau(\xi) - \xi ) 
    \leq \theta_\tau(\xi) \leq -\delta,
  \end{equation}
  thus, from Equation \eqref{eq:dmuub_psi_mvt},
  \begin{equation}
    \Psi_\tau\big( \xi + \beta^k ( g_\tau(\xi) - \xi ) \big) - \Psi_\tau(\xi) - \alpha \beta^k \theta_\tau(\xi) 
    \leq - ( 1 - \alpha ) \delta \beta^k + 4 A^2 L \beta^{2k}.
  \end{equation}
  Hence, for each $k \in \N$ such that $\beta^k \leq \frac{ (1-\alpha) \delta}{4 A^2 L}$ we have that
  \begin{equation}
    \label{eq:dmuub_res_3}
    \Psi_\tau\bigl( \xi + \beta^k ( g_\tau(\xi) - \xi ) \bigr) - \Psi_\tau(\xi) \leq \alpha \beta^k \theta_\tau(\xi).    
  \end{equation}

  Finally, let 
  \begin{equation}
    M^*_\delta = 1 + \max \left\{ \log_{\beta} \left( \frac{ (1-\alpha) \delta}{4 A^2 L} \right), \log_{\beta} \left( \frac{1}{\gamma} \right) \right\},
  \end{equation}
  then from Equations \eqref{eq:dmuub_res_1}, \eqref{eq:dmuub_res_2}, and \eqref{eq:dmuub_res_3}, we get that $\mu_\tau(\xi) \leq M^*_\delta$ as desired.
\end{proof}

The proof of the following corollary follows directly from the estimates of $M^*_\delta$ in the proof of Lemma \ref{lemma:discrete_mu_upper_bound}.
\begin{corollary}
  \label{corollary:discrete_mu_upper_bound}
	Let $\alpha \in (0,1)$ and $\beta \in (0,1)$.
	There exists a $\delta_0 > 0$ and $C > 0$ such that if $\delta \in (0,\delta_0)$ and $\theta_\tau(\xi) \leq  -\delta$ for $N \in \N$, $\tau \in {\cal T}_N$, and $\xi \in {\cal X}_{\tau,p}$, then $\mu_\tau(\xi) \leq 1 + \log_\beta( C \delta )$, where $\theta_{\tau}$ is as defined in Equation \eqref{eq:discrete_theta} and $\mu_\tau$ is as defined in Equation \eqref{eq:discrete_mu}.
\end{corollary}

Next, we prove a bound between the discretized trajectory for a point in the discretized relaxed optimization space and the discretized trajectory for the same point after projection by $\rho_N$ that we use in a later argument.
\begin{lemma}
  \label{lemma:discrete_state_approx}
  Consider $\rho_N$ defined as in Equation \eqref{eq:rho} and $\sigma_N$ defined as in Equation \eqref{eq:sigmaN_def}.
  There exists $K > 0$ such that for each $N_0,N \in \N$, $\tau \in {\cal T}_{N_0}$, $\xi = (u,d) \in {\cal X}_{r,\tau}$, and $t \in [0,1]$:
  \begin{equation}
    \bigl\| \phi_{\sigma_N(\xi),t}\bigl( \rho_N(\xi) \bigr) - \phi_{\tau,t}(\xi) \bigr\|_2
    \leq K \left( \left( \frac{1}{\sqrt{2}} \right)^N \bigl( \|\xi\|_{BV} + 1 \bigr) + \left( \frac{1}{2} \right)^{N_0} \right),
  \end{equation}
where $\phi_{\tau,t}$ is as defined in Equation \eqref{eq:discretized_flow_xi}.
\end{lemma}
\begin{proof}
  We prove this argument for $t = 1$, but the argument follows identically for all $t \in [0,1]$. Using the triangular inequality we have that
  \begin{multline}
    \bigl\| \phi_{\sigma_N(\xi),1}\bigl( \rho_N(\xi) \bigr) - \phi_{\tau,1}(\xi) \bigr\|_2 
    \leq \bigl\| \phi_{\sigma_N(\xi),1} \bigl( \rho_N(\xi) \bigr) - \phi_1 \bigl( \rho_N(\xi) \bigr) \bigr\|_2 +
      \bigl\| \phi_1 \bigl( \rho_N(\xi) \bigr) - \phi_1(\xi) \bigr\|_2 + \\
    + \bigl\| \phi_1 (\xi) - \phi_{\tau,1}(\xi) \bigr\|_2.
  \end{multline}
  Thus, by Theorem \ref{thm:quality_of_approximation} and Lemma \ref{lemma:convergence_discretized_traj_xi} there exists $K_1$, $K_2$, and $K_3$ such that
  \begin{equation}
    \bigl\| \phi_{\sigma_N(\xi),1}\bigl( \rho_N(\xi) \bigr) - \phi_{\tau,1}(\xi) \bigr\|_2 
    \leq K_1 \left( \frac{1}{\sqrt{2}} \right)^N \bigl( \|\xi\|_{BV} + 1 \bigr) + \frac{K_2}{2^N} + \frac{K_3}{2^{N_0}},
  \end{equation}
  hence the result follows after organizing the constants and noting that $2^{\frac{N}{2}} \leq 2^N$ for each $N \in \N$.
\end{proof}

Using this previous lemma, we can prove that $\nu_{\tau}$ is eventually finite for all $\xi$ such that $\theta(\xi) < 0$.
\begin{lemma}
  \label{lemma:discrete_nu_upper_bound}
  Let $N_0 \in \N$, $\tau_0 \in {\cal T}_{N_0}$, and $\xi \in {\cal X}_{\tau,r}$. If $\theta(\xi) < 0$ then for each $\eta \in \N$ there exists a finite $N \geq N_0$ such that $\nu_{\sigma_N(\xi)}(\xi,N+\eta)$ is finite.
\end{lemma}
\begin{proof}
  Recall $\nu_\tau$, as defined in Equation \eqref{eq:discrete_nu}, is infinity only when the optimization problem it solves is not feasible. 
  To simplify our notation, let $\xi' \in {\cal X}_{\sigma_N(\xi),r}$ defined by $\xi' = \xi + \beta^{\mu_{\sigma_N(\xi)}(\xi)} \bigl( g_{\sigma_N(\xi)}(\xi) - \xi \bigr)$. 
  Then, using Lemma \ref{lemma:discrete_state_approx}, for $k \in \N$, $k \in [N,N+\eta]$, 
  \begin{equation}
    \begin{aligned}
      J_{\sigma_k(\xi')}\bigl( \rho_k(\xi') \bigr) - J_{\sigma_N(\xi)}(\xi') 
      &\leq L K \left( \left( \frac{1}{\sqrt{2}} \right)^k \left( \|\xi'\|_{BV} + 1 \right) + \left( \frac{1}{2} \right)^N \right) \\
      &\leq L K \left( \frac{1}{\sqrt{2}} \right)^N \left( \|\xi'\|_{BV} + 2 \right) \\
    \end{aligned}
  \end{equation}
  
  Also, from Theorem \ref{thm:consistent_approximation} we know that for $N$ large enough,
  \begin{equation}
    \frac{1}{2} \theta(\xi) \geq \theta_{\sigma_N(\xi)}(\xi).
  \end{equation}
  Thus, given $\delta > \frac{1}{2} \theta(\xi)$, there exists $N^* \in \N$ such that, for each $N \geq N^*$ and $k \in [N,N+\eta]$,
  \begin{equation}
    \begin{aligned}
      \label{eq:dnub_J}
      J_{\sigma_k(\xi')} \bigl( \rho_k(\xi') \bigr) - J_{\sigma_N(\xi)}(\xi') 
      &\leq -\bar{\alpha} \bar{\beta}^N \frac{1}{2} \theta(\xi) \\
      &\leq -\bar{\alpha} \bar{\beta}^N \theta_{\sigma_N(\xi)}(\xi). \\
    \end{aligned}
  \end{equation}
  and at the same time
  \begin{equation}
    \label{eq:dnub_betabar}
    \bar{\alpha} \bar{\beta}^N 
    \leq ( 1 - \omega ) \alpha \beta^{M_\delta^*} 
    \leq ( 1 - \omega ) \alpha \beta^{\mu_{\sigma_N}(\xi)},
  \end{equation}
  where $M_\delta^*$ is as in Lemma \ref{lemma:discrete_mu_upper_bound}.

  Similarly, given ${\cal A}_\tau(\xi) = \left\{ (j,t) \in {\cal J} \times [0,1] \mid \Psi_\tau(\xi) = \psi_{\tau,j,t}(\xi) \right\}$, let $(j,t) \in {\cal A}_{\sigma_N(\xi')}(\xi')$. Thus, for $N \geq N^*$, $k \in [N,N+\eta]$, and using Lemma \ref{lemma:discrete_state_approx},
  \begin{equation}
    \begin{aligned}
      \label{eq:dnub_psi}
      \Psi_{\sigma_k(\xi')} \bigl( \rho_k(\xi') \bigr) - \Psi_{\sigma_N(\xi)}(\xi') 
      &= \psi_{\sigma_k(\xi'),j,t} \bigl( \rho_k(\xi') \bigr) - \Psi_{\sigma_N(\xi)}(\xi') \\
      &\leq \psi_{\sigma_k(\xi'),j,t} \bigl( \rho_k(\xi') \bigr) - \psi_{\sigma_N(\xi),j,t}(\xi') \\
      &\leq L K \left( \frac{1}{\sqrt{2}} \right)^N \left( \| \xi' \|_{BV} + 2 \right) \\
      &\leq - \bar{\alpha} \bar{\beta}^N \theta_{\sigma_N(\xi)}(\xi).
    \end{aligned}
  \end{equation}

  Therefore, for $N \geq N^*$, if $\Psi_{\sigma_N(\xi)}(\xi) \leq 0$, then by Equations \eqref{eq:dnub_J}, \eqref{eq:dnub_psi}, and the inequalities from the computation of $\mu_\tau(\xi)$,
  \begin{gather}
    J_{\sigma_k(\xi')} \bigl( \rho_k(\xi') \bigr) - J_{\sigma_N(\xi)}(\xi) 
    \leq \bigl( \alpha \beta^{\mu_{\sigma_N}(\xi)} - \bar{\alpha} \bar{\beta}^N \bigr) \theta_{\sigma_N(\xi)}(\xi), \\
    \Psi_{\sigma_k(\xi')} \bigl( \rho_k(\xi') \bigr) 
    \leq \bigl( \alpha \beta^{\mu_{\sigma_N}(\xi)} - \bar{\alpha} \bar{\beta}^N \bigr) \theta_{\sigma_N(\xi)}(\xi) \leq 0,
  \end{gather}
  which together with Equation \eqref{eq:dnub_betabar} implies that the feasible set is not empty. 
  Similarly, if $\Psi_{\sigma_N(\xi)}(\xi) > 0$, by Equation \eqref{eq:dnub_psi},
  \begin{equation}
    \Psi_{\sigma_k(\xi')} \bigl( \rho_k(\xi') \bigr) - \Psi_{\sigma_N(\xi)}(\xi) 
    \leq \bigl( \alpha \beta^{\mu_{\sigma_N}(\xi)} - \bar{\alpha} \bar{\beta}^N \bigr) \theta_{\sigma_N(\xi)}(\xi),
  \end{equation}
  as desired.

  Hence for all $N \geq N^*$ the feasible sets of the optimization problems associated with $\nu_{\sigma_N(\xi)}$ are not empty, and therefore $\nu_{\sigma_N(\xi)}( \xi, N + \eta ) < \infty$.
\end{proof}

In fact, the discretization precision constructed by Algorithm \ref{algo:discrete_algo} increases arbitrarily.
\begin{lemma}
  \label{lemma:discrete_N_diverges}
  Let $\{ N_i \}_{i \in \N}$, $\{ \tau_i \}_{i \in \N}$, and $\{ \xi_i \}_{i \in \N}$ be the sequences generated by Algorithm \ref{algo:discrete_algo}.
  Then $N_i \to \infty$ as $i \to \infty$.
\end{lemma}
\begin{proof}
  Suppose that $N_i \leq N^*$ for all $i \in \N$. 
  Then, by definition of Algorithm \ref{algo:discrete_algo}, there exists $i_0 \in \N$ such that $\theta(\xi_i) \leq - \Lambda 2^{-\chi N_i} \leq - \Lambda 2^{-\chi N^*}$ and $\xi_{i+1} = \Gamma_{\tau_i}(\xi_i)$ for each $i \geq i_0$, where $\Gamma_\tau$ is defined in Equation \eqref{eq:discrete_gamma}. 
  
  Moreover, by definition of $\nu_\tau$ we have that if there exists $i_1 \geq i_0$ such that $\Psi_{\tau_{i_1}} ( \xi_{i_1} ) \leq 0$, then $\Psi_{\tau_i}( \xi_i ) \leq 0$ for each $i \geq i_1$.
  Let us assume that there exists $i_1 \geq i_0$ such that $\Psi_{\tau_{i_1}} ( \xi_{i_1} ) \leq 0$, then, using Lemma \ref{lemma:discrete_mu_upper_bound},
  \begin{equation}
    \begin{aligned}
      J_{\tau_{i+1}} ( \xi_{i+1} ) - J_{\tau_i} ( \xi_i ) 
      &\leq \bigl( \alpha \beta^{\mu_{\tau_i}(\xi_i)} - \bar{\alpha} \bar{\beta}^{\nu_{\tau_i}(\xi_i,N_i+\eta)} \bigr) \theta(\xi_i) \\
      &\leq - \omega \alpha \beta^{M^*_{\delta'}} \delta',
    \end{aligned}
  \end{equation}
  for each $i \geq i_1$, where $\delta' = \Lambda 2^{-\chi N^*}$. 
  But this implies that $J_{\tau_i}(\xi_i) \to -\infty$ as $i \to \infty$, which is a contradiction since $h_0$, and therefore $J_{\tau_i}$, is lower bounded.
  
  The argument is completely analogous in the case where the sequence is perpetually infeasible. Indeed, suppose that $\Psi_{\tau_i}( \xi_i ) > 0$ for each $i \geq i_0$, then by Lemma \ref{lemma:discrete_mu_upper_bound},
  \begin{equation}
    \begin{aligned}
      \Psi_{\tau_{i+1}} ( \xi_{i+1} ) - \Psi_{\tau_i} ( \xi_i ) 
      &\leq \bigl( \alpha \beta^{\mu_{\tau_i}(\xi_i)} - \bar{\alpha} \bar{\beta}^{\nu_{\tau_i}(\xi_i,N_i+\eta)} \bigr) \theta(\xi_i) \\
      &\leq - \omega \alpha \beta^{M^*_{\delta'}} \delta',      
    \end{aligned}
  \end{equation}
  for each $i \geq i_0$, where $\delta' = \Lambda 2^{-\chi N^*}$.
  But again this implies that $\Psi_{\tau_i}(\xi_i) \to -\infty$ as $i \to \infty$, which is a contradiction since we had assumed that $\Psi_{\tau_i}( \xi_i ) > 0$.
\end{proof}

Next, we prove that if Algorithm \ref{algo:discrete_algo} find a feasible point, then every point generated afterwards remains feasible.
\begin{lemma}
  \label{lemma:discrete_phase12}
  Let $\{ N_i \}_{i \in \N}$, $\{ \tau_i \}_{i \in \N}$, and $\{ \xi_i \}_{i \in \N}$ be the sequences generated by Algorithm \ref{algo:discrete_algo}.
  Then there exists $i_0 \in \N$ such that, if $\Psi_{\tau_{i_0}}( \xi_{i_0} ) \leq 0$, then $\Psi( \xi_i ) \leq 0$ and $\Psi_{\tau_i}( \xi_i ) \leq 0$ for each $i \geq i_0$, where $\Psi_{\tau}$ is as defined in Equation \eqref{eq:discretized_constraint}.
\end{lemma}
\begin{proof}
  Let ${\cal I} \subset \N$ be a subsequence defined by
  \begin{equation}
    \label{eq:dp12_Idef}
    {\cal I} = \left\{ i \in \N \mid 
      \theta_{\tau_i}(\xi_i) \leq - \frac{\Lambda}{2^{\chi N_i}}\ \text{and}\ 
      \nu_{\tau_i}(\xi_i,N_i + \eta) < \infty \right\}.    
  \end{equation}
  Note that, by definition of Algorithm \ref{algo:discrete_algo}, $\Psi(\xi_{i+1}) = \Psi(\xi_i)$ for each $i \notin {\cal I}$. 
  Now, for each $i \in {\cal I}$ such that $\Psi_{\tau_i}(\xi_i) \leq 0$, by definition of $\nu_\tau$ in Equation \eqref{eq:discrete_nu} together with Corollary \ref{corollary:discrete_mu_upper_bound},
  \begin{equation}
    \begin{aligned}
      \Psi_{\tau_{i+1}}(\xi_{i+1}) 
      &\leq \left( \alpha \beta^{\mu_{\tau_i}(\xi_i)} - \bar{\alpha} \bar{\beta}^{\nu_{\tau_i}(\xi_i,N_i+\eta)} \right) \theta_{\tau_i}(\xi_i) \\
      &\leq - \omega \alpha \beta^{\mu_{\tau_i}(\xi_i)} \frac{\Lambda}{ 2^{\chi N_i} } \\
      &\leq - \omega \alpha \beta C \left( \frac{\Lambda}{ 2^{\chi N_i} } \right)^2,
    \end{aligned}
  \end{equation}
  where $C > 0$.
  By Lemma \ref{lemma:roc_constraint} and the fact that $N_{i+1} \geq N_i$, we have that
  \begin{equation}
    \label{eq:dp12_psi_bound}
    \begin{aligned}
      \Psi( \xi_{i+1} ) 
      &\leq \frac{B}{2^{N_i}} - \omega \alpha \beta C \left( \frac{\Lambda}{ 2^{\chi N_i} } \right)^2 \\
      &\leq \frac{1}{ 2^{2 \chi N_i} } \left( \frac{B}{ 2^{ ( 1 - 2 \chi ) N_i } } - \omega \alpha \beta C \Lambda^2 \right).
    \end{aligned}
  \end{equation}
  Hence, if $\Psi_{\tau_{i_1}}( \xi_{i_1} ) \leq 0$ for $i_1 \in \N$ such that $N_{i_1}$ is large enough, then $\Psi( \xi_i ) \leq 0$ for each $i \geq i_1$.

  Moreover, from Equation \eqref{eq:dp12_psi_bound} we get that for each $N \geq N_i$ and each $\tau \in {\cal T}_N$,
  \begin{equation}
    \label{eq:dp12_psitau_bound}
    \begin{aligned}
      \Psi_\tau( \xi_{i+1} ) 
      &\leq \frac{1}{ 2^{2 \chi N_i} } \left( \frac{B}{ 2^{ ( 1 - 2 \chi ) N_i } } - \omega \alpha \beta C \Lambda^2 \right) + \frac{B}{2^N} \\
      &\leq \frac{1}{ 2^{2 \chi N_i} } \left( \frac{2B}{ 2^{ ( 1 - 2 \chi ) N_i } } - \omega \alpha \beta C \Lambda^2 \right).
    \end{aligned}
  \end{equation}
  Thus, if $\Psi_{\tau_{i_2}}( \xi_{i_2} ) \leq 0$ for $i_2 \in \N$ such that $N_{i_2}$ is large enough, then $\Psi_\tau( \xi_{i_2} ) \leq 0$ for each $\tau \in {\cal T}_N$ such that $N \geq N_i$. 
  But note that this is exactly the case when $i_2+k \notin {\cal I}$ for $k \in \{1,\ldots,n\}$, thus we can conclude that $\Psi_{\tau_{i_2+k}}( \xi_{i_2+k} ) \leq 0$. 
  Also note that the case of $i \in {\cal I}$ is trivially satisfied by the definition of $\nu_\tau$.

  Finally, by setting $i_0 = \max\{ i_1, i_2 \}$ we get the desired result.
\end{proof}

Next, we prove $\theta_{\tau}$ converges to zero. 
\begin{lemma}
	\label{lemma:discretized_theta_to_zero}
  Let $\{ N_i \}_{i \in \N}$, $\{ \tau_i \}_{i \in \N}$, and $\{ \xi_i \}_{i \in \N}$ be the sequences generated by Algorithm \ref{algo:discrete_algo}. 
  Then $\theta_{\tau_i}(\xi_i) \to 0$ as $i \to \infty$, where $\theta_{\tau}$ is as defined in Equation \eqref{eq:discrete_theta}.
\end{lemma}
\begin{proof}
  Let us suppose that $\lim_{i \to \infty} \theta_{\tau_i}( \xi_i ) \neq 0$. Then there exists $\delta > 0$ such that
  \begin{equation}
    \liminf_{i \to \infty} \theta_{\tau_i}( \xi_i ) < -4 \delta,
  \end{equation}
  and hence, using Theorem \ref{thm:consistent_approximation} and Lemma \ref{lemma:discrete_N_diverges}, there exists an infinite subsequence ${\cal K} \subset \N$ defined by
  \begin{equation}
    \label{eq:dconv_Kdef}
    {\cal K} = \bigl\{ i \in \N \mid 
      \theta_{\tau_i}(\xi_i) < -2 \delta \ \text{and}\ \theta(\xi_i) < - \delta \bigr\}.
  \end{equation}

  Let us define a second subsequence ${\cal I} \subset \N$ by
  \begin{equation}
    \label{eq:dconv_Idef}
    {\cal I} = \left\{ i \in \N \mid 
      \theta_{\tau_i}(\xi_i) \leq - \frac{\Lambda}{2^{\chi N_i}}\ \text{and}\ 
      \nu_{\tau_i}(\xi_i,N_i + \eta) < \infty \right\}.
  \end{equation}
  Note that by the construction of the subsequence ${\cal K}$, together with Lemma \ref{lemma:discrete_nu_upper_bound}, we get that ${\cal K} \cap {\cal I}$ is an infinite set.

  Now we analyze Algorithm \ref{algo:discrete_algo} by considering the behavior of each step as a function of its membership to each subsequence.
  First, for each $i \notin {\cal I}$, $\xi_{i+1} = \xi_i$, thus $J( \xi_{i+1} ) = J( \xi_i )$ and $\Psi( \xi_{i+1} ) = \Psi( \xi_i )$.
  Second, let $i \in {\cal I}$ such that $\Psi_{\tau_i}(\xi_i) \leq 0$, then
  \begin{equation}
    \begin{aligned}
      J_{\tau_{i+1}}(\xi_{i+1}) - J_{\tau_i}(\xi_i) 
      &\leq \left( \alpha \beta^{ \mu_{\tau_i}(\xi_i) } - \bar{\alpha} \bar{\beta}^{ \nu_{\tau_i}(\xi_i,N_i+\eta) } \right) 
            \theta_{\tau_i}( \xi_i ) \\
      &\leq - \omega \alpha \beta^{ \mu_{\tau_i}(\xi_i) } \frac{\Lambda}{2^{\chi N_i}} \\
      &\leq - \omega \alpha \beta C \left( \frac{\Lambda}{2^{\chi N_i}} \right)^2,
    \end{aligned}
  \end{equation}
  where $C > 0$ and the last inequality follows from Corollary \ref{corollary:discrete_mu_upper_bound}. 
  Recall that $N_{i+1} \geq N_i$, thus using Lemmas \ref{lemma:roc_cost} and \ref{lemma:discrete_N_diverges} we have that 
  \begin{equation}
    \begin{aligned}
      J(\xi_{i+1}) - J(\xi_i)
      &\leq \frac{2B}{2^{N_i}} - \omega \alpha \beta C \left( \frac{\Lambda}{2^{\chi N_i}} \right)^2 \\
      &\leq \frac{1}{ 2^{2 \chi N_i} } \left( \frac{2B}{ 2^{ (1 - 2 \chi) N_i } } - \omega \alpha \beta C \Lambda^2 \right),
    \end{aligned}
  \end{equation}
  and since $\chi \in \left( 0, \frac{1}{2} \right)$, we get that for $N_i$ large enough $J(\xi_{i+1}) \leq J(\xi_i)$.
  Similarly, if $\Psi_{\tau_i}(\xi_i) > 0$ then
  \begin{equation}
    \Psi(\xi_{i+1}) - \Psi(\xi_i)
    \leq \frac{1}{ 2^{2 \chi N_i} } \left( \frac{2B}{ 2^{ (1 - 2 \chi) N_i } } - \omega \alpha \beta C \Lambda^2 \right),    
  \end{equation}
  thus for $N_i$ large enough, $\Psi(\xi_{i+1}) \leq \Psi(\xi_i)$.
  Third, let $i \in {\cal K} \cap {\cal I}$ such that $\Psi_{\tau_i}(\xi_i) \leq 0$, then, by Lemma \ref{lemma:discrete_mu_upper_bound},
  \begin{equation}
    \begin{aligned}
      J_{\tau_{i+1}}(\xi_{i+1}) - J_{\tau_i}(\xi_i) 
      &\leq \left( \alpha \beta^{ \mu_{\tau_i}(\xi_i) } - \bar{\alpha} \bar{\beta}^{ \nu_{\tau_i}(\xi_i,N_i+\eta) } \right) 
            \theta_{\tau_i}( \xi_i ) \\
      &\leq - 2 \omega \alpha \beta^{ M^*_{2 \delta} } \delta,
    \end{aligned}
  \end{equation}
  thus, by Lemmas \ref{lemma:roc_cost} and \ref{lemma:discrete_N_diverges}, for $N_i$ large enough,
  \begin{equation}
    \label{eq:dconv_J_KI}
    J(\xi_{i+1}) - J(\xi_i) 
    \leq - \omega \alpha \beta^{ M^*_{2 \delta} } \delta.
  \end{equation}
  Similarly, if $\Psi_{\tau_i}(\xi_i) > 0$, using the same argument and Lemma \ref{lemma:roc_constraint}, for $N_i$ large enough,
  \begin{equation}
    \label{eq:dconv_Psi_KI}
    \Psi(\xi_{i+1}) - \Psi(\xi_i) 
    \leq - \omega \alpha \beta^{ M^*_{2 \delta} } \delta.
  \end{equation}

  Now let us assume that there exists $i_0 \in \N$ such that $N_{i_0}$ is large enough and $\Psi_{\tau_{i_0}}(\xi_{i_0}) \leq 0$. 
  Then by Lemma \ref{lemma:discrete_phase12} we get that $\Psi_{\tau_i}( \xi_i ) \leq 0$ for each $i \geq i_0$.
  But as shown above, either $i \notin {\cal K} \cap {\cal I}$ and $J(\xi_{i+1}) \leq J(\xi_i)$ or $i \in {\cal K} \cap {\cal I}$ and Equation \eqref{eq:dconv_J_KI} is satisfied, and since ${\cal K} \cap {\cal I}$ is an infinite set we get that $J(\xi_i) \to - \infty$ as $i \to \infty$, which is a contradiction as $J$ is lower bounded.

  On the other hand, if we assume that $\Psi_{\tau_i}(\xi_i) > 0$ for each $i \in \N$, then either $i \notin {\cal K} \cap {\cal I}$ and $\Psi(\xi_{i+1}) \leq \Psi(\xi_i)$ or $i \in {\cal K} \cap {\cal I}$ and Equation \eqref{eq:dconv_Psi_KI} is satisfied, thus implying that $\Psi(\xi_i) \to -\infty$ as $i \to \infty$.
  But this is a contradiction since, by Lemma \ref{lemma:roc_constraint}, this would imply that $\Psi_{\tau_i}(\xi_i) \to - \infty$ as $i \to \infty$.

  Finally, both contradictions imply that $\theta_{\tau_i}(\xi_i) \to 0$ as $i \to \infty$ as desired.
\end{proof}

In conclusion, we can prove that the sequence of points generated by Algorithm \ref{algo:discrete_algo} converges to a point that is a zero of $\theta$ or a point that satisfies our optimality condition. 
\begin{theorem}
	\label{thm:discrete_convergence}
  Let $\{ N_i \}_{i \in \N}$, $\{ \tau_i \}_{i \in \N}$, and $\{ \xi_i \}_{i \in \N}$ be the sequences generated by Algorithm \ref{algo:discrete_algo}, then 
  \begin{equation}
    \lim_{i \to \infty} \theta(\xi_i) = 0,
  \end{equation}
  where $\theta$ is as defined in Equation \eqref{eq:theta_definition}.
\end{theorem}
\begin{proof}
	This result follows immediately from Lemma \ref{lemma:discretized_theta_to_zero} after noticing that the Discretized Relaxed Switched System Optimal Control Problem is a consistent approximation of the Switched System Optimal Control Problem, as is proven in Theorem \ref{thm:consistent_approximation}, and applying Theorem \ref{thm:consistent_approximation_means_convergence}.
\end{proof}

\section{Examples}
\label{sec:examples}

\begin{table}[tbh]
	\resizebox{\textwidth}{!}{
  \begin{tabular}{|c|c|c|c|} \hline
	\parbox[c][0.75cm][c]{0cm}{} Example & Mode $1$ & Mode $2$ & Mode $3$ \\ \hline
 	\parbox[c][1.5cm][c]{0cm}{} LQR & 
	$\dot{x}(t) = Ax(t) + \begin{bmatrix} 0.9801 \\ -0.1987 \\ 0 \\ \end{bmatrix} u(t)$ &
	$\dot{x}(t) = Ax(t) + \begin{bmatrix} 0.1743 \\ 0.8601 \\ -0.4794 \\ \end{bmatrix} u(t)$ & 
	$\dot{x}(t) = Ax(t) + \begin{bmatrix} 0.0952 \\ 0.4699 \\ 0.8776 \\ \end{bmatrix} u(t)$\\ \hline
	\parbox[c][1.1cm][c]{0cm}{} Tank & 
	$\dot{x}(t) = \begin{bmatrix} 1 - \sqrt{x_1(t)} \\ \sqrt{x_1(t)} - \sqrt{x_2(t)} \end{bmatrix}$ &  
	$\dot{x}(t) = \begin{bmatrix} 2 - \sqrt{x_1(t)} \\ \sqrt{x_1(t)} - \sqrt{x_2(t)} \end{bmatrix}$	& 
	N/A \\ \hline
	\parbox[c][1.5cm][c]{0cm}{} Quadrotor & 
	$\ddot{x}(t) = \begin{bmatrix} \frac{\sin x_3(t)}{M}\left( u(t) + Mg \right) \\	\frac{\cos x_3(t)}{M}\left( u(t) + Mg \right) - g \\ 0 \\ \end{bmatrix}$ & 
	$\ddot{x}(t) = \begin{bmatrix} g\sin x_3(t) \\ g\cos x_3(t) - g \\ \frac{-L u(t)}{I} \\ \end{bmatrix}$ &
	$\ddot{x}(t) = \begin{bmatrix} g\sin x_3(t) \\	g\cos x_3(t) - g \\ \frac{L u(t)}{I} \\ \end{bmatrix}$ \\ \hline
	\parbox[c][2.75cm][c]{0cm}{} Needle & 
	$\dot{x}(t) = \begin{bmatrix}  \sin\big(x_5(t)\big) u_1(t) \\ -\cos\big(x_5(t)\big)\sin\big(x_4(t)\big)  u_1(t) \\  \cos\big(x_4(t)\big) \cos\big(x_5(t)\big)  u_1(t) \\ \kappa \cos\big(x_6(t)\big) \sec\big(x_5(t)\big) u_1(t) \\ \kappa \sin\big(x_6(t)\big)  u_1(t) \\ -\kappa \cos\big(x_6(t)\big) \tan\big(x_5(t)\big) u_1(t) \\ \end{bmatrix}$ & 
	$\dot{x}(t) = \begin{bmatrix} 0 \\ 0 \\ 0 \\ 0 \\ 0 \\ u_2(t) \\ \end{bmatrix}$ &
	 N/A \\ \hline
  \end{tabular}
}
 \caption{The dynamics of each of the modes of the switched system examples considered in Section \ref{sec:examples}. The parameters employed during the application of Algorithm \ref{algo:discrete_algo} are defined explicitly in Section \ref{sec:examples}.}
 \label{tab:dynamics} 
\end{table}

\begin{table}[t]
\resizebox{\textwidth}{!}{
  \begin{tabular}{|c|c|c|c|c|c|c|c|c|c|c|c|c|c|} \hline
 \parbox[c][0.75cm][c]{0cm}{} Example & $L(x(t),u(t),t)$ & $\phi\left(x^{(\xi)}(t_f) \right)$ & $U$ & $\gamma$ & $\alpha$ & $\beta$ & $\bar{\alpha}$ & $\bar{\beta}$ & $\Lambda$ & $\chi$ & $\omega$ & $t_0$ & $t_f$   \\ \hline
	\parbox[c][1.65cm][c]{0cm}{} LQR & $0.01 \cdot (u(t))^2$ & $\left\| \begin{bmatrix} x_1(t_f) - 1 \\ x_2(t_f) - 1 \\ x_3(t_f) - 1   \end{bmatrix} \right\|_2^2$ & $u(t) \in [-20,20] $ & $1$ & $0.1$ & $0.87$ & $0.005$ & $0.72$ & $10^{-4}$ & $\frac{1}{4}$ & $10^{-6}$ & $0$ & $2$ \\ \hline
	\parbox[c][0.9cm][c]{0cm}{} Tank & $2 \cdot \left(x_2(t) - 3\right)^2$ & $0$ & N/A & $100$ & $0.01$ & $0.75$ & $0.005$ & $0.72$ & $10^{-4}$ & $\frac{1}{4}$ &  $10^{-6}$ & $0$ & $10$ \\ \hline
	\parbox[c][1.9cm][c]{0cm}{} Quadrotor & $5\cdot(u(t))^2$ & $\left\| \begin{bmatrix} \sqrt{5} \cdot (x_1(t_f) - 6) \\ \sqrt{5}\cdot(x_2(t_f) - 1) \\ \sin\left( \frac{x_3(t_f)}{2} \right)   \end{bmatrix} \right\|_2^2 $ & $u(t) \in [0,10^{-3}]$ & $10$ & $0.01$ & $0.80$ & $5 \times 10^{-4}$ & $0.72$ & $10^{-4}$ & $\frac{1}{4}$ & $10^{-6}$ & $0$ & $7.5$ \\ \hline
	\parbox[c][1.65cm][c]{0cm}{} Needle & $0.01 \cdot \left\| \begin{bmatrix} u_1(t) \\ u_2(t) \end{bmatrix} \right\|_2^2$ & $\left\| \begin{bmatrix} x_1(t_f) + 2 \\ x_2(t_f) - 3.5 \\ x_3(t_f) - 10 \end{bmatrix} \right\|_2^2$ & $\begin{matrix}  u_1(t) \in [0,5] \\ u_2(t) \in [\frac{-\pi}{2},\frac{\pi}{2}] \end{matrix}$ & $100$ & $0.002$ & $0.72$ & $0.001$ & $0.71$ & $10^{-4}$ & $\frac{1}{4}$ & $0.05$ & $0$ & $8$  \\ \hline
  \end{tabular}
}
  \caption{The algorithmic parameters and cost function used for each of the examples during the implementation of Algorithm \ref{algo:discrete_algo}.}
  \label{tab:alg_parameters}
\end{table}

\begin{table}[t]
	\resizebox{\textwidth}{!}{
  \begin{tabular}{|c|c|c|c|c|c|c|} \hline
 \multirow{2}{*}{Example}  & Initial Continuous & Initial Discrete & Algorithm \ref{algo:discrete_algo}& Algorithm \ref{algo:discrete_algo} & MIP & MIP \\ \parbox[c][0.3cm][c]{0cm}{} 
& Input, $\forall t \in [t_0,t_f]$ & Input, $\forall t \in [t_0,t_f]$ & Computation Time & Final Cost & Computation Time & Final Cost   \\ \hline
    \parbox[c][1.5cm][c]{0cm}{} LQR	 & $u(t) = 0$ & $d(t) = \begin{bmatrix} 1 \\ 0 \\ 0 \end{bmatrix}$  & $9.827$[s] & $1.23 \times 10^{-3}$ & $753.0$[s] & $1.89 \times 10^{-3}$ \\ \hline
	\parbox[c][1.15cm][c]{0cm}{} Tank & N/A  & $d(t) = \begin{bmatrix} 1 \\ 0 \end{bmatrix}$ & $32.38[s]$ & $4.829$ & $119700[s]$ & $4.828$ \\ \hline
	\parbox[c][1.5cm][c]{0cm}{} Quadrotor & $u(t) = 5 \times 10^{-4}$ & $d(t) = \begin{bmatrix} 0.33 \\ 0.34 \\ 0.33 \end{bmatrix}$ & $8.350$[s] & $0.128$ & $2783$[s] & $0.165$ \\ \hline
	\parbox[c][1.15cm][c]{0cm}{} Needle & $u(t) = \begin{bmatrix} 0 \\ 0 \end{bmatrix}$ & $d(t) = \begin{bmatrix} 0.5 \\ 0.5 \end{bmatrix}$ & $62.76$[s] & $0.302$ & did not converge & did not converge \\ \hline
  \end{tabular}
}
  \caption{The initialization parameters used for each of the examples during the implementation of Algorithm \ref{algo:discrete_algo} and the MIP described in \cite{fletcher1994solving}, and the computation time and the result for each of the examples as a result of the application of Algorithm \ref{algo:discrete_algo} and the MIP described in \cite{fletcher1994solving}.}
  \label{tab:alg_implement}
\end{table}

In this section, we apply Algorithm \ref{algo:discrete_algo} to calculate an optimal control for four examples. Before describing each example, we begin by describing the numerical implementation of Algorithm \ref{algo:discrete_algo}. First, observe that the analysis presented thus far does not require that the initial and final times of the trajectory of switched system be fixed to $0$ and $1$, respectively. Instead, the initial and final times of the trajectory of the switched system are treated as fixed parameters $t_0$ and $t_f$, respectively. Second, we employ a MATLAB implementation of LSSOL from TOMLAB in order to compute the optimality function at each iteration of the algorithm since it is a quadratic program \cite{holmstrom1999tomlab}. Third, for each example we employ a stopping criterion that terminates Algorithm \ref{algo:discrete_algo}, if $\theta_{\tau}$ becomes too large. Each of these stopping criteria is described when we describe each example. Next, for the sake of comparison we compare the performance of Algorithm \ref{algo:discrete_algo} on each of the examples to a traditional Mixed Integer Program (MIP). To perform this comparison, we employ a TOMLAB implementation of a MIP described in \cite{fletcher1994solving} which mixes branch and bound steps with sequential quadratic programming steps. Finally, all of our comparisons are performed on an Intel Xeon, $6$ core, $3.47$ GHz, $100$ GB RAM machine.

\subsection{Constrained Switched Linear Quadratic Regulator (LQR)}

\begin{figure}[tp]
	\centering
	\begin{minipage}{\columnwidth}
		\subfloat[MIP Final Result]{\includegraphics[clip, trim = 3cm 0.5cm 3cm 0.5cm, width = .5\textwidth, keepaspectratio = true]{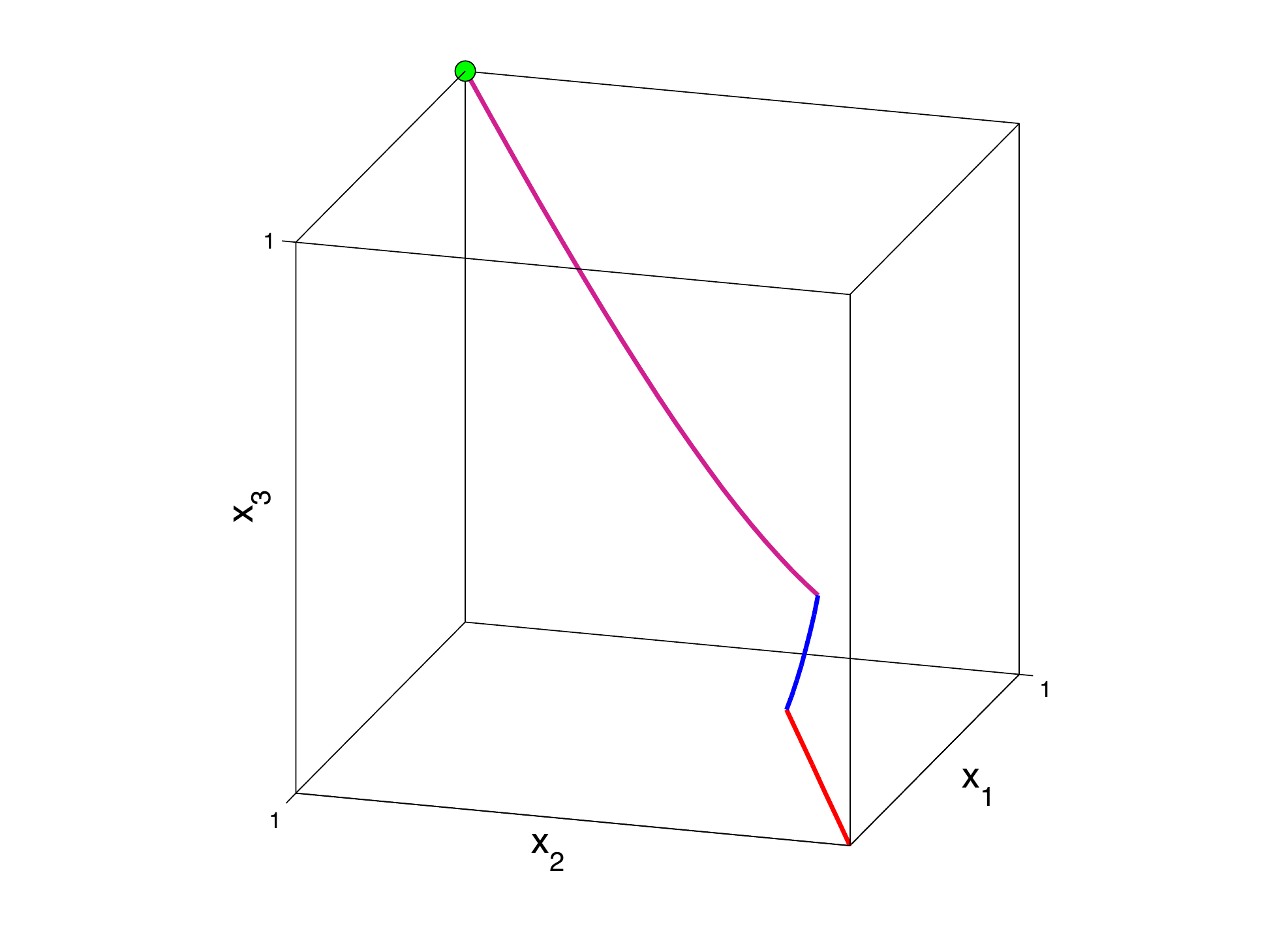}}
		\subfloat[Algorithm \ref{algo:discrete_algo} Final Result]{\includegraphics[clip, trim = 3cm 0.5cm 3cm 0.5cm, width = .5\textwidth, keepaspectratio = true]{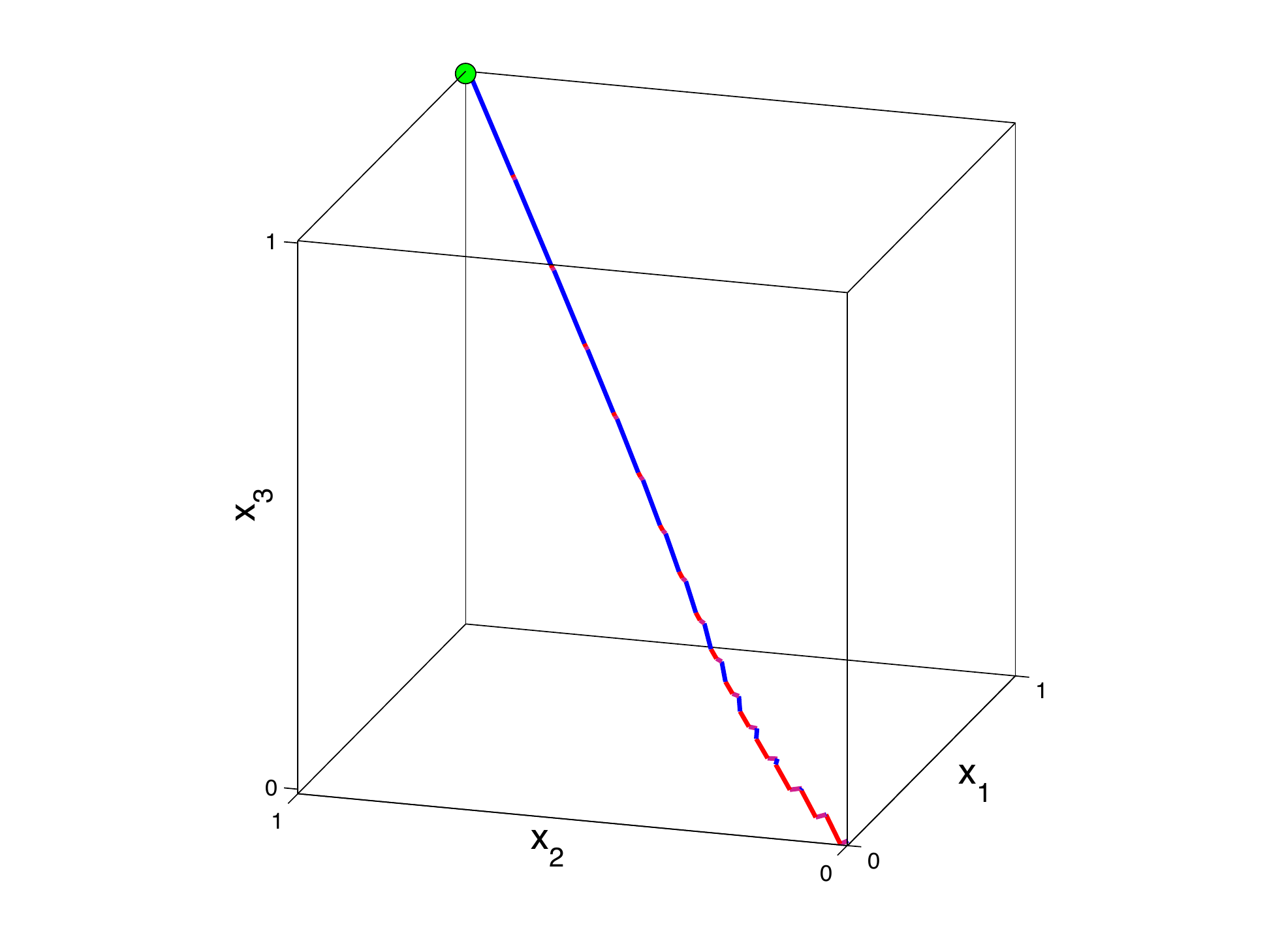}}
	\end{minipage}
	\caption{Optimal trajectories for each of the considered optimization algorithms where the point $(1,1,1)$ is drawn in green, and where the trajectory is drawn in blue when in mode $1$, in purple when in mode $2$, and in red when in mode $3$.}
	\label{fig:linear3dexample}	
\end{figure}

Switched Linear Quadratic Regulator (LQR) examples have been used to illustrate the utility of a variety of proposed optimal control algorithms \cite{Egerstedt2006,xu2002}. We consider an LQR system in three dimensions, with three discrete modes, and a single continuous input. The dynamics in each mode are as described in Table \ref{tab:dynamics} where: 
\begin{equation}
	A  = \begin{bmatrix} 1.0979   & -0.0105 & 0.0167  \\ -0.0105  & 1.0481 & 0.0825 \\ 0.0167 & 0.0825 & 1.1540 \\ \end{bmatrix}.
\end{equation}
The system matrix is purposefully chosen to have three unstable eigenvalues and the control matrix in each mode is only able to control along single dimension. Hence, while the system and control matrix in each mode is not a stabilizable pair, the system and all the control matrices taken together simultaneously is stabilizable and is expected to appropriately switch between the modes to reduce the cost. The objective of the optimization is to have the trajectory of the system at time $t_f$ be at $(1,1,1)$ while minimizing the input required to achieve this task. This objective is reflected in the chosen cost function which is described in Table \ref{tab:alg_parameters}.

Algorithm \ref{algo:discrete_algo} and the MIP are initialized at $x_0 = (0,0,0)$ with continuous and discrete inputs as described in Table \ref{tab:alg_implement} with $16$ equally spaced samples in time. Algorithm \ref{algo:discrete_algo} took $11$ iterations, ended with $48$ time samples, and terminated after the optimality condition was bigger than $-10^{-2}$. The result of both optimization procedures is illustrated in Figure \ref{fig:linear3dexample}. The computation time and final cost of both algorithms can be found in Table \ref{tab:alg_implement}. Notice that Algorithm \ref{algo:discrete_algo} is able to compute a lower cost continuous and discrete input when compared to the MIP and is able to do it more than $75$ times faster.


\subsection{Double Tank System}

\begin{figure}[tp]
	\centering
	\begin{minipage}{\columnwidth}
		\subfloat[MIP Final Result]{\includegraphics[clip, trim = 1cm 0cm 1cm 0cm, width = .5\textwidth, keepaspectratio = true]{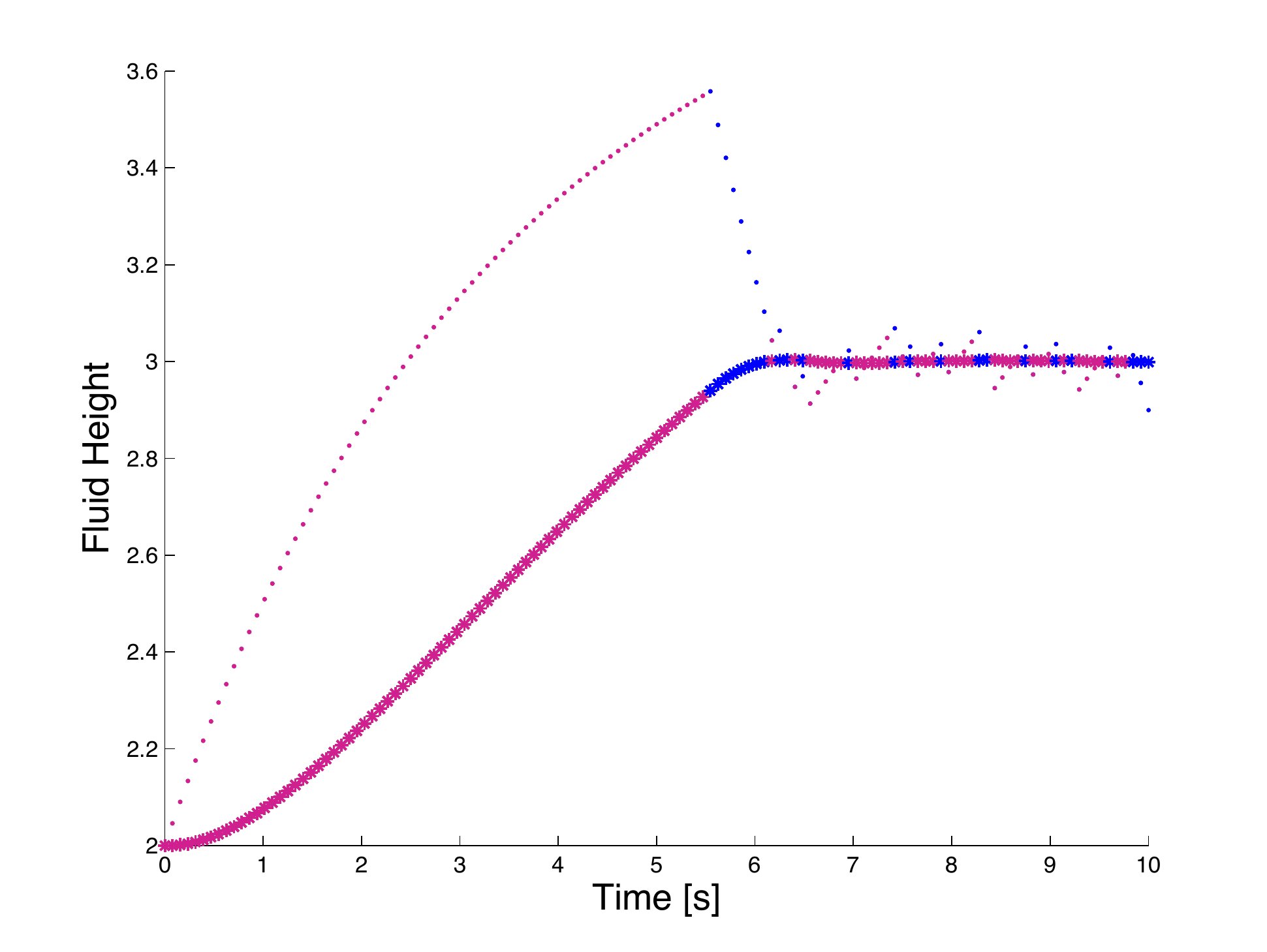}}
		\subfloat[Algorithm \ref{algo:discrete_algo} Final Result]{\includegraphics[clip, trim = 1cm 0cm 1cm 0cm, width = .5\textwidth, keepaspectratio = true]{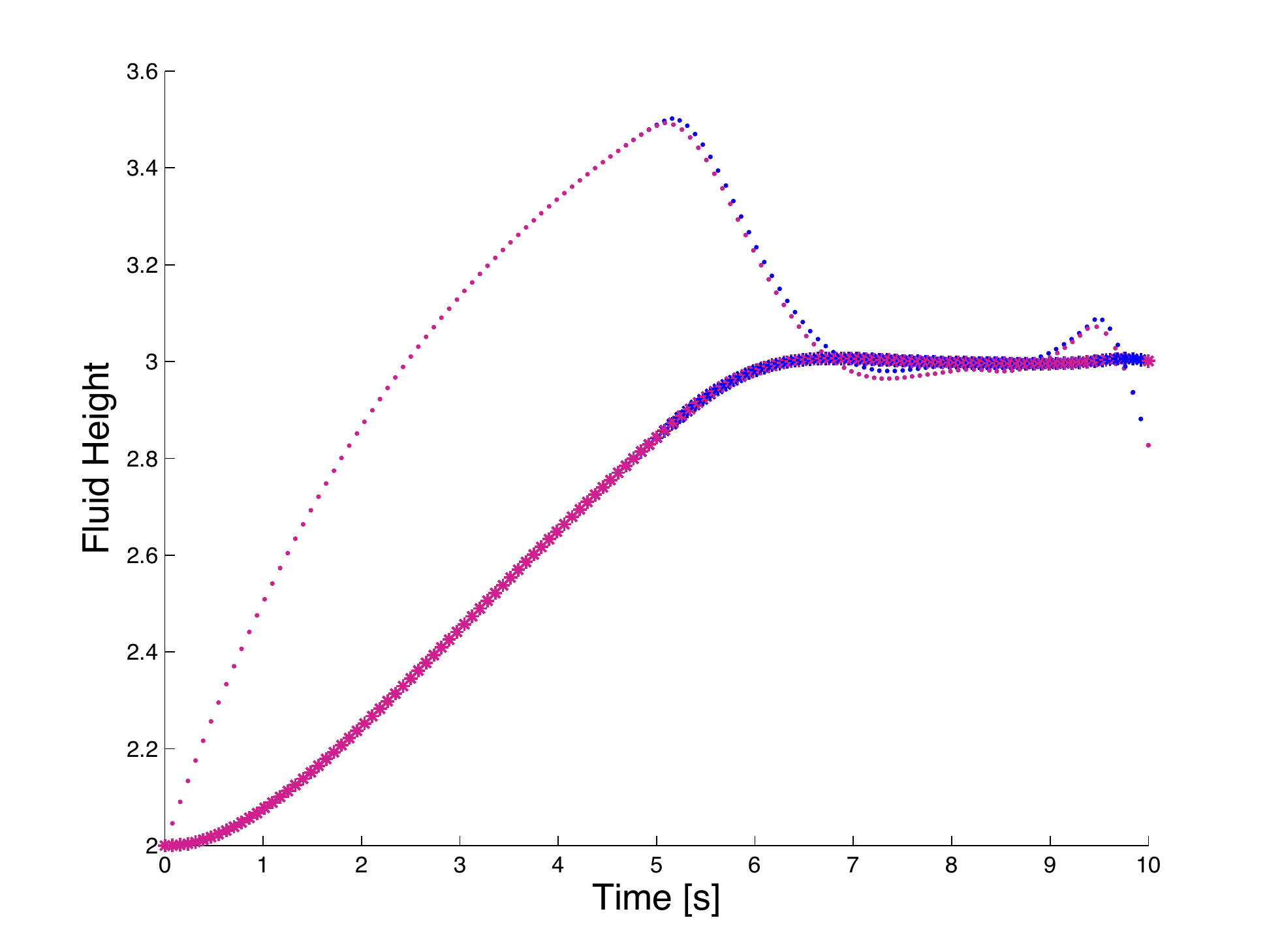}}
	\end{minipage}
	\caption{Optimal trajectories for each of the considered optimization algorithms where $x_1(t)$ is drawn using points and $x_2(t)$ is drawn using stars and where each state trajectory is drawn in blue when in mode $1$ and in purple when in mode $2$.}
	\label{fig:tankexample}	
\end{figure}

To illustrate the performance of Algorithm \ref{algo:discrete_algo} when there is no continuous input present, we consider a double-tank example. The two states of the system correspond to the fluid levels of an upper and lower tank. The output of the upper tank flows into the lower tank, the output of the lower tank exits the system, and the flow into the upper tank is restricted to either $1$ or $2$. The dynamics in each mode are then derived using Toricelli's Law and are describe in Table \ref{tab:dynamics}. The objective of the optimization is to have the fluid level in the lower tank track $3$ and this is reflected in the chosen cost function described in Table \ref{tab:alg_parameters}.

Algorithm \ref{algo:discrete_algo} and the MIP are initialized at $x_0 = (0,0)$ with a discrete input described in Table \ref{tab:alg_implement} with $128$ equally spaced samples in time. Algorithm \ref{algo:discrete_algo} took $67$ iterations, ended with $256$ time samples, and terminated after the optimality condition was bigger than $-10^{-2}$. The result of both optimization procedures is illustrated in Figure \ref{fig:tankexample}. The computation time and final cost of both algorithms can be found in Table \ref{tab:alg_implement}. Notice that Algorithm \ref{algo:discrete_algo} is able to compute a comparable cost discrete input compared to the MIP and is able to do it nearly $3700$ times faster.


\subsection{Quadrotor Helicopter Control}

\begin{figure}[tp]
	\centering
	\begin{minipage}{\columnwidth}
		\subfloat[MIP Final Result]{\includegraphics[clip, trim = 1.75cm 3cm 1.75cm 3cm, width = .49\textwidth, keepaspectratio = true]{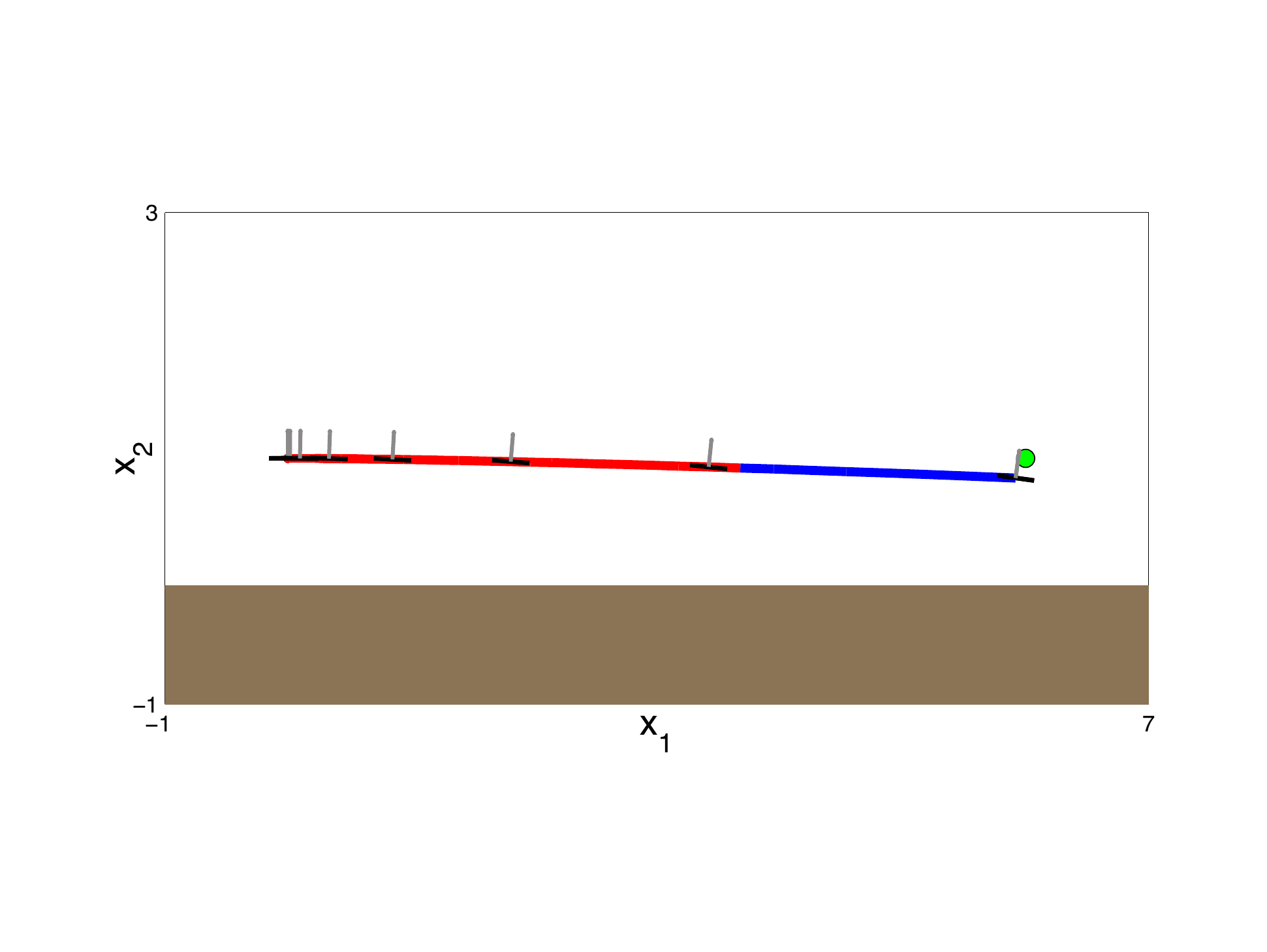}}
		\hspace*{0.05cm}
		\subfloat[Algorithm \ref{algo:discrete_algo} Final Result]{\includegraphics[clip, trim = 1.75cm 3cm 1.75cm 3cm, width = .49\textwidth, keepaspectratio = true]{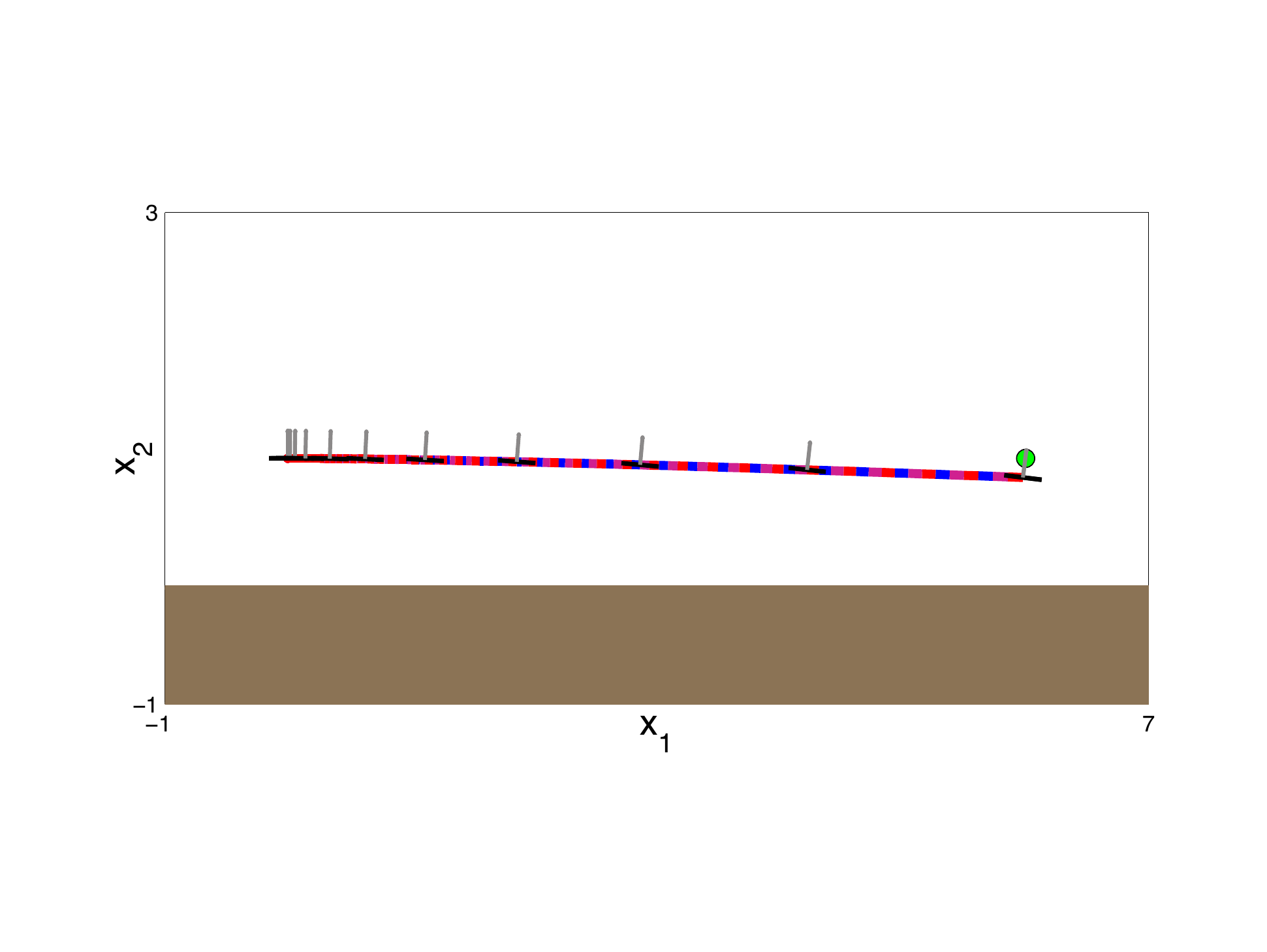}}
	\end{minipage}
	\caption{Optimal trajectories for each of the considered optimization algorithms where the point $(6,1)$ is drawn in green, where the trajectory is drawn in blue when in mode $1$, in purple when in mode $2$, and in red when in mode $3$, and where the quadrotor is drawn in black and the normal direction to the frame is drawn in gray.}
	\label{fig:quadrotorexample}	
\end{figure}

Next, we consider the optimal control of a quadrotor helicopter in 2D using a model described in \cite{gillula2011applications}. The evolution of the quadrotor can be defined with respect to a fixed 2D reference frame using six dimensions where the first three dimensions represent the position along a horizontal axis, the position along the vertical axis and the roll angle of the helicopter, respectively, and the last three dimensions represent the time derivative of the first three dimensions. We model the dynamics as a three mode switched system (the first mode describes the dynamics of going up, the second mode describes the dynamics of moving to the left, and the third mode describes the dynamics of moving to the right) with a single input as described in Table \ref{tab:dynamics} where $L = 0.3050$ meters, $M = 1.3000$ kilograms, $I = 0.0605$ kilogram meters squared, and $g = 9.8000$ meters per second squared. The objective of the optimization is to have the trajectory of the system at time $t_f$ be at position $(6,1)$ with a zero roll angle while minimizing the input required to achieve this task. This objective is reflected in the chosen cost function which is described in Table \ref{tab:alg_parameters}. A state constraint is added to the optimization to ensure that the quadrotor remains above ground.

Algorithm \ref{algo:discrete_algo} and the MIP are initialized at position $(0,1)$ with a zero roll angle, with zero velocity, with continuous and discrete inputs as described in Table \ref{tab:alg_implement}, and with $64$ equally spaced samples in time. Algorithm \ref{algo:discrete_algo} took $31$ iterations, ended with $192$ time samples, and terminated after the optimality condition was bigger than $-10^{-4}$. The result of both optimization procedures is illustrated in Figure \ref{fig:quadrotorexample}. The computation time and final cost of both algorithms can be found in Table \ref{tab:alg_implement}. Notice that Algorithm \ref{algo:discrete_algo} is able to compute a lower cost continuous and discrete input when compared to the MIP and is able to do it more than $333$ times faster.

\subsection{Bevel-Tip Flexible Needle}

\begin{figure}[tp]
	\centering
	\begin{minipage}{\columnwidth}
		\subfloat[Algorithm \ref{algo:discrete_algo} Final Result]{\includegraphics[clip, trim = 5cm 1cm 5cm 1cm, width = .38\textwidth, keepaspectratio = true]{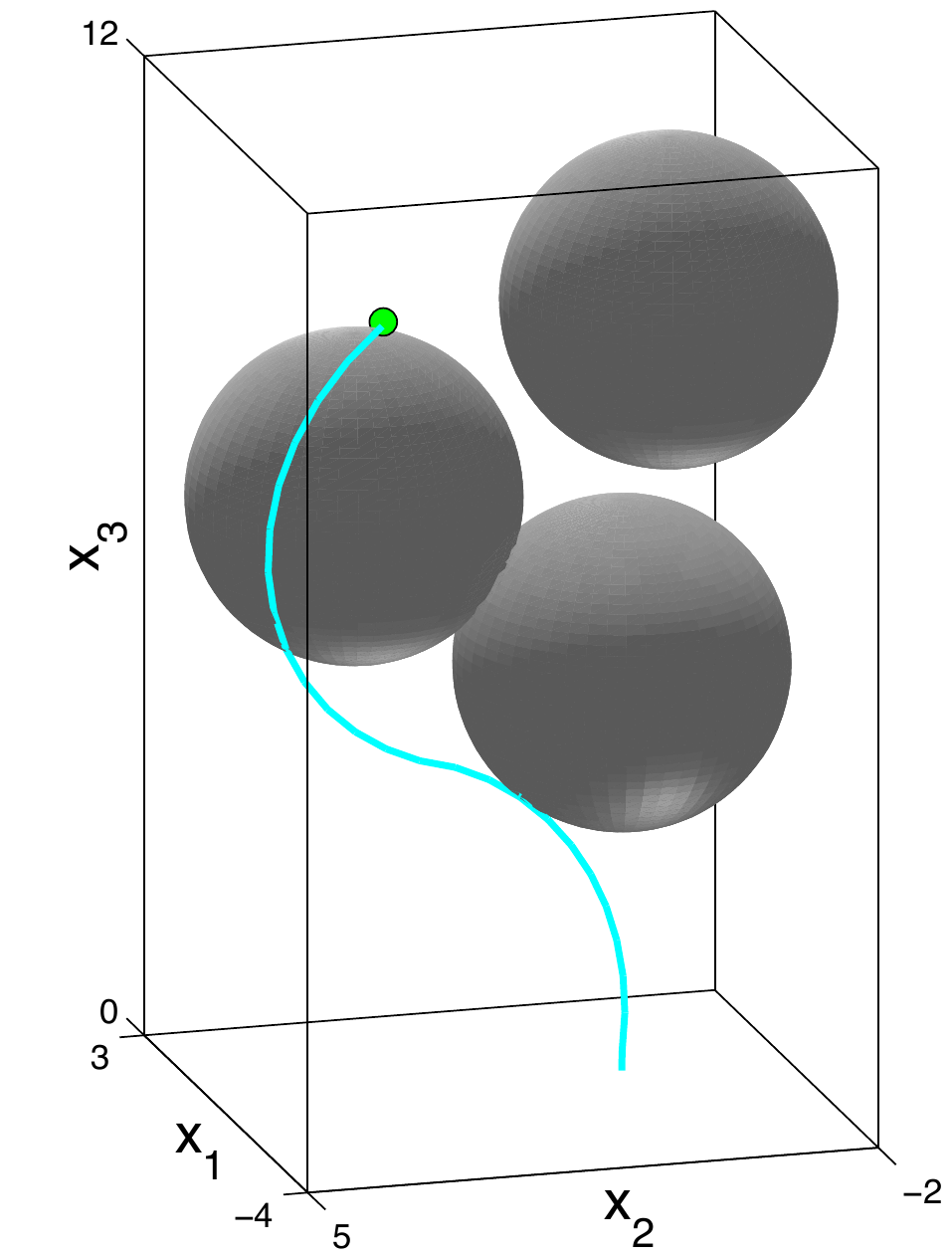}}
		\subfloat[Algorithm \ref{algo:discrete_algo} Final Discrete Input]{\includegraphics[clip, trim = 1cm 0cm 1cm 0cm, width = .62\textwidth, keepaspectratio = true]{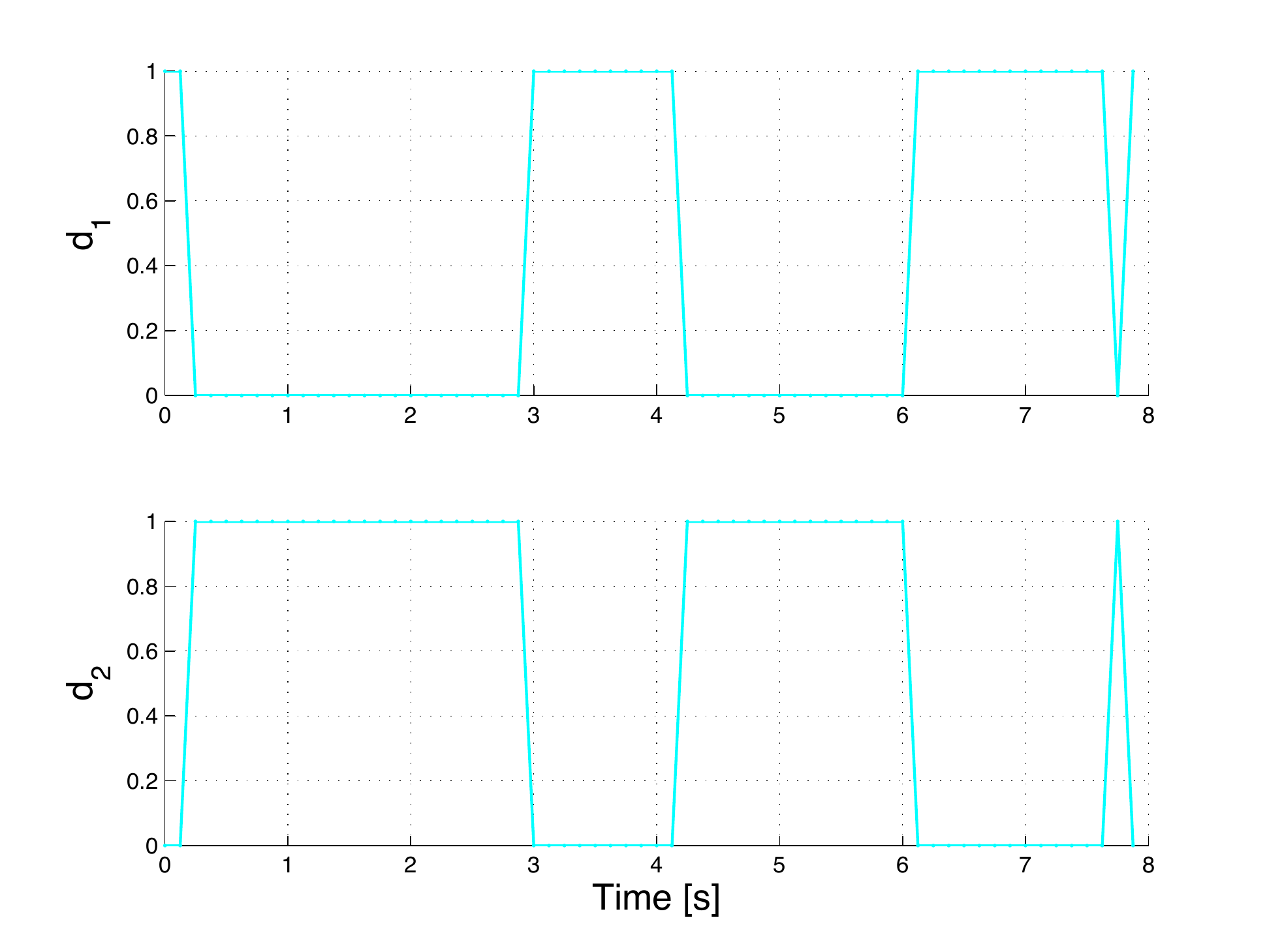}}
	\end{minipage}
	\caption{The optimal trajectory and discrete inputs drawn in cyan generated by Algorithm \ref{algo:discrete_algo} where the point $(-2,3.5,10)$ is drawn in green and obstacles are drawn in grey.}
	\label{fig:needleexample}	
\end{figure}

Bevel-tip flexible needles are asymmetric needles that move along curved trajectories when a forward pushing force is applied. The 3D dynamics of such needles has been described in \cite{Kallem2007} and the path planning in the presence of obstacles has been heuristically considered in \cite{Duindam2008}. The evolution of the needle can be defined using six dimensions where the first three dimensions represent the position of the needle relative to the point of entry and the last three dimensions represent the yaw, pitch and roll of the needle relative to the plane, respectively. As suggested by \cite{Duindam2008}, the dynamics of the needle are naturally modeled as a two mode (the first mode describes the dynamics of going forward while the second mode describes the dynamics of the needle turning) switched system as described in Table \ref{tab:dynamics} with two continuous inputs: $u_1$ representing the insertion speed and $u_2$ representing the rotation speed of the needle and where $\kappa$ is the curvature of the needle and is equal to $.22$ inverse centimeters. The objective of the optimization is to have the trajectory of the system at time $t_f$ be at position $(-2,3.5,10)$ while minimizing the input required to achieve this task. This objective is reflected in the chosen cost function which is described in Table \ref{tab:alg_parameters}. A state constraint is added to the optimization to ensure that the needle remains outside of three spherical obstacles centered at $(0\,,\,0\,,\,5)$, $(1\,,\,3\,,\,7)$, and $(-2\,,\,0\,,\,10)$ all with radius $2$.

Algorithm \ref{algo:discrete_algo} and the MIP are initialized at position $(0,0,0)$ with continuous and discrete input described in Table \ref{tab:alg_implement} with $64$ equally spaced samples in time. Algorithm \ref{algo:discrete_algo} took $103$ iterations, ended with $64$ time samples, and terminated after the optimality condition was bigger than $-10^{-3}$. The computation time and final cost of both algorithms can be found in Table \ref{tab:alg_implement}. The MIP was unable to find any solution. The result of Algorithm \ref{algo:discrete_algo} is illustrated in Figure \ref{fig:needleexample}.


\section{Conclusion}
\label{sec:conclusion}

In this paper, we devise a first order numerical optimal control algorithm for the optimal control of constrained nonlinear switched systems. The algorithm works by first relaxing the discrete-valued input, performing traditional optimal control, and projecting the computed relaxed discrete-valued input by employing a projection constructed by an extension to the classical Chattering Lemma. We prove that the sequence of points constructed by recursive application of our algorithm converge to a point that satisfies a necessary condition for optimality of the Switched System Optimal Control Problem. We then devise an implementable algorithm that operates over finite dimensional subspaces of the optimization spaces. We prove the convergence of the sequence of points constructed by recursive application of our computationally tractable algorithm to a point that satisfies a necessary condition for optimality of the Switched System Optimal Control Problem.

\section*{Acknowledgements}
We are grateful to Maryam Kamgarpour and Claire Tomlin whose help was critical in writing our first set of papers on switched system optimal control and realizing the deficiencies of our original approach. 
We are also grateful to Magnus Egerstedt and Yorai Wardi whose comments on our original papers were critical to us realizing the deficiencies of the sufficient descent argument in infinite dimensional spaces.
The Chattering Lemma argument owes its origin to a series of tremendously helpful discussions with Ray DeCarlo.
The entire contents of this article reflect a longstanding collaboration with Elijah Polak.
The preparation of this article required a great deal of research effort which would not have been possible without the support of the National Science Foundation.

\phantomsection
\addcontentsline{toc}{chapter}{References}
\bibliographystyle{plain}
\bibliography{library}

\end{document}